\documentclass[11pt,letter]{amsart}
\usepackage{amscd}
\usepackage{amssymb}
\usepackage[centering,text={15.5cm,24cm}]{geometry}
\usepackage[dvips]{graphicx}
\usepackage{color}
\usepackage[all]{xy}
\usepackage{mathrsfs}
\usepackage{marvosym}
\usepackage{stmaryrd}
\usepackage{srcltx}

\usepackage{pstricks}

\usepackage{hyperref}
\hypersetup{
  colorlinks   = true,          
  urlcolor     = violet,          
  linkcolor    = blue,          
  citecolor   = violet             
}

\definecolor{shadecolor}{rgb}{1,0.9,0.7}

\newtheorem*{namedtheorem}{\theoremname}
\newcommand{\theoremname}{testing}
\newenvironment{named}[1]{\renewcommand\theoremname{#1}
\begin{namedtheorem}}
{\end{namedtheorem}}

\newtheorem{theorem}{Theorem}[section]
\newtheorem{lemma}[theorem]{Lemma}
\newtheorem{proposition}[theorem]{Proposition}

\theoremstyle{definition}
\newtheorem{definition}[theorem]{Definition}
\newtheorem{construction}[theorem]{Construction}

\newtheorem{assumption}[theorem]{Assumption}
\newtheorem{assumptions}[theorem]{Assumptions}

\newtheorem{example}[theorem]{Example}

\theoremstyle{remark}
\newtheorem{remark}[theorem]{Remark}

\newtheorem{notation}[theorem]{Notation}

\numberwithin{equation}{section}
\numberwithin{figure}{section}

\newcommand{\NN} {\mathbb{N}}
\newcommand{\ZZ} {\mathbb{Z}}
\newcommand{\QQ} {\mathbb{Q}}
\newcommand{\RR} {\mathbb{R}}
\newcommand{\CC} {\mathbb{C}}

\newcommand{\PP} {\mathbb{P}}
\renewcommand{\AA} {\mathbb{A}}
\newcommand{\GG} {\mathbb{G}}

\newcommand\cM{\mathcal{M}}

\newcommand\cO{\mathcal{O}}

\newcommand\cX{\mathcal{X}}

\newcommand\fM{\mathfrak{M}}

\newcommand {\shA}  {\mathcal{A}}

\newcommand {\shI}  {\mathcal{I}}
\newcommand {\shJ}  {\mathcal{J}}
\newcommand {\shK}  {\mathcal{K}}

\newcommand {\shM}  {\mathcal{M}}

\newcommand {\shS}  {\mathcal{S}}

\newcommand {\shU}  {\mathcal{U}}

\newcommand {\shP}  {\mathcal{P}}

\newcommand {\shX}  {\mathcal{X}}

\newcommand {\foC}  {\mathfrak{C}}

\newcommand {\foM}  {\mathfrak{M}}

\newcommand {\foX}  {\mathfrak{X}}

\newcommand {\fob}  {\mathfrak{b}}

\newcommand {\fod}  {\mathfrak{d}}

\newcommand {\foj}  {\mathfrak{j}}

\newcommand {\fom}  {\mathfrak{m}}

\newcommand {\fop}  {\mathfrak{p}}

\newcommand {\fot}  {\mathfrak{t}}
\newcommand {\fou}  {\mathfrak{u}}

\newcommand {\Aut}  {\operatorname{Aut}}

\newcommand\bu{\mathbf{u}}
\newcommand\bA{\mathbf{A}}
\newcommand\bsigma{{\boldsymbol{\sigma}}}

\newcommand\bomega{{\boldsymbol{\omega}}}
\newcommand\btau{{\boldsymbol{\tau}}}

\newcommand {\can} {{\mathrm{can}}}

\newcommand {\codim} {\operatorname{codim}}
\newcommand {\colim} {\operatorname{colim}}

\newcommand {\coker} {\operatorname{coker}}

\newcommand {\Div}  {\operatorname{Div}}

\newcommand {\ev}  {{\operatorname{ev}}}

\newcommand {\fs} {\mathrm{fs}}

\newcommand {\GL}  {\operatorname{GL}}

\newcommand {\gp}  {{\operatorname{gp}}}

\newcommand {\Hom}  {\operatorname{Hom}}

\newcommand {\id}  {\operatorname{id}}
\newcommand {\im}  {\operatorname{im}}
\newcommand {\inc}  {\mathrm{in}}

\newcommand {\inte}  {\mathrm{int}}
\newcommand {\Int}  {\operatorname{Int}}

\renewcommand {\ker } {\operatorname{ker}}
\newcommand {\kk} {\Bbbk}

\newcommand {\lra}  {\longrightarrow}

\newcommand {\M} {\mathcal{M}}

\renewcommand {\max} {{\operatorname{max}}}

\renewcommand{\O}  {\mathcal{O}}

\newcommand {\out}  {\mathrm{out}}

\renewcommand{\P}  {\mathscr{P}}

\newcommand {\PL} {\operatorname{PL}}

\newcommand {\pr}  {\operatorname{pr}}

\newcommand {\Proj} {\operatorname{Proj}}

\renewcommand {\red}  {{\operatorname{red}}}

\newcommand {\rk} {\operatorname{rk}}

\newcommand {\scrD}  {\mathscr{D}}

\newcommand {\scrB}  {\mathscr{B}}
\newcommand {\scrM}  {\mathscr{M}}

\newcommand {\scrP}  {\mathscr{P}}
\newcommand {\scrS}  {\mathscr{S}}
\newcommand {\sat}  {{\operatorname{sat}}}

\newcommand {\Sing} {\operatorname{Sing}}

\newcommand {\Spec} {\operatorname{Spec}}

\newcommand {\Star} {\mathrm{Star}}

\newcommand {\trop}  {\mathrm{trop}}
\newcommand {\tors}  {\mathrm{tors}}

\newcommand {\ul} {\underline}

\newcommand {\virt} {\mathrm{virt}}

\newcommand\arr{\ifinner\to\else\longrightarrow\fi}

\def\mydate{\ifcase\month \or January\or February\or March\or
April\or May\or June\or July\or August\or September\or October\or 
November\or December\fi \space\number\day,\space\number\year}

\newcommand{\Cones}{\mathbf{Cones}}

\newcommand\ocM{\overline{\mathcal{M}}}

\begin{document}

\title{The canonical wall structure and intrinsic mirror symmetry}

\author{Mark Gross} \address{DPMMS, Centre for Mathematical Sciences,
Wilberforce Road, Cambridge, CB3 0WB, UK}
\email{mgross@dpmms.cam.ac.uk}

\author{Bernd Siebert} 
\address{\tiny Department of Mathematics, The Univ.\ of Texas at Austin,
2515 Speedway, Austin, TX 78712, USA}
\email{siebert@math.utexas.edu}

\begin{abstract}
As announced in \cite{Utah} in 2016, we construct and prove consistency of the
\emph{canonical wall structure}. This construction starts with a log Calabi-Yau
pair $(X,D)$ and produces a wall structure, as defined in \cite{Theta}. Roughly
put, the canonical wall structure is a data structure which encodes an
algebro-geometric analogue of counts of Maslov index zero disks. These
enumerative invariants are defined in terms of the punctured invariants of
\cite{ACGSII}. There are then two main theorems of the paper. First, we prove
consistency of the canonical wall structure, so that, using the setup of
\cite{Theta}, the canonical wall structure gives rise to a mirror family.
Second, we prove that this mirror family coincides with the intrinsic mirror
constructed in \cite{Assoc}. While the setup of this paper is narrower than that
of \cite{Assoc}, it gives a more detailed description of the mirror.
\end{abstract}
\thanks{M.G. was supported by the ERC grant MSAG,  EPSRC grant EP/N03189X/1
and a Royal Society Wolfson Research Merit Award.
Research by B.S. was partially supported by NSF grant DMS-1903437.}

\date{\today}
\maketitle

\setcounter{tocdepth}{2}
\tableofcontents



\addtocontents{toc}{\protect\setcounter{tocdepth}{1}}
\section*{Introduction}

The mathematical investigation of the mirror phenomenon observed in physics
\cite{GrPl,CLS,COGP} has long been vexed by the basic question how broadly
mirror pairs exist. Mirror symmetry is certainly too wide a phenomenon to even
imagine an answer that is immediately satisfying from all perspectives. The
authors have nevertheless long been convinced that there is a fundamental
mechanism underlying all mirror phenomena that is both key to general
constructions of mirror pairs and to establish various mirror correspondences
between them.

\subsubsection*{Intrinsic mirror symmetry}
Following a long line of developments and refinements around the SYZ philosophy
of dual torus fibrations \cite{SYZ, GrossTMS, KS1, Fukaya, PartI, KS2, Annals,
Auroux, GHK, Tu, Abouzaid1, Abouzaid2}, the authors in \cite{Assoc} proposed an
intrinsic mirror construction for logarithmic Calabi-Yau varieties.

The construction takes as input a pair $(X,D)$ of a normal crossings divisor $D$
on a smooth variety $X$, projective over a point or a curve, such that
$X\setminus D$ supports a holomorphic volume form with at most logarithmic poles
along $D$. We assume further that $D$ has a zero-dimensional stratum. The mirror
is constructed as the affine or projective spectrum of a newly defined degree
zero relative quantum cohomology ring for the pair $(X,D)$ that is expected to
only depend on $X\setminus D$. The definition of the relative quantum cohomology
ring has been made possible by the development of punctured Gromov-Witten
invariants, a variant of logarithmic Gromov-Witten invariants \cite{JAMS, Chen,
AC} admitting negative contact orders, in joint work of the authors with
Abramovich and Chen \cite{ACGSII}. The structure coefficients of relative
quantum cohomology involve invariants with two positive and one negative or zero
contact orders.

In our opinion this is a very clean and satisfying construction from the
algebraic geometric point of view: It generalizes many of the known geometric
mirror constructions to a large, \emph{birationally distinguished} class of
algebraic varieties.\footnote{It is a certain joke of history that a slight
twist to one of the first objects arising in the context of mirror symmetry,
namely quantum cohomology, turns out to hold an answer to the puzzling
fundamental question in mirror symmetry in such great generality. } For example,
maximally unipotent degenerations of smooth proper Calabi-Yau varieties, the
original object of study in mirror symmetry, and Fano manifolds with an
anticanonical divisor with zero-dimensional stratum both fall in this class of
varieties.

There is also a clear tentative interpretation of relative quantum cohomology in
terms of symplectic geometry as a degree zero symplectic cohomology ring
\cite{Pascaleff, Pomerleano, GanatraPomerleano, Pomerleano2}. This link both
points to possible generalizations of the construction to a purely symplectic
framework and to higher degree relative quantum cohomology rings \cite{GPoS}, as
well as to a proof of homological mirror symmetry \cite{Ko} in this framework
\cite{PS}.

\subsubsection*{The canonical wall structure and the SYZ interpretation of
intrinsic mirror symmetry}
The present paper puts the intrinsic mirror construction of
\cite{Assoc} into the context of our long-term program aiming at an
algebraic-geometric implementation of the SYZ picture of mirror symmetry. Our
program is based on a tropicalization of the geometry, leading to a polyhedral
manifold $B$ with an integral affine structure with singularities (a
\emph{polyhedral affine manifold}) together with a wall structure. The
polyhedral affine manifold provides the central fiber of the mirror family, the
uncorrected mirror, while the wall structure carries the quantum corrections to
build a consistent deformation of this uncorrected mirror. In \cite{Annals} we
gave an algorithmic inductive construction of the wall structure for cases with
affine singularities that are locally rigid in some sense, purely in terms of
affine geometry. We have long suspected, but could make precise only in very
limited cases \cite{GPS,GHK}, that for the construction of the mirror of a pair
$(X,D)$ the wall structure should be constructible from the tropicalizations of
rational curves on $X$ virtually intersecting $D$ at one point; the enumerative
invariants carried by the wall structure should then be the corresponding
punctured Gromov-Witten invariants, see e.g.\ \cite[\S4]{ICM}. This picture is
the algebraic-geometric analogue of the quantum corrections on the symplectic
side from Maslov index zero pseudoholomorphic disks with boundaries on
Lagrangian fibers that give rise to the bounding cochains in \cite{FOOO}.

With punctured Gromov-Witten theory sufficiently developed \cite{ACGSII}, we
have now been able to make this picture a reality. Moreover, using the general
framework of constructing families from consistent wall structures worked out in
\cite{Theta}, we give an alternative and technically simpler proof of the
associativity of the degree zero relative quantum cohomology ring from
\cite{Assoc} in the case of interest for mirror symmetry, where $(X,D)$ has
logarithmic Kodaira dimension zero and $D$ has a zero-dimensional stratum. We
emphasize that only the present paper makes the clear connection of the
intrinsic mirror construction to the SYZ picture. The assumption that
$D$ has a zero-dimensional stratum indeed provides families of Lagrangian tori
in $X\setminus D$ degenerating to the zero dimensional stratum. Thus this
assumption can be both viewed as a replacement of the existence of Lagrangian
torus fibrations or, in a degeneration situation, of the maximally unipotent
monodromy assumption \cite{Morrison2}.

The present paper also provides the link to the algorithmic construction of
walls \cite{Annals}, and to previous mirror constructions in two-dimensions
\cite{GHK} and for cluster varieties \cite{GHKK}, also giving rise to rich
combinatorial structures, see e.g.~\cite{Cheung, ChMN, CheungMandel, Magee, Mandel1,
RietschWilliams, Prince}. All these previous constructions generalize known and
tested mirror constructions such as \cite{Batyrev,BB, HoriVafa, GrossBB}, and
they have also been independently tested, see e.g.\ \cite{PartII, GrossP2, Lau,
RS, Mandel2, BBMN, HK}, thus providing further evidence that \cite{Assoc} really does
produce mirror pairs. Wall structures also give powerful methods for explicit
computations \cite{Annals, GPS, HDTV}, and in fact, all examples of mirror pairs
we have computed explicitly were first obtained via their wall structures.

Another motivation is that the wall structures contain a lot of information not
directly accessible from the relative quantum cohomology ring in \cite{Assoc}.
For example, our wall structures suggest to generalize Mikhalkin-style tropical
correspondence theorems for curve counting from toric ambient geometries to
logarithmic Calabi-Yau varieties. In this correspondence, straight lines are to
be replaced by the \emph{broken lines} reviewed in \S\ref{subsec:broken lines}.
Broken lines are certain piecewise straight lines with possible bends when
crossing a wall due to interaction with the tropical disks carried by the wall.
The present paper also gives a systematic treatment of the geometric
information carried by the wall structure, which sometimes contains interesting
enumerative information in its own right \cite{GPS, Bousseau4, HDTV, Tim}.

\subsubsection*{Statement of main results}
We now describe the results of the paper in more detail. In \S\ref{sec:basic} we
define the relevant polyhedral affine manifold $B$ as the dual intersection
complex of $D$, study the relevant affine geometry and discuss the additional
conditions compared to \cite{Assoc} to fulfill the assumptions on $B$ in
\cite{Theta}. As in \cite{Assoc} we distinguish the \emph{absolute} and
\emph{relative case}. In both cases $X$ is a smooth variety over an
algebraically closed field $\kk$ of characteristic zero and $D\subset X$ is a
normal crossing divisor. In the absolute case $X$ is projective over $\kk$ while
in the relative case we have a projective morphism $g: X\to S$ with $S$ an
affine curve or spectrum of a DVR, $g$ smooth away from a closed point $0\in S$ and $g^{-1}(0)\subseteq
D$. In other words, $g$ induces a log smooth morphism $(X,\M_X)\to (S,\M_S)$
when endowing $X,S$ with their respective divisorial log structures. Let
$\Sigma(X)$ denote the tropicalization of $(X,D)$ introduced in
\S\ref{subsec:basic trop}. 

\begin{named}{Assumption T}
The pair $(X,D)$ fulfills Assumptions~\ref{ass:absolute} and
Assumptions~\ref{ass:relative} related to its tropicalization $\Sigma(X)$ in the
absolute and relative cases, respectively.
\end{named}

We show in Proposition~\ref{prop:dlt CY models work} that Assumption~T holds for
resolutions of log Calabi-Yau minimal models with connected $D$ having a
zero-dimensional stratum, using a result of Koll\'ar and Xu \cite{KX}. While the
existence of log minimal models is not yet generally known in dimensions greater
than three, a recent result of Birkar paraphrased in Theorem~\ref{Thm: Birkar}
shows their existence in the relevant situation under a technical assumption
\cite[Cor.~1.5]{Birkar}. The upshot is that we expect Assumption~T to be
fulfilled for all practical purposes for cases of interest in mirror symmetry.

\S\ref{sec:punctured log maps} reviews some material concerning
punctured maps, their moduli theory and their tropicalizations. A
key result for this paper is Lemma~\ref{lem:key tropical lemma} classifying
those tropical punctured maps that later appear in walls and broken lines.

After these preparations, \S\ref{sec:wall structures and log broken lines}
introduces the two main players of this paper, the canonical wall structure
$\scrS_\can$ on $B$ and logarithmic broken lines, both defined in terms of
certain punctured Gromov-Witten invariants. Walls are defined in terms of
punctured Gromov-Witten invariants with one non-zero contact order. There are in
fact two canonical wall structures $\scrS_\can^{\mathrm{undec}}$,
$\scrS_\can$ (Constructions~\ref{const:wall construction} and
\ref{const:decorated wall structure}), where $\scrS_\can^{\mathrm{undec}}$ only
fixes the total curve class while $\scrS_\can$ prescribes a curve class for each
vertex of the tropical disk (``decorated wall type''). The refinement concerning
curve classes is necessary for some proofs. The formulas defining the walls in
terms of punctured Gromov-Witten invariants are stated in \eqref{lem:wall
virtual dim}--\eqref{eq:wall def} and in \eqref{eq:refined wall def} for the
undecorated and the decorated cases, respectively. A similar definition, but
with punctured invariants with two rather than one contact orders, one of which is negative, provides the notion of \emph{logarithmic broken line}.

At this point we have in principle the same objects as in \cite{Theta}, but we
neither know that the wall structure is consistent nor that our logarithmic
broken lines have anything in common with the broken lines from \cite{Theta},
which are defined algebraic-combinatorially from the canonical wall structure
$\scrS_\can$. Our first main result, covered in \S\ref{sec:correspondence}, is
that these two notions of broken lines indeed agree.

\begin{named}{Theorem~A}
Let $(X,D)$ fulfill Assumption~T. Then the broken lines for the canonical wall
structures $\scrS_\can$ and $\scrS_\can^{\mathrm{undec}}$ are exactly the logarithmic broken lines.
\end{named}

Theorem~A follows from the bijection between logarithmic broken line types and
families of broken lines of fixed combinatorial type (Proposition~\ref{prop:one
to one}), with equality of the corresponding coefficients (Theorem~\ref{thm:main
correspondence theorem}).

The proof of the decisive Theorem~\ref{thm:main correspondence theorem} relies
on a gluing formula for punctured Gromov-Witten invariants that gives the
crucial interpretation of bending at a wall as attaching a number $l$ of genus
zero punctured maps with one puncture (``bubbles'') to a genus zero punctured
map with two punctures, by adding an irreducible component with trivial
numerical information and $l+2$ punctures and identifying $l+1$ pairs of
punctures to nodes. This gluing problem is simplified a lot since broken lines
only interact with walls in a cell of codimension zero or one. The codimension
zero case corresponds to gluing in a zero-dimensional stratum of $D$, with the
added component to the stable map necessarily contracted. In codimension one the
gluing happens along a one-dimensional stratum. Luckily, an argument by Koll\'ar
shows that Assumption~T implies that a one-dimensional stratum is actually
isomorphic to $\PP^1$ with stratified structure given by two points
(Proposition~\ref{prop:polyhedral pseudo}). Both cases can then be treated in a
rather straightforward manner by the numerical gluing formula for punctured
invariants proved by Yixian Wu \cite{Wu}, which applies when all gluing strata
are toric. For the reader's convenience and to fix notations we recall this
formula in Appendix~\ref{App: gluing formula}.

The essential property of wall structures needed for constructing deformations
is a certain notion of ``consistency''. Consistency says that the schemes
obtained by gluing local standard models for the deformation in codimension zero
and one in a way prescribed by the walls containing any given codimension two
subset $\foj\subset B$ has enough global regular functions. We prove this
property by restricting to the wall structure on the star of $\foj$ given by the
walls containing $\foj$, and then showing that the theta functions obtained from
sums over logarithmic broken lines with fixed endpoint provide such functions.
Thus consistency comes from invariance properties of certain sums of punctured
invariants. This reverses the logic in \cite{Theta} where we assume consistency
to construct theta functions, due to a lack of an a priori interpretation of
broken lines. We obtain our second main result, Theorem~\ref{thm:consistency} in
the body of the text.

\begin{named}{Theorem B}
Let $(X,D)$ fulfill Assumption~T. Then the canonical wall structures
$\scrS_\can$ and $\scrS_\can^{\mathrm{undec}}$ are consistent in the sense of
\cite{Theta}.
\end{named}

\cite{Theta} now provides a ring of theta functions that serves as the affine or
projective coordinate ring of our mirror family, depending if we are in the
absolute or relative case. The last section, \S\ref{sec:intrinsic mirror
symmetry}, is devoted to proving that this ring agrees with the relative quantum
cohomology of \cite{Assoc}.

\begin{named}{Theorem C}
Let $(X,D)$ fulfill Assumption~T. The ring of theta functions associated to the
canonical wall structure $\scrS_\can$ via \cite{Theta} agrees with the degree zero
relative quantum cohomology ring from \cite{Assoc}. In particular, the canonical
wall structure in connection with \cite{Theta} produces the intrinsic mirror
family from \cite{Assoc}.
\end{named}

This theorem is stated as an equality of structure coefficients in
Theorem~\ref{thm:assoc comparison}. The key step in the proof is to split the
structure coefficients from \cite{Assoc} according to types of tropical
punctured maps with two unbounded and one bounded leg. A dimension count
shows that such tropical punctured maps split into two connected components when
removing the vertex containing the bounded leg, with each connected component
leading to a tropical punctured map as they appear in broken lines. The relation
to the structure coefficients of the rings defined by the theta functions then
boils down to another application of Yixian Wu's gluing formula \cite{Wu}, in a
particularly simple situation with a zero-dimensional gluing stratum and
transversality of the tropical gluing situation, see Lemma~\ref{lem:product
gluing}.

\subsubsection*{Related work}
Under the assumption that $X\setminus D$ is affine and contains a
full-dimensional algebraic torus Sean Keel and Tony Yu gave an alternative
construction of the degree zero relative quantum cohomology ring with Berkovich
non-archimedean methods \cite{KeelYu}. The presence of the algebraic torus makes
it possible to avoid negative contact orders by tropically extending negative
contact order legs out to infinity, so that contact orders become positive.
\cite{KeelYu} also construct wall structures with the same assumptions, via a
different, but presumably equivalent, approach to ours, by counting the effect
on analytic cylinders interacting with the wall.

The paper \cite{HDTV} of the first author with H\"ulya Arg\"uz has been written
in parallel to the present work. Generalizing \cite{GPS} to higher dimensions,
this work treats the case that $(X,D)$ is a blowing up of a toric variety in
hypersurfaces contained in its toric boundary with $D$ the strict transform of
that toric boundary. The main result says that in this situation our canonical
wall structure agrees with an algorithmically constructed consistent wall
structure, following \cite{Annals}. Thus \cite{HDTV} provides a rich source of
explicit examples. Because of this, we refer the interested reader to that paper
for examples more complicated than the ones given in this paper.

Honglu Fan, Longting Wu and Fenglong You, in the case of a smooth boundary
divisor, and Hsian-Hua Tseng and Fenglong You in the case of a normal crossings
divisor, made an alternative proposal for a
relative quantum cohomology ring based on orbifold Gromov-Witten invariants
\cite{FWY1, FWY2}. As pointed out by Dhruv Ranganathan, their invariants do not
have the correct invariance properties under log \'etale modifications
to immediately agree with our relative quantum cohomology ring. 
However, since the first version of the current paper, 
\cite{BNR} has achieved a comparison result between orbifold invariants
and log invariants with non-negative contact orders.
It remains to be
seen if a modification of their definitions could also give
negative contact orders. In any
case, there are no wall structures in this picture since walls are genuinely
tropical objects while orbifold Gromov-Witten theory is not known to be linked to
tropical geometry.

In two dimensions, Pierrick Bousseau \cite{Bousseau,Bousseau2,Bousseau3} gave a
higher genus interpretation for quantum versions of wall structures. It remains
an interesting question as to whether there is a higher-dimensional
generalization of this work.

Yoel Groman and Umut Varolgunes kindly informed us of their work in progress
aiming at a construction of mirrors based on symplectic cohomology type
invariants for compact subsets of symplectic manifolds. One of their geometric
frameworks rely on decomposing a symplectic manifold into simpler pieces using a
multiple cut configuration as in \cite[Def.~13]{TMZ}. A multiple cut
configuration gives rise to an SC symplectic degeneration, which is a symplectic
analogue of a semi-stable degeneration. Therefore, this approach is a symplectic
hybrid version of both the patching construction from \cite{Annals} and the
symplectic cohomology interpretation of \cite{Assoc}. It presently does not
involve a wall structure.

\subsubsection*{Acknowledgements}
Our various collaborations have played a key role in the development of the
ideas presented here. Paul Hacking and Sean Keel helped influence our
understanding of wall structures. The collaboration with Dan Abramovich and Qile
Chen developing punctured invariants and a context for developing gluing
formulae lies at the technical heart of this paper. Indeed, it was already in
2011 in conversations with Abramovich and Chen that it became clear that there
would be new types of logarithmic invariants which might be useful for
completing our program.

We further thank Yixian Wu for discussions on the numerical gluing formula, and
H\"ulya Arg\"uz, with the collaboration \cite{HDTV} providing the opportunity to
sharpen some of the statements in the current paper. We thank Caucher Birkar for
providing a partial solution to the existence of log minimal models following a
question by M.G.

We have also benefited from many conversations about wall structures and
scattering diagrams across the years with Man-Wai Cheung, Maxim Kontsevich,
Dhruv Ranganathan, Tony Yue Yu and many others. Finally, we would like to thank
Sam Johnston for pointing out a missed case in the original argument of
\S\ref{sec:intrinsic mirror symmetry}, as well as Evgeny Goncharov,
Jeff Hicks and Ben Morley for pointing out typos.


\addtocontents{toc}{\protect\setcounter{tocdepth}{2}}

\section{The basic setup}
\label{sec:basic}

\subsection{The tropicalization of a log Calabi-Yau pair}
\label{subsec:basic trop}
We use the same setup for the log Calabi-Yau case as in \cite{Assoc}, \S1.
Explicitly, we fix a non-singular variety $\ul{X}$ of dimension $n$ and a simple
normal crossings divisor $D\subseteq \ul{X}$ yielding a divisorial log structure
on $\ul{X}$, log smooth over $\Spec\kk$. We write $X$ as this log scheme, and
$D=D_1+\cdots+D_s$ the decomposition of $D$ into its irreducible 
components. We consider the \emph{absolute case}, in which $X$ is
projective over $S=\Spec\kk$, and the \emph{relative case}, in which one has in
addition a projective log smooth morphism $g:X\rightarrow S$ with $S$ a
regular one-dimensional scheme over $\Spec\kk$ with a divisorial 
log structure coming from a single closed point $0\in S$. We include here
the case that $S$ is the spectrum of a complete DVR, or is an affine curve.
Necessarily $g^{-1}(0)\subseteq D$, but we do not require equality.

To avoid combinatorial complexities and to fit with the hypotheses of
\cite{Theta}, we will assume that for any index set $I\subseteq\{1,\ldots,s\}$,
the (possibly empty) stratum $X_I:=\bigcap_{i\in I} D_i$ of $D$ is
connected. This can always be achieved via a log \'etale birational modification
of $X$. In this case, the \emph{tropicalization} $\Sigma(X)$ (see
\cite[\S2.1.4]{ACGSI}) has a simple description, following \cite[Ex.\
1.4]{Assoc}, as a polyhedral cone complex in $\Div_D(X)^*_{\RR}$. Here
$\Div_D(X)$ is the group of divisors supported on $D$ and $\Div_D(X)^*_{\RR}$ is
the dual vector space. If $D_1^*,\ldots,D_s^*$ is the basis dual to
$D_1,\ldots,D_s$, then
\[
\Sigma(X):=\left\{\sum_{i\in I} \RR_{\ge 0} D_i^*\,|\,I\subseteq\{1,\ldots,s\},
X_I\not=\emptyset\right\}.
\]
For $\rho=\sum_{i\in I} \RR_{\ge 0}D_i^*\in\Sigma(X)$, we will often write $X_{\rho}$
instead of $X_I$ in the sequel, keeping with the convention
of \cite{ACGSI,ACGSII}. We usually view $X_{\rho}$ as a log scheme,
with the strict closed embedding $X_{\rho}\hookrightarrow X$.
We write $|\Sigma(X)|$ for the support
of this polyhedral cone complex.

For a stratum $X_I=X_{\rho}$, we write $\partial X_I$ or
$\partial X_{\rho}$ for the reduced
divisor on $X_I$ given as
\[
\sum_{j\not\in I} D_j\cap X_I.
\]

The pair $(X,D)$ is
\emph{log Calabi-Yau} if the logarithmic canonical class
$K_X+D$ is
numerically equivalent to an effective $\QQ$-divisor supported on $D$.
We then fix once and for all an explicit representation
\begin{equation}
\label{eq:KXD}
K_X+D\equiv_{\QQ} \sum_i a_i D_i 
\end{equation}
with the $a_i\ge 0$. In general, our mirror construction depends
on this choice in that it determines the set of good divisors below
and hence the Kontsevich-Soibelman skeleton. For the situations 
considered in this paper, however, we expect this skeleton to be
well-defined. This is known to be the case if $(X,D)\rightarrow S$ is
a degeneration of Calabi-Yau manifolds with $D$ the fibre over $0\in
S$, see \cite{NX}. More generally, see
\cite[Ex.~1.23]{Assoc} for further discussion
on this point. Recall that $K_X+D$ is the first Chern class
of the sheaf $\Omega^1_X=\Omega^1_{\ul{X}}(\log D)$ 
of logarithmic differential forms.
We call $D_i$ \emph{good} if $a_i=0$.
We call a stratum $X_I$ \emph{good} if $a_i=0$ for all $i\in I$.
We write
\[
\partial^{\mathrm{good}} X_I:=\sum_{j\not\in I, a_j=0} D_j\cap X_I,\quad\quad
\partial^{\mathrm{bad}} X_I:=\sum_{j\not\in I, a_j>0} D_j\cap X_I,
\]
so that $\partial X_I = \partial^{\mathrm{good}} X_I + \partial^{\mathrm{bad}} 
X_I$.

The \emph{Kontsevich-Soibelman skeleton} of $X$ is the pair $(B,\P)$
where 
\[
\P:=\{\sigma\in\Sigma(X)\,|\,\hbox{$X_{\sigma}$ is good}\}
\]
and
\[
B=\bigcup_{\sigma\in\P}\sigma.
\]
Thus $B$ is a topological space with polyhedral cone decomposition $\P$
with all cones standard simplicial cones.

We denote by $\P^{[i]}$ the set of $i$-dimensional cones of $\P$ and
$\P_{\max}$ the set of maximal cones.

In \cite{Assoc}, we were able to construct a ring from the above data
with no further assumptions. However, for the setup of wall crossing
structures in \cite{Theta}, we need to impose some additional 
hypotheses on the pair $(X,D)$ and the map to $S$. We list these
assumptions here, which will be in force for the remainder of the
paper, and discuss in \S\ref{sec:MMP} the naturality of these assumptions.

\begin{assumptions}[The absolute case] 
\label{ass:absolute}
We assume that
\begin{enumerate}
\item $\P$ contains an $n$-dimensional cone, where $n=\dim X$.
\item 
Whenever $\rho\in\P^{[n-1]}$, $\sigma\in\Sigma(X)$ with 
$\dim\sigma=n$ and $\rho\subseteq\sigma$, we also have $\sigma\in\P$.
Put another way, a good one-dimensional stratum $X_{\rho}$ only
intersects good divisors.
\item Whenever $\rho\in\P$ with $\dim X_{\rho}>1$, $\partial^{\mathrm{good}} 
X_{\rho}$ is connected.
\end{enumerate}
\end{assumptions}

\begin{assumptions}[The relative case]
\label{ass:relative}
We assume that, with $g:X\rightarrow S$ in the relative case and $\dim X=n$: 
\begin{enumerate}
\item
Conditions (1)--(3) of Assumptions~\ref{ass:absolute} hold for
the pair $(X,D)$.
\item For any $\rho\in \P$, $g|_{X_{\rho}}$ has 
geometrically connected fibres.
(Note $g|_{X_{\rho}}$ may be constant.)
\end{enumerate}
\end{assumptions}

\begin{proposition}
\label{prop:polyhedral pseudo}
If $(X,D)$ satisfies Assumptions~\ref{ass:absolute} or \ref{ass:relative}
in the absolute or relative cases respectively, then $(B,\P)$ is a 
pseudomanifold in the sense of satisfying conditions (1)--(5)
of \cite[Constr.\ 1.1]{Theta}. Further, for each one-dimensional
good stratum $X_{\rho}$, either (a) $X_{\rho}\cong \PP^1$ and
$\partial X_{\rho}$ consists of two points, or (b) we are in the
relative case and $g|_{X_{\rho}}$ is an isomorphism, with $\partial X_{\rho}$
a single point.
\end{proposition}

\begin{proof}
First, conditions (1) and (2) of \cite[Constr.\ 1.1]{Theta}
follow immediately from the construction of $\P$ as a fan
in $\Div_D(X)^*_{\RR}$. Note this description is possible precisely
because we have made the assumption that the $X_I$ are connected. 

Condition (3) states that every $\rho\in\P$ is contained in an
$n$-dimensional cone $\sigma\in\P$. In other words:
\begin{equation}
\label{eq:induction hypothesis}
\hbox{Every good stratum of $X$ contains a zero-dimensional good stratum.}
\end{equation}
We will show \eqref{eq:induction hypothesis} 
by adapting an argument
of Koll\'ar \cite[Thm.\ 10]{Kollar2}. First note that
by adjunction, for $\rho\in\P$, we may inductively write
\[
K_{X_{\rho}}+\partial X_{\rho}\equiv_{\QQ}\sum_{i: D_i\not\supseteq X_{\rho}}
a_i (D_i\cap X_{\rho}).
\]
Thus the pair $(X_{\rho},\partial X_{\rho})$ may also be viewed as
a log Calabi-Yau variety in our sense, with Assumptions
\ref{ass:absolute}, (2),(3) holding automatically for this pair.

We now prove \eqref{eq:induction hypothesis} by induction on $\dim X$, 
with base case $\dim X=1$. In this case, \eqref{eq:induction hypothesis}
is obvious, but we observe more.
In the absolute case, $X$ is a compact non-singular curve and
$D$ is non-empty. By Assumptions~\ref{ass:absolute}, (2),
all components of $D$ must be good, and thus $K_X+D\equiv 0$. Thus
the only possibility is that $X\cong \PP^1$ and $D$ consists of
two distinct points. In the relative case, by Assumptions~\ref{ass:relative},
(2), $X\rightarrow S$ has connected fibres and hence is an isomorphism,
and thus $D$ consists of one point lying over $0\in S$.

Now assume \eqref{eq:induction hypothesis} is true when $\dim X<n$, and consider
the case $\dim X=n$. Suppose
that $\rho$ is maximal in $\P$, i.e., is not contained in a larger cone
of $\P$, but with $\dim\rho<n$. 
By assumption, there exists an $n$-dimensional cone $\sigma\in\P$.
This gives strata $X_{\rho}$, $X_{\sigma}$ with $\dim X_{\rho}>0$,
$\dim X_{\sigma}=0$. After reordering the indices $1,\ldots,s$, we
may assume $X_{\sigma}\subseteq D_1$ and $X_{\rho}\subseteq D_r$,
and by connectivity of $\partial^{\mathrm{good}}X$ there exists
a sequence of good divisors $D_2,\ldots,D_{r-1}$ such that
$D_i\cap D_{i+1}\not=\emptyset$ for $1\le i \le r-1$.
Noting that $(D_1,\partial D_1)$
satisfies Assumptions~\ref{ass:absolute}, we see by the
inductive hypothesis that necessarily the good stratum $D_1\cap D_2$ 
of $(D_1,\partial D_1)$
contains a zero-dimensional good stratum. Thus the same is true of
$(D_2,\partial D_2)$. Continuing in this way, we see the same
is true of $(D_r,\partial D_r)$, and hence again by the inductive
hypothesis $X_{\rho}$ contains a zero-dimensional good stratum.

Note this argument also shows that all good one-dimensional strata satisfy
conditions (a) or (b) in the statement of the proposition. 
Condition (4) of \cite[Constr.\ 1.1]{Theta}, which states that every
$n-1$-dimensional cone $\rho$ of $\P$ is contained in one or two
$n$-dimensional cones, thus follows immediately. In case (a), $\rho$ is
contained in two maximal cones and in case (b), $\rho$ is contained
in one maximal cone.

Finally, Condition (5) of \cite[Constr.\ 1.1]{Theta} follows
immediately from the assumed connectedness of 
$\partial^{\mathrm{good}} X_{\tau}$ for $\tau\in\P$ with
$\dim\tau\le n-2$.
\end{proof}

\begin{example}
\label{ex:running example 1}
We will illustrate some features of our construction via a simple example.
This is a special case of the blow-ups of toric varieties considered
more generally in \cite{HDTV}, and we refer the reader to that paper for
more details and other more interesting examples.

Write $(\overline X,\overline D)$ for the toric pair given by $\overline
X=\PP^1\times\PP^1\times\PP^1$ and $\overline D$ its toric boundary.
Label the six toric boundary divisors as $\overline D_{i,0}$, $\overline
D_{i,\infty}$, where $\overline D_{i,0}$ indicates the product obtained by
replacing the $i^{th}$ factor in $\PP^1\times\PP^1\times\PP^1$ with
$\{0\}$, and $\overline D_{i,\infty}$ instead replaces the $i^{th}$ factor with
$\{\infty\}$.

Let $Z_1,Z_2\subseteq \overline X$ be the curves
\[
Z_1:=\PP^1\times \{0\}\times \{1\} \subseteq \overline D_{2,0},
\quad Z_2:= \overline D_{1,\infty}\cap \overline D_{2,\infty}=
\{\infty\}\times\{\infty\}\times \PP^1.
\]

Let $\pi:X\rightarrow \overline{X}$ be the blow-up with center
$Z_1\cup Z_2$, with exceptional divisors $E_1,E_2$ sitting over
$Z_1$ and $Z_2$ respectively, see Figure \ref{Fig: Expl.1.4}. Let $D_{i,0}$, $D_{i,\infty}$ be
the strict transforms of $\overline D_{i,0}, \overline D_{i,\infty}$.
\begin{figure}
\begin{center}
\input{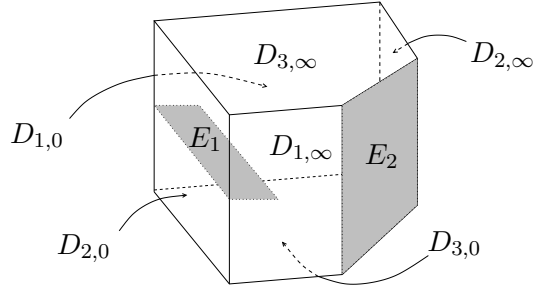}
\caption{\small Sketch of $(X,D)$ in Example~\ref{ex:running example 1}.}
\label{Fig: Expl.1.4}
\end{center}
\end{figure}
Then we take 
\[
D=D_{1,0}+D_{2,0}+D_{3,0}+D_{1,\infty}+D_{2,\infty}+D_{3,\infty}+E_2.
\]
Note that the blow-up along $Z_2$ is a toric blow-up, so if
$\pi$ is factored as $X\rightarrow X' \rightarrow \overline X$
with $X'$ this toric blow-up, $D$ is the strict transform of
the toric boundary of $X'$, and $K_X+D=0$.
Thus all divisors are good, and $(X,D)$ satisfies
Assumptions~\ref{ass:absolute}. 

In this case $(B,\P)$ is piecewise linear isomorphic to the fan
for $X'$.
\end{example}

\subsection{Minimal log Calabi-Yau pairs}
\label{sec:MMP}
We now briefly discuss how restrictive Assumptions~\ref{ass:absolute}
and \ref{ass:relative} might be. Certainly, they impose much stronger
conditions than those imposed in \cite{Assoc}. Indeed, for a simple
example, consider $X$ a rational elliptic surface with $f:X\rightarrow
\PP^1$ the elliptic fibration, and assume the fibre of $f$ over
$p_1\in \PP^1$ is a Kodaira type $I_n$ fibre for some $n\ge 3$.
Take $D=f^*(p_1+p_2+p_3)$ where $p_2,p_3\in \PP^1$ are points over
which the fibre of $f$ is non-singular.
Since $K_X= f^*(-p_1)$, we may in fact write 
\[
K_X+D \sim 2 f^*(p_3).
\]
Thus with this choice of representation \eqref{eq:KXD},
each irreducible component of $f^{-1}(p_1)$ and the 
fibre $f^{-1}(p_2)$ are good divisors. So 
$\partial^{\mathrm{good}}X$ is not connected. Note also 
that $f^{-1}(p_2)$ does not contain a zero-dimensional stratum
even though $\partial^{\mathrm{good}}X$ does.

An obvious problem with this example is that the pair $(X,D)$
is not log Kodaira dimension zero. In particular, if we assume
$(X,D)$ is log Kodaira dimension zero, we conjecturally obtain
the desired hypotheses. The standard conjectures of the minimal model
program in particular would imply the existence of a log Calabi-Yau
minimal model:

\begin{definition}
A \emph{log Calabi-Yau minimal model} is a $\QQ$-factorial
divisorial log terminal (dlt) pair $(X',D')$ such that $K_{X'}+D'\equiv 0$. 
\end{definition}

Let $(X',D')$ be a log Calabi-Yau minimal model. As necessarily each
irreducible component of $D'$ is $\QQ$-Cartier, the argument of
\cite[Thm.\ 4.5]{NXY} shows that
if $C\subseteq D'$ is a one-dimensional
stratum of $D'$, then $X'$ is non-singular and $D'$ is simple normal crossings
in a neighbourhood of $C$.\footnote{The context of \cite{NXY} is in the case
of a degeneration of Calabi-Yau manifolds, but the argument still applies
in our case.} Hence, we may choose a resolution of
singularities $\phi:X\rightarrow X'$ with exceptional divisor $E$,
such that $D=\phi^{-1}_*(D')+E$ is
simple normal crossings, and such that $\phi$ is an isomorphism
in a neighbourhood of each one-dimensional stratum of $D'$.
(Here $\phi^{-1}_*$ is the notation of \cite{Kollar} for the
strict transform.) We may then write
\begin{equation}
\label{eq:discrepancy}
K_X+\phi_*^{-1}D'=\phi^*(K_{X'}+D')+\sum_i a(E_i,X',D')E_i,
\end{equation}
where the $E_i$ are the irreducible components of $E$, and
$a(E_i,X',D')$ is, by definition, the log discrepancy. By
the definition of dlt (see e.g., \cite[Def.\ 2.8]{Kollar}), we have
$a(E_i,X',D')>-1$ for all $i$. Since $K_{X'}+D'\equiv 0$, we may
add $E$ to both sides to obtain
\begin{equation}
\label{eq:choice of KX}
K_X+D\equiv \sum_i a_i D_i
\end{equation}
with $a_i\ge 0$ and $a_i=0$ if and only if $D_i=\phi_*^{-1}(D_j')$
for some $j$. Thus we see that good components of $D$
are precisely the strict transforms of components of $D'$.

\begin{proposition}
\label{prop:dlt CY models work}
Let $(X',D')$ be a log Calabi-Yau minimal model with $X'$
projective, and let
$\phi:(X,D)\rightarrow (X',D')$ be a resolution of singularities
which is an isomorphism in neighbourhoods of one-dimensional strata of $D'$.
Suppose that any intersection of components
of $D'$ is connected\footnote{This requirement is only necessary to
ensure the assumptions on connectivity of strata made throughout this
paper.} and that $D'$ has a zero-dimensional
stratum. Then $(X,D)$ satisfies Assumptions
\ref{ass:absolute} and $(B,\P)$
is the dual intersection cone complex of $(X',D')$.
\end{proposition}

\begin{proof}
Since $D'$ has a zero-dimensional stratum, it follows that every stratum of $D'$
has a zero-dimensional stratum by \cite[Thm.~2,(1)]{KX}. The assumption that
$\phi$ is an isomorphism on neighbourhoods of one-dimensional strata of $D'$
then implies that for any stratum $Z'$ of $D'$, there is a dense open subset
$U'\subseteq Z'$ such that $\phi^{-1}(U')\rightarrow U'$ is an isomorphism. Let
$Z$ be the closure of $\phi^{-1}(U')$: this allows us to pass from a stratum of
$D'$ to a stratum of $D$. In particular, writing $D'=D_1'+\cdots+D'_{s'}$,
$D=D_1+\cdots+D_s$ with $s\ge s'$, we may order the latter divisors so that $D_i
=\phi_*^{-1}D_i'$ for $i\le s'$. Thus if $Z'=\bigcap_{i\in I} D_i'$, then the
corresponding stratum $Z$ coincides with $D_I$. Since the good divisors amongst
the $D_i$ are precisely $D_1,\ldots, D_{s'}$, the claim concerning $(B,\P)$
follows.

Now conditions (1) and (2) of Assumptions~\ref{ass:absolute} hold by assumption,
the first since $D'$ is assumed to have a zero dimensional stratum and the
second because $\phi$ is assumed to be an isomorphism in a neighbourhood of each
one-dimensional stratum of $D'$. Finally (3) holds from results of Koll\'ar and
Xu \cite{KX}. Indeed, if $Z'$ is a stratum of $D'$ of dimension at least two,
and $Z''\subseteq Z'$ is the union of strata properly contained in $Z'$, then
$Z''$ is connected, see \S2, paragraph 16 of \cite{KX}, as well as the
discussion immediately preceding \cite[\S2]{KX}, which shows that $(Z',Z'')$ is
also a dlt log Calabi-Yau pair.
\end{proof}

In dimension larger than three,
the existence of log Calabi-Yau minimal models of pairs $(X,D)$
of log Kodaira dimension zero is expected but currently unknown.
Caucher Birkar, however, has given a criterion for the existence of
minimal models and the necessary connectedness statement of
Assumption~\ref{ass:absolute}. We paraphrase \cite[Thm.~1.4, Cor.~1.5]{Birkar}:

\begin{theorem}
\label{Thm: Birkar}
Let $(X,D)$ be a pair as in \S\ref{subsec:basic trop} with 
$K_X+D\equiv_{\RR} \sum_i a_iD_i$ with $a_i\ge 0$. Suppose
$\P$ has an $n=\dim X$-dimensional cone corresponding to a zero
dimensional stratum consisting of $x\in X$. Let 
$\phi:Y\rightarrow X$ be the blow-up of $X$ at $x$ with exceptional
divisor $E$, and suppose that $\phi^*(K_X+D)-tE$ is not 
pseudo-effective\footnote{i.e., is not a limit of effective $\RR$-divisors.}
for every real number $t>0$. Then $(X,D)$ has a log Calabi-Yau minimal
model, and the connectedness hypothesis Assumptions~\ref{ass:absolute},(3)
holds.
\end{theorem}

In the relative case, there is a different kind of result to which one may
appeal. With $S$ as usual, write $S^{\circ}=S\setminus \{0\}$,
and suppose given a $\QQ$-factorial dlt pair $(X^{\circ},D^{\circ})$ equipped
with a flat morphism $g^{\circ}:X^{\circ}\rightarrow S^{\circ}$
which is a relatively minimal log Calabi-Yau, i.e., the intersection
number of $K_{X^{\circ}}+D^{\circ}$ with any curve contracted by
$g^{\circ}$ is zero. Then \cite[Thm.\ 2]{KNX}, generalizing 
the statements of \cite[Thm.\ 1.4]{BirkarOld} and \cite[Thm.\ 1.1]{HX},
states that there is a finite map $\tau:S'\rightarrow S$ totally ramified
over $0\in S$ and a morphism
$g':(X',D')\rightarrow S'$ such that:
\begin{enumerate}
\item 
The restriction of $g'$ to $\tau^{-1}(S^{\circ})$ is isomorphic
to the pull-back morphism 
\[
(X^{\circ},D^{\circ})\times_{S^{\circ}}
\tau^{-1}(S^{\circ})\rightarrow \tau^{-1}(S^{\circ}).
\]
\item 
$(X',D'+F')$ is dlt and relatively log Calabi-Yau, 
where $F'$ is the fibre of 
$g'$ over $0\in S'$.
\end{enumerate}

Given this, if the general fibre of $g^{\circ}$ has a zero-dimensional stratum
then $(X',D'+F')$ has a zero-dimensional stratum over $0\in S'$. In any event,
as long as $(X',D'+F')$ has a zero-dimensional stratum, we may then resolve
singularities to obtain $(X,D) \rightarrow (X',D'+F')$ and apply
Proposition~\ref{prop:dlt CY models work} to obtain the needed connectedness
statement of Assumption~\ref{ass:absolute}, (3). Some additional care may be
necessary, however, to guarantee that $(X,D)\rightarrow S'$ is log smooth. 

In the classical case of a degeneration of genuine Calabi-Yau manifolds,
the issue of log smoothness is not a concern, and the following proposition 
summarizes the above discussion in this case.

\begin{proposition}
In the situation above, suppose $X^{\circ} \rightarrow S^{\circ}$ is a 
relatively minimal family of non-singular Calabi-Yau manifolds, i.e.,
$K_{X^{\circ}/S^{\circ}}\equiv 0$. Then:
\begin{enumerate}
\item
After a possible base-change $S'\rightarrow S$ branched at $0\in S$, 
there is a dlt relatively minimal model $(X',F')\rightarrow S'$
where $F'$ is the fibre over $0\in S'$ and $X'\setminus F'\rightarrow
S'\setminus\{0\}$ is the
base-change of $X^{\circ}\rightarrow S^{\circ}$.  
\item If the monodromy of the family $X^{\circ}\rightarrow S^{\circ}$
is maximally unipotent, 
then there is a resolution of singularities
$\phi:(X,D)\rightarrow (X',F')$ and an expression
\[
K_X+D\equiv \sum a_i D_i
\]
with $a_i\not= 0$ if and only if $D_i$ is exceptional for $\phi$.
Further, possibly after first performing a further base-change,
the data $(X,D)\rightarrow S'$ together with this description
of $K_X+D$ satisfies Assumptions~\ref{ass:relative}.
\end{enumerate}
\end{proposition}

\begin{proof}
(1) follows from the discussion preceding the proposition. 

For (2), first note that by \cite[Thm.~4.5]{NXY}, $X'$ is non-singular
in some neighbourhood of each one-dimensional stratum of $F'$ and
$D'$ is simple normal crossings in that neighbourhood. Thus we may
find a resolution $\phi:(X,D)\rightarrow (X',F')$ with (1) $D$ a 
normal crossings divisor which is the reduction
of $\phi^{-1}(F')$; (2) $\phi$ is an 
isomorphism in neighbourhoods of one-dimensional strata of $F'$; and
(3) $\phi$
induces an isomorphism $X\setminus D\cong X'\setminus F'$. 
After taking further base-change and resolving the resulting 
singularities in a standard crepant
toric way appropriately (see \cite[II,\S 4, III, Thm.~4.1]{KKMS}), 
we may further assume that any intersection of
irreducible components of $D$ is connected. We may then use
\eqref{eq:discrepancy} and \eqref{eq:choice of KX} to obtain the
desired expression for $K_X+D$.

The condition of Assumption~\ref{ass:absolute}, (1) is then implied 
by \cite[Thm.~4.1.10]{NX}
if the family $X\rightarrow S'$ is maximally unipotent, 
The remaining conditions of Assumptions~\ref{ass:absolute} and
\ref{ass:relative} are then immediate by construction and
Proposition~\ref{prop:dlt CY models work}.
\end{proof}

Of course the existence of a maximally unipotent degeneration for a
type of Calabi-Yau manifold has been classically viewed as a 
prerequisite for the existence of a mirror.

The conclusion is that the hypotheses of Assumptions~\ref{ass:absolute}
and \ref{ass:relative} should be expected to hold in the cases of
interest for mirror symmetry.

\subsection{The affine structure}
\label{subsec:the affine structure}
We continue to assume that $(X,D)$ satisfies Assumptions~\ref{ass:absolute}
or \ref{ass:relative} in the absolute or relative cases, with $\dim X=n$.

We set 
\[
\Delta:=\bigcup_{\sigma\in\P\atop\dim\sigma=n-2}\sigma
\]
and 
\[
B_0:=B\setminus\Delta.
\]
We also denote by $\partial B$ the union of those $\rho\in \P^{[n-1]}$
contained in only one top-dimensional cone, and write $\P_{\partial}\subseteq
\P$ for the set of cones contained in $\partial B$.
Set $\scrP_{\inte}=\P\setminus \P_{\partial}$.
By Proposition~\ref{prop:polyhedral pseudo}, in the
absolute case $\partial B=\emptyset$.

We define an integral affine structure 
on $B_0$ as follows.  For $\tau\in\P$, denote
the \emph{open star} of $\tau$ to be
\[
\Star(\tau):=\bigcup_{\tau\subseteq\sigma\in\P} \Int(\sigma).
\]
If $\sigma\in\P_{\max}$, then $\sigma$, hence $\Int(\sigma)$,
is already endowed naturally with an affine coordinate chart arising
from its linear embedding in $\Div_D(X)^*_{\RR}$.
If $\rho\in\P^{[n-1]}_{\partial}$, then $\Star(\rho)$ inherits
an affine structure with boundary from the unique maximal
cone $\sigma$ containing $\rho$. On the other hand, if
$\rho\in\P^{[n-1]}_{\inte}$, then $\rho=\sigma\cap\sigma'$ with
$\sigma,\sigma'\in\P^{[n]}$. We then define an embedding
\begin{equation}
\label{eq:psi rho def}
\psi_{\rho}:\sigma\cup\sigma'\rightarrow \RR^n,
\end{equation}
well-defined up to an element of $\GL_n(\ZZ)$, 
as follows. Let $\rho$ be generated by $D_{i_1}^*,\ldots,D_{i_{n-1}}^*$,
and assume that $\sigma$ is generated by $\rho$ and $D_{i_n}^*$,
while $\sigma'$ is generated by $\rho$ and $D_{i_n'}^*$.
Choose integral bases $e_1,\ldots,e_n$ and $e_1,\ldots,e_{n-1}, e_n'$
of $\RR^n$ subject to the constraint that 
\begin{equation}
\label{eq:invariant form of chart}
e_n+e_n'=-\sum_{j=1}^{n-1} (D_{i_j}\cdot X_{\rho}) e_j.
\end{equation}
We note the intersection numbers are defined: the fact that
$\rho\in \P_{\inte}$ implies $X_{\rho}$ is proper.
We then define $\psi_{\rho}$ to be piecewise linear, linear on
each cone, via 
\[
\hbox{$\psi_\rho(D_{i_j}^*)=e_j$, $1\le j\le n$, and
$\psi_{\rho}(D_{i_n'}^*)=e_n'$.}
\]

This gives rise to affine charts $\psi_{\rho}:\Star(\rho)\rightarrow
\RR^n$, and hence an affine structure on $B_0$. In the language
of \cite[Const.\ 1.1]{Theta}, this gives $(B,\P)$ the structure
of a \emph{polyhedral affine pseudomanifold}.

\begin{remark}
This affine structure was first given in \cite{Utah} in the case
that $K_X+D=0$. More generally, in the relative case $g:X\rightarrow
S$ with $D=g^{-1}(0)$, so that $g$ is a degeneration of Calabi-Yau
manifolds, \cite{NXY} showed that the above affine structure
was the correct affine structure on the base of the non-Archimedean
SYZ fibration, at least in the case that $(X,D)$ is 
a resolution
of a dlt relatively minimal model $(X',D')$ as in \S\ref{sec:MMP}.
In particular, they showed
that if $Z\subseteq D$ was a one-dimensional stratum and $D_i\cdot Z<0$ 
for all $D_i$
containing $Z$, then the formal completion of $X$ along $Z$ is
isomorphic to the formal completion of a one-dimensional
stratum in a toric variety. The fan determining the toric variety
then determines the affine structure above.
\end{remark}

In our situation, the affine structure can be
viewed as natural from the logarithmic point of view.

\begin{lemma}
\label{lem:toric dim one}
Let $\rho\in\P^{[n-1]}_{\inte}$, 
$\rho\subseteq\sigma,\sigma'\in\P^{\max}$,
with chart $\psi_{\rho}:\sigma\cup\sigma'\rightarrow\RR^n$ constructed
above. Let $\Sigma_{\rho}$ be the fan consisting of the cones 
$\psi_{\rho}(\sigma)$, $\psi_{\rho}(\sigma')$
and their faces, and let $X_{\Sigma_{\rho}}$ be the corresponding
toric variety. For $\tau\in\Sigma_{\rho}$, write $X_{\Sigma_{\rho},\tau}$ for
the corresponding stratum of $X_{\Sigma_{\rho}}$. Let
$X_{\rho}\subseteq X$, $X_{\Sigma_{\rho},\psi_{\rho}(\rho)}\subseteq X_{\Sigma_{\rho}}$
be given the log structures making these inclusions strict, with
$X_{\Sigma_{\rho}}$ carrying the standard toric log structure.
Then there is an isomorphism
\[
X_{\rho}\cong X_{\Sigma_{\rho},\psi_{\rho}(\rho)}
\]
as log schemes over $\Spec\kk$.
\end{lemma}

\begin{proof}
By Proposition~\ref{prop:polyhedral pseudo}, $X_{\rho}\cong\PP^1$, and certainly
$X_{\Sigma_{\rho},\psi_{\rho}(\rho)}\cong\PP^1$, so the underlying schemes
are isomorphic. Further, $X_{\rho}$ contains precisely two
zero-dimensional strata, $X_{\sigma}$ and $X_{\sigma'}$. 
Let $X_{\rho}'$ denote the log structure on $\ul{X}_{\rho}$
induced by the divisor $X_{\sigma}\cup X_{\sigma'}$. If
$X_{\rho}$ is contained in divisors $D_{i_1},\ldots,D_{i_{n-1}}$ 
corresponding to the edges of $\rho$,
let $\shM_j$ denote the restriction to $\ul{X}_{\rho}$ of the divisorial log
structure on $X$ induced by the divisor $D_{i_j}$. This log
structure is determined by the line bundle $\O_X(D_{i_j})|_{X_{\rho}}$.
Finally,
\[
\shM_{X_{\rho}}=\shM_{X'_{\rho}}\oplus_{\O_{X_{\rho}}^{\times}}
\bigoplus_j \shM_j,
\]
where all pushouts are over $\O_{X_{\rho}}^{\times}$.

On the other hand, $\shM_{X_{\Sigma_{\rho},\psi_{\rho}(\rho)}}$ has
a similar description, and if $D'_{i_1},\ldots,D'_{i_{n-1}}$
are the toric divisors of $X_{\Sigma_{\rho}}$ corresponding to the
edges of $\psi_{\rho}(\rho)$, we have 
\begin{equation}
\label{eq:degrees agree}
\deg\O_X(D_{i_j})|_{X_{\rho}}=\deg\O_{X_{\Sigma_{\rho}}}(D'_{i_j})
|_{X_{\Sigma_{\rho},\psi_{\rho}(\rho)}}
\end{equation}
by \cite{Oda}, pg. 52 and the definition of $\psi_{\rho}$.
Thus $X_{\rho}$ and $X_{\Sigma_{\rho},\psi_{\rho}(\rho)}$ 
are in fact isomorphic as 
log schemes.  
\end{proof}

\begin{example}
\label{ex:running example 2}
We return to Example~\ref{ex:running example 1}. We first
remark that in general for a minimal log Calabi-Yau pair
satisfying Assumptions~\ref{ass:absolute},
the affine structure described on $B_0$,
the complement of the union of codimension two cones of $\P$, in fact extends
across the interior of all cones $\rho\in \P$ such that
$(X_{\rho},\partial X_{\rho})$ is a toric pair.
See \cite[Prop.~2.3]{HDTV} for details. In the case of
Example~\ref{ex:running example 1},
only the strict transforms of boundary divisors of $\overline{X}$
which meet the center $Z_1$ transversally fail to be toric: these are
$D_{1,0}$ and $D_{1,\infty}$.
Hence, after extending the affine structure where possible, the
smaller discriminant locus $\Delta$ is homeomorphic to the line
$\RR_{\ge 0} D_{1,0}^*\cup \RR_{\ge 0} D_{1,\infty}^*$.
The monodromy about this singularity may be calculated, and takes the form
$\begin{pmatrix} 1&1&0\\0&1&0\\ 0&0&1\end{pmatrix}$. See
\cite[\S3.2]{HDTV} for related calculations carried out in a more
general setting. This is the simplest
kind of singularity appearing in affine manifolds in our program.
\end{example}

We recall some standard notation, e.g., from \cite{Theta}.

\begin{definition}
We write
\[
B(\ZZ):= B \cap \Div_D(X)^*,
\]
the set of integral points on $B$.
\end{definition}

\begin{definition}
\label{def:Lambda}
We take $\Lambda$ to be the local system on $B_0=B\setminus\Delta$ consisting of
integral tangent vectors, and write
$\check\Lambda:=\Hom(\Lambda,\ZZ)$.
\end{definition}

\begin{construction}
\label{const:phi canonical}
We now fix a finitely generated abelian group $H_2(X)$ of possible
degree data for curves on $X$, see, e.g.,
\cite{Assoc}, Basic Setup 1.6 for a discussion on this.
We may take, for example, $H_2(X)=H_2(X,\ZZ)$ if working over
$\kk=\CC$, or may take $H_2(X)=N_1(X)$, curve classes modulo numerical
equivalence. We require that every proper curve $C\subseteq X$
gives a curve class $[C]\in H_2(X)$, and that for any
stable map $f:C\rightarrow X$, there is a well-defined 
curve class $f_*[C]\in H_2(X)$. We also require use of
intersection numbers with divisors, i.e., a pairing
$H_2(X)\times \Div(X)\rightarrow \ZZ$.
As in \cite{Assoc}, Basic Setup~1.6, we assume that $f_*[C]$
is torsion if and only if $f$ is a constant map, in which case 
$f_*[C]=0$.

In the relative case, the inclusion $X_s \hookrightarrow X$ for
$s\in S$ should induce a natural map
\begin{equation}
\label{eq:iota def}
\iota:H_2(X_s)\rightarrow H_2(X).
\end{equation}

We further choose a finitely generated monoid
$Q\subseteq H_2(X)$ such that (1) $Q$ contains the classes of all
stable maps to $\ul{X}$; (2) $Q$ is saturated; (3) the group of
invertible elements of $Q$ coincides with the torsion part of
$H_2(X)$.

We may then define a multi-valued piecewise linear (MPL) function
$\varphi$ with values in $Q^{\gp}_{\RR}$, see \cite{Theta}, Definition 1.8. 
Such a function consists of a choice, 
for each $\rho\in\P^{[n-1]}_{\inte}$, of a single-valued PL function
$\varphi_{\rho}:\Star(\rho)\rightarrow Q^{\gp}_{\RR}$, well-defined
up to linear functions. The kink of $\varphi_{\rho}$ is defined as follows. Let
$x\in\Int(\rho)$, and let $\rho\subseteq \sigma,\sigma'\in\P_{\max}$. 
Let $n,n'\in \check\Lambda_x$ be the slopes of 
$\varphi_{\rho}|_{\sigma}$ and $\varphi_{\rho}|_{\sigma'}$. Then
we may write
\begin{equation}
\label{def:kink}
n'-n=\delta\cdot \kappa_{\rho}
\end{equation}
where $\kappa_{\rho}\in Q^{\gp}$ and
$\delta:\Lambda_x\rightarrow \ZZ$ is the surjective map
which vanishes on tangent vectors to $\rho$ and is positive on
tangent vectors pointing into $\sigma'$.
Then $\kappa_{\rho}$ is the kink of $\varphi_{\rho}$.

An MPL function is completely determined by giving its kinks,\footnote{In 
\cite{Theta}, we always worked with a monoid $Q$ which was
torsion free. Of course, if $Q$ has torsion, then the kink as defined
above only lies in $Q^{\gp}/Q^{\gp}_{\tors}$. However, if we wish
to work with a curve class group $H_2(X)$ which has torsion, then
it is more natural to view the MPL function simply as a collection of
kinks in $Q^{\gp}=H_2(X)$. This will only affect the definition of
the sheaf $\shP$, discussed in \S\ref{subsec:recall of monomials}.}
see \cite[Prop.~1.9]{Theta}. Throughout
this paper, we will work with $\varphi$ such that
for $\rho\in\P^{[n-1]}_{\inte}$,
\[
\kappa_{\rho}=[X_{\rho}] \in H_2(X).
\]

Note that this curve class makes sense as $\rho\not\subseteq
\partial B$ implies that $X_{\rho}$ is proper. Further, the class
$[X_{\rho}]$ is then not invertible in $Q$.
\end{construction}

The data of the polyhedral affine manifold $(B,\P)$ and the MPL
function $\varphi$ are the necessary ingredients for the setup
of wall structures of \cite{Theta}, \S2. 

\medskip

We end this subsection by examining, 
in the relative case, some additional structure. As 
tropicalization is functorial, we obtain a map
\[
\Sigma(g):\Sigma(X)\rightarrow \Sigma(S)=\RR_{\ge 0},
\]
which restricts to a map
\[
g_{\trop}:B\rightarrow \RR_{\ge 0}.
\]
Explicitly, viewing $\Sigma(X)$ or $B$ as a subset of $\Div_D(X)^*_{\RR}$,
the (not necessarily reduced) divisor $g^*(0)$ induces 
a linear function on the latter space,
which restricts to $\Sigma(g)$ or $g_{\trop}$ in the two cases.

\begin{proposition}
\label{prop:gtrop affine submersion}
$g_{\trop}$ is an affine submersion, and $\partial B$ is the 
union of all $(n-1)$-dimensional cones contained in $g_{\trop}^{-1}(0)$.
\end{proposition}

\begin{proof}
For the first statement,
it is sufficient to show that $(g_{\trop}|_{\Star(\rho)})\circ 
\psi_{\rho}^{-1}$ is a surjective linear function 
on the fan $\Sigma_{\rho}$ for each 
$\rho\in \P^{[n-1]}_{\inte}$. Here $\psi_{\rho}$ is the
local chart \eqref{eq:psi rho def} and $\Sigma_{\rho}$ is as defined
in Lemma~\ref{lem:toric dim one}.

If we write the divisor $g^*(0)$ as $\sum_i b_i D_i$, then
$g_{\trop}$ takes the value $b_i$ on $D_i^*$. In particular,
$(g_{\trop}|_{\Star(\rho)})\circ \psi_{\rho}^{-1}$ is the piecewise
linear function on $\Sigma_{\rho}$
corresponding to a toric divisor $D'$ on $X_{\Sigma_{\rho}}$ which
may be written explicitly as follows. Denote by $D'_j$ the toric
divisor of $X_{\Sigma_{\rho}}$ corresponding to the ray $\psi_{\rho}(\RR_{\ge 0}
D^*_j)$, for $j\in \{i_1,\ldots,i_n,i_n'\}$. Then
\[
D'= 
b_{i_n'}D'_{i_n'}+
\sum_{j=1}^n 
b_{i_j}D'_{i_j}.
\]
However, by \eqref{eq:degrees agree}, $D'\cdot 
X_{\Sigma_{\rho},\psi_{\rho}(\rho)}
= g^*(0)\cdot X_{\rho}=0$ as $g^*(0)$ is algebraically equivalent to 
the trivial divisor.
Thus $D'$ is linearly equivalent to zero, so 
$(g_{\trop}|_{\Star(\rho)})\circ \psi_{\rho}^{-1}$ is linear, as claimed.
This latter map is also non-zero, as $X_{\rho}$ is contained in $g^{-1}(0)$,
and hence it is surjective.

For the second statement, by definition $\partial B$
is the union of those codimension one cones $\rho$ only contained in
one maximal cone. But these correspond to the non-compact one-dimensional
good strata $X_{\rho}$, which are precisely the ones for which $g|_{X_{\rho}}$
is surjective. However by the definition of $g_{\trop}$, these are
in one-to-one correspondence with the set of those $\rho\in \P^{[n-1]}$
with $g_{\trop}(\rho)=0$.
\end{proof}

Frequently, in the relative case $g:X\rightarrow S$,
it is more intuitive to work not
with $(B,\P)$ but with the fibre of $g_{\trop}:B
\rightarrow \RR_{\ge 0}$ over $1$. 
This fits more
closely with the earlier point of view of affine manifolds associated
with degenerations and wall structures in \cite{PartI,Annals}. Set
$B':=g_{\trop}^{-1}(1)$, equipped with the polyhedral decomposition
\[
\scrP':=\{\sigma \cap B' \,|\, \sigma\in\P\}.
\]
Note that the elements of $\scrP'$ are no longer, in general, cones.
We then observe:

\begin{proposition}
\label{prop:relative B prime}
Suppose that all good divisors contained in $g^{-1}(0)$ have multiplicity
one in $g^{-1}(0)$. Then $(B',\scrP')$ is a polyhedral affine
pseudomanifold.
\end{proposition}

\begin{proof}
As $g_{\trop}$ is an affine submersion on $B_0$ by
Proposition~\ref{prop:gtrop affine submersion}, $B'\cap B_0$ acquires
an affine structure. Thus we just need to check that 
$(B',\scrP')$ satisfies the conditions of \cite[Constr.~1.1]{Theta}.
Note that if $\rho\in\P$ is a ray with $g_{\trop}|_{\rho}:\rho
\rightarrow \RR_{\ge 0}$ surjective, the corresponding
divisor $X_{\rho}$ is contained in $g^{-1}(0)$ and is good, hence 
appears with multiplicity one in $g^{-1}(0)$ by assumption. Thus 
$g_{\trop}^{-1}(1)\cap\rho$ is an integral point of $\rho$, so all elements
of $\scrP'$ are in fact lattice polytopes. Conditions (1)--(5)
of \cite[Constr.~1.1]{Theta} now follow immediately from the same
conditions for $\scrP$.
\end{proof}

\begin{remark}
\label{rem:phi restrict}
In the above situation, given an MPL function $\varphi$ on $(B,\scrP)$,
we obtain by restricting representatives of an MPL function
$\varphi|_{B'}$. Note that if $\rho\in\P_{\inte}^{[n-1]}$, then
the kink $\kappa_{\rho}$ of $\varphi$ at $\rho$ then agrees with
the kink $\kappa_{\rho\cap B'}$ of $\varphi|_{B'}$ along $\rho\cap B'$.
This follows immediately from the definition \eqref{def:kink}.
\end{remark}

\begin{remark}
\label{rem:cone}
In \cite[\S4.2]{Theta}, given a polyhedral affine pseudomanifold with
singularities $B$, we defined a polyhedral affine pseudomanifold
$\mathbf{C}B$, the \emph{cone} over $B$. In our current situation,
$B$ is a conical polyhedral affine pseudomanifold, and $B'$ as defined
above satisfies $\mathbf{C}B'=B$. Here we find it more natural
to work with the cone.
\end{remark}

We next consider a special case of the relative situation, satisfying
Assumption~\ref{ass:relative}, when $S$ is an affine curve and
$(X,D)\rightarrow S$ is a degeneration of log Calabi-Yau varieties
which themselves satisfy Asssumption~\ref{ass:absolute}.
Recall from Assumption~\ref{ass:relative} that
in the relative case, for $s\in S$ a closed point, $s\not=0$, 
we have $D_j\cap X_s$ is an irreducible
divisor on $X_s$. We may then write
\[
K_{X_s}+D_s\equiv_{\QQ} \sum_{j} a_j (D_j\cap X_s).
\]
Using this representative for the numerical equivalence class
of $K_{X_s}+D_s$, it then makes sense to ask that the pair
$(X_s,D_s)$ satisfy Assumption~\ref{ass:absolute}. We
write $B_{(X_s,D_s)}$ for the polyhedral affine pseudomanifold
associated to the pair $(X_s,D_s)$: this will be pure $(n-1)$-dimensional,
still with $n=\dim X$. We wish to understand the relationship between
$B_{(X_s,D_s)}$ and $B$, the polyhedral affine pseudomanifold
associated to the pair $(X,D)$. This is described by the following
proposition, which will be of use in \cite{HDTV}.

\begin{proposition}
\label{prop:relative case maximal fibres}
Assume we are in the relative case, with Assumption~\ref{ass:relative} holding. 
Suppose that as described above, for
$s\in S$ a closed point, $s\not=0$, $(X_s,D_s)$ satisfies
Assumption~\ref{ass:absolute}.
Then:
\begin{enumerate}
\item $\partial B=g_{\trop}^{-1}(0)$ is naturally identified with the
polyhedral cone complex $B_{(X_s,D_s)}$.
\item The structure of integral affine manifold with boundary on 
$B_0$ can be extended across the interiors of cells 
$\omega\in\P^{[n-2]}_{\partial}$. Further, the induced polyhedral affine
pseudomanifold structure on $\partial B$ then agrees
with that determined by the pair $(X_s,D_s)$.
\item Let $\iota:H_2(X_s)\rightarrow H_2(X)$ be the map of groups
of curve classes induced by the inclusion of the fibre $X_s\hookrightarrow X$.
Then for each $\omega\in \P^{[n-2]}_{\partial}$, the
MPL function $\varphi$ of Construction~\ref{const:phi canonical} has a 
single-valued representative in $\Star(\omega)$. Further, the restriction
of this single-valued representative to $\Star(\omega)\cap \partial B$
has kink $\iota([X_{\omega}\cap X_s])$. We write the MPL function
on $\partial B$ with such kinks as $\varphi|_{\partial B}$.
\end{enumerate}
\end{proposition}

\begin{proof}
{\bf Step 1.} \emph{Proof of (1).}
We call a stratum $Z$ of $X$ \emph{vertical} if $g|_Z$ is constant
and otherwise say $Z$ is \emph{horizontal}.
Note that all vertical strata lie over $0\in S$. The rays of $\P$
contained in $g_{\trop}^{-1}(0)$ correspond precisely to the good
horizontal components of $D$. Thus there is a one-to-one
correspondence between cones $\rho\in\P$ with $\rho\subseteq
g_{\trop}^{-1}(0)$ and good horizontal strata of $D$. 
There is also a one-to-one correspondence 
between horizontal strata of $X$ and strata of $X_s$, taking a 
horizontal stratum $Z$ to the stratum $Z\cap X_s$. 
Note that by Assumption~\ref{ass:relative},(2), this is indeed an 
irreducible stratum. Conversely, every stratum of $X_s$ is necessarily
of this form by definition of $D_s$.
This correspondence takes good strata to good strata. 
Thus $\P_{(X_s,D_s)}$
may be identified with the set of cones of $\P$ contained in 
$g_{\trop}^{-1}(0)$. Further, since $(X_s,D_s)$ satisfies Assumptions 
\ref{ass:absolute}, $B_{(X_s,D_s)}$ is pure $(n-1)$-dimensional, and hence
coincides with $\partial B$. This gives (1).

\medskip

{\bf Step 2.} \emph{Analysis of horizontal good two-dimensional strata.}
For (2), we need to give, for $\omega\in\P^{[n-2]}_{\partial}$, 
a chart $\psi_{\omega}:\Star(\omega)\rightarrow \RR^n$ which agrees, up to 
elements of $\GL_n(\ZZ)$, with $\psi_{\rho}$ on 
$\Star(\rho)\subseteq\Star(\omega)$ whenever $\omega\subset\rho\in
\P^{[n-1]}$.

To do so, we first analyze the geometry of the surface $X_{\omega}$. 
As $X_{\omega}$ is
horizontal, we obtain a non-constant map $g_{\omega}:X_{\omega}\rightarrow S$.
We first analyze $\partial X_{\omega}$. Since $X_{\omega}\cap X_s$
is one-dimensional, $X_{\omega}\cap X_s$ is disjoint from any bad
divisor on $X_s$ by Assumptions~\ref{ass:absolute}, (2). Hence
$X_{\omega}$ is disjoint from any horizontal bad divisor. On the
other hand, $X_{\omega}\cap X_s$ necessarily contains precisely
two zero-dimensional strata, and hence $X_{\omega}$ contains two 
horizontal one-dimensional strata $X_{\rho}, X_{\rho'}$ with
$\omega\subseteq\rho,\rho'$. These are the only horizontal one-dimensional
strata of $X_{\omega}$. Note they are necessarily sections
of $g_{\omega}$ by Assumptions~\ref{ass:relative},(2).
However, since $\partial^{\mathrm{good}}X_{\omega}$ is
connected, there must be a chain of vertical good 
one-dimensional strata connecting
$X_{\rho}$ and $X_{\rho'}$. Again, by Assumptions~\ref{ass:absolute},
(2), these are disjoint from any bad divisor. 
Now the union of vertical strata coincides with
$g_{\omega}^{-1}(0)$ and
$g^{-1}_{\omega}(0)$
is connected (again by Assumption~\ref{ass:relative},(2)),
so it follows that all one-dimensional strata of 
$X_{\omega}$ are disjoint from bad divisors. Thus $X_{\omega}$
is disjoint from all bad divisors.

By Proposition~\ref{prop:polyhedral pseudo},
each vertical one-dimensional
stratum in $X_{\omega}$, being good, contains precisely two 
zero-dimensional strata. On the other hand, $g_{\omega}^{-1}(0)$
contains two zero-dimensional strata
given by $X_{\rho}\cap g_{\omega}^{-1}(0)$ and $X_{\rho'}\cap
g_{\omega}^{-1}(0)$, and all other zero-dimensional strata
of $X_{\omega}$ are then intersections of two good vertical one-dimensional
strata. Thus the only possibility for the set of one-dimensional strata
of $X_{\omega}$ is as follows: this set of strata
may be written as $\{C_0,\ldots, C_r\}$, with $C_0=X_{\rho}$,
$C_r=X_{\rho'}$, $C_1,\ldots,C_{r-1}$ vertical strata, and with
$C_i\cap C_j\not=\emptyset$ if and only if $|i-j|\le 1$. Further,
$C_i\cap C_{i+1}$ consists of one point. See Figure~\ref{Fig: Xomega}.
\begin{figure}
\begin{center}
\input{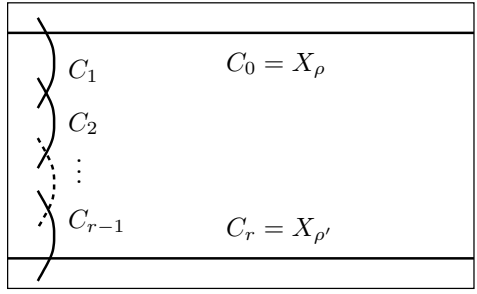}
\caption{\small The two-dimensional stratum $X_\omega$.}
\label{Fig: Xomega}
\end{center}
\end{figure}

Now let $(B_{\omega},\P_{\omega})$ be the two-dimensional
polyhedral affine pseudomanifold
associated to the log Calabi-Yau variety $(X_{\omega},\partial X_{\omega})$
defined over $S$.
We will describe $B_{\omega}$ and $\scrP_{\omega}$ explicitly.
Let $\tau_0,\ldots,\tau_r$ be the rays of
$\scrP_{\omega}$ corresponding to $C_0,\ldots,C_r$. Now define
$n_{\tau_0}=(0,1)\in\ZZ^2$, $n_{\tau_1}=(1,0)\in\ZZ^2$, and inductively
obtain $n_{\tau_2},\ldots,n_{\tau_r}\in\ZZ^2$ by requiring that
\begin{equation}
\label{eq:ntau relation}
n_{\tau_{\ell-1}}+n_{\tau_{\ell+1}}=-C_\ell^2 n_{\tau_\ell},\quad 1\le\ell 
\le r-1.
\end{equation}
We may define a coordinate chart
$\psi_{\tau_\ell}:\Star(\tau_\ell)\rightarrow \RR^2$ by mapping primitive 
generators of $\tau_{\ell-1}, \tau_\ell$ and $\tau_{\ell+1}$ to 
$n_{\tau_{\ell-1}}, n_{\tau_\ell}$ and $n_{\tau_{\ell+1}}$ respectively.
By \eqref{eq:invariant form of chart}, this coordinate chart is
compatible with the affine structure on $B_{\omega}\setminus \{0\}$.
Indeed, in \eqref{eq:invariant form of chart}, we have
$D_{i_j}=X_{\tau_\ell}=C_\ell$.

Now let $g_{\omega,\trop}:B_{\omega}\rightarrow \RR_{\ge 0}$
be the tropicalization of $X_{\omega}\rightarrow S$.
Note that $g_{\omega,\trop}|_{\Star(\tau_1)}\circ \psi^{-1}_{\tau_1}$
is given by $(1,0)\in (\ZZ^2)^*$. Indeed, as $C_0$ is not contained
in $g^{-1}(0)$, $g_{\omega,\trop}$ vanishes on $C_0$. On the other hand,
the multiplicity of $C_1$ in $g_{\omega}^{-1}(0)$ must be $1$, as $C_0
\cdot g_{\omega}^{-1}(0)=1$, $C_0$ being a section. Thus $g_{\omega,\trop}$
takes the value $1$ on the primitive generator of $\tau_1$.

Since $\psi_{\tau_\ell}$ and $\psi_{\tau_{\ell+1}}$ agree on 
$\Star(\tau_\ell)\cap \Star(\tau_{\ell+1})$, in fact $g_{\omega,\trop}$ is
given by $(0,1)$ on all charts. Since $g_{\omega,\trop}$ is positive
on generators of $\tau_1,\ldots,\tau_{r-1}$ and takes the value $0$
on $\tau_r$, we see that all $n_{\tau_\ell}$ lie in $\RR_{\ge 0}\times 
\RR$ and $n_{\tau_r}=(0,-1)$. 
Taking $\Sigma_{\omega}$ to be the fan in $\RR^2$ with one-dimensional
cones generated by $n_{\tau_\ell}$, $0\le \ell \le r$ and two-dimensional cones
spanned by $n_{\tau_\ell}, n_{\tau_{\ell+1}}$, $0\le \ell \le r-1$, we see
that the charts $\psi_{\rho_\ell}$ glue to give an identification of
$B_{\omega}$ with $\RR_{\ge 0}\times\RR$ and $\P_{\omega}$ with
$\Sigma_{\omega}$, so that $B_{\omega}\setminus \{0\}$ is
affine isomorphic to $(\RR_{\ge 0}\times \RR)\setminus \{(0,0)\}$.
Further, under this identification, $g_{\omega,\trop}$ is given by
projection onto $\RR_{\ge 0}$.

\medskip

{\bf Step 3.} \emph{A chart on $\Star(\omega)$.}
We may now build a chart for $\Star(\omega)$ as follows. Let
$\omega$ be generated by $D_{i_1}^*,\ldots,D_{i_{n-2}}^*$, and let
$\psi_j$, $1\le j \le n-2$, be the $\RR$-valued piecewise linear function on
$B_{\omega}$ with kink along $\tau_\ell$ being $D_{i_j}\cdot C_\ell$.
Explicitly, such $\psi_j$ may be constructed by having it take the
value $0$ on some choice of two-dimensional cone in $\P_{\omega}$,
and then using the relation
\begin{equation}
\label{eq:psij}
\psi_j(n_{\tau_{\ell-1}})+\psi_j(n_{\tau_{\ell+1}})=
D_{i_j}\cdot C_\ell  -C_\ell^2 \psi_j(n_{\tau_{\ell}})
\end{equation}
to determine the value of $\psi_j$ on $n_{\tau_{\ell+1}}$ inductively.
Indeed, the reader may easily check that the function defined on
$\Star(\tau_{\ell})$ satisfying \eqref{eq:psij} has kink $D_{i_j}\cdot
C_{\ell}$ using \eqref{def:kink} and \eqref{eq:ntau relation}.

Let $\rho_0,\ldots,\rho_r$ be the codimension one cones containing
$\omega$, with $X_{\rho_\ell}=C_\ell$, and let $D_{k_\ell}^*$ be such that
$\omega$ and $D_{k_{\ell}}^*$ generate $\rho_\ell$, $0\le\ell\le r$.
We may then define
$\psi_{\omega}:\Star(\omega)\rightarrow B_{\omega}\times \RR^{n-2}$ via, with $e_1,\ldots,e_{n-2}$
the standard basis for $\RR^{n-2}$,
\[
\psi_{\omega}(D_{i_j}^*)= (0,e_j), \quad\quad
 1\le j \le n-2
\]
and
\begin{equation}
\label{eq:psi omega def}
\psi_{\omega}(D_{k_{\ell}}^*)=  \left(n_{\tau_{\ell}},-\sum_{j=1}^{n-2} 
\psi_j(n_{\tau_{\ell}})e_j \right), \quad\quad 0\le \ell\le r.
\end{equation}
We then extend $\psi_{\omega}$ linearly on each cone containing $\omega$.
It is easy to see that $\psi_{\omega}$ induces an embedding of
$\Star(\omega)$ into $B_{\omega}\times\RR^{n-2}$. It is thus sufficient
to check that $\psi_{\omega}$ is compatible with each $\psi_{\rho_\ell}$,
$1\le \ell\le r-1$, as defined in \eqref{eq:psi rho def} and
\eqref{eq:invariant form of chart}.
For this, it is enough to verify the equality
\eqref{eq:invariant form of chart}. We have for $1\le \ell\le r-1$,
\begin{align*}
\psi_{\omega}(D_{k_{\ell-1}}^*) + \psi_{\omega}(D_{k_{\ell+1}}^*)
\stackrel{\eqref{eq:psi omega def}}{=} {} & 
\left(n_{\tau_{\ell-1}}+n_{\tau_{\ell+1}},
-\sum_{j=1}^{n-2}(\psi_j(n_{\tau_{\ell-1}})+\psi_j(n_{\tau_{\ell+1}}))e_j
\right)\\
\stackrel{\eqref{eq:ntau relation},\eqref{eq:psij}}{=} {} & 
\left( -C_\ell^2 n_{\tau_{\ell}} , -\sum_{j=1}^{n-2} 
(D_{i_j}\cdot C_{\ell}-C_\ell^2\psi_j(n_{\tau_\ell}))e_j\right)\\
= {} & -(D_{k_{\ell}}\cdot X_{\rho_{\ell}})\psi_{\omega}(D_{k_{\ell}}^*)-
\sum_{j=1}^{n-2} (D_{i_j}\cdot X_{\rho_{\ell}}) \psi_{\omega}(D_{i_j}^*),
\end{align*}
the last equality by $C_{\ell}^2=D_{k_\ell}\cdot X_{\rho_{\ell}}$.
This verifies \eqref{eq:invariant form of chart},
as desired. Thus we obtain an affine chart on $\Star(\omega)$
compatible with the affine structure on $B_0$.
\medskip

{\bf Step 4.} \emph{ Restricting the affine structure on $\Star(\omega)$ 
to $\partial B$.}
To complete the proof of (2), we compare the restriction of this chart
to $\Star(\omega)\cap \partial B$ with the corresponding chart
for $\partial B$ arising from the pair $(X_s,D_s)$. It is enough
to verify that
\begin{equation}
\label{eq:what we need to verify}
\psi_{\omega}(D_{k_0}^*)+\psi_{\omega}(D_{k_r}^*)
= - \sum_{j=1}^{n-2}\big( (D_{i_j}\cap X_s)\cdot (X_{\omega}\cap X_s)\big)
\psi_{\omega}(D_{i_j}^*),
\end{equation}
where the intersection number is calculated on $X_s$. 
We first note that 
\begin{equation}
\label{eq:big kink}
\psi_j(n_{\tau_0})+\psi_j(n_{\tau_r})= D_{i_j}\cdot (X_{\omega}\cap g^{-1}(0))
= D_{i_j}\cdot (X_{\omega}\cap X_s).
\end{equation}
Indeed, let $\sigma_{\ell,\ell+1}\in \P_{\omega}$ be the two-dimensional
cone with boundary rays $\tau_\ell$ and $\tau_{\ell+1}$, and let
$n_{\tau_\ell}=(\alpha_{\ell},\beta_{\ell})$. Then by the definition of kink
\eqref{def:kink},
for $1\le \ell \le r-1$,
\begin{equation}
\label{eq:kink of psij}
(d\psi_j|_{\sigma_{\ell,\ell+1}})(0,-1)+
(d\psi_j|_{\sigma_{\ell-1,\ell}})(0,1) = \alpha_{\ell} (D_{i_j}\cdot 
X_{\rho_{\ell}}).
\end{equation}
Since $\alpha_{\ell}$ is the multiplicity of $D_{k_{\ell}}$ in
$g^{-1}(0)$ by definition of $g_{\trop}$, summing \eqref{eq:kink of psij}
for $1\le \ell \le r-1$ gives \eqref{eq:big kink}.
Thus we obtain
\begin{align*}
\psi_{\omega}(D_{k_0}^*)+\psi_{\omega}(D_{k_r}^*) = {} & 
\left(n_{\tau_0}+ n_{\tau_r}, - \sum_{j=1}^{n-2}
\big(\psi_j(n_{\tau_0})+\psi_j(n_{\tau_r})\big) e_j\right)\\
\stackrel{\eqref{eq:big kink}}{=} {} & 
\left(0, - \sum_{j=1}^{n-2}\big((D_{i_j}\cap X_s)\cdot 
(X_{\omega}\cap D_s)\big) 
e_j\right).
\end{align*}
Noting that $\psi_{\omega}(D_{i_j}^*)=(0,e_j)$, this
gives \eqref{eq:what we need to verify}.

\medskip
{\bf Step 5.} \emph{Proof of (3).}
For (3), note that we may obtain such a single-valued representative
by constructing a single-valued piecewise linear function 
$\bar\varphi_{\omega}:B_{\omega}\rightarrow Q^{\gp}_{\RR}$ with kink
$[X_{\rho_{\ell}}]$ along $\tau_\ell$, just as we constructed
$\psi_j$ as above. Let $\pr_1:B_{\omega}\times \RR^{n-2}\rightarrow B_{\omega}$
be the first projection. We then obtain a single-valued function 
$\varphi_{\omega}=\bar\varphi_{\omega}\circ \pr_1\circ \psi_{\omega}$ on
$\Star(\omega)$ with the correct kinks along the $\rho_{\ell}$.
The kink of the restriction of $\varphi_{\omega}$ to
$\Star(\omega)\cap \partial B$ is then easily calculated
as $\bar\varphi_{\omega}(n_{\tau_0})+\bar\varphi_{\omega}(n_{\tau_r})
=[X_{\omega}\cap g^{-1}(0)]$ as before, and the latter coincides with
$\iota([X_{\omega}\cap X_s])$, as desired.
\end{proof}


\section{Punctured log maps and tropical geometry}
\label{sec:punctured log maps}

Throughout this section we will assume given $g:X\rightarrow S$
satisfying Assumptions~\ref{ass:absolute} and \ref{ass:relative}
in the absolute and relative cases without further comment.

\subsection{Review of notation}
We briefly review notation from \cite{ACGSI,ACGSII} for tropical maps
to $\Sigma(X)$ and punctured maps to $X$. We first summarize the tropical
language as developed in \cite[\S2.5]{ACGSI} and \cite[\S2.2]{ACGSII}.

In what follows, $\mathbf{Cones}$ denotes the category of 
rational polyhedral cones with integral structure, i.e., objects
are rational polyhedral cones $\omega \subseteq 
\Lambda_{\omega}\otimes_{\ZZ}\RR$ for
$\Lambda_{\omega}$
the lattice of integral tangent vectors\footnote{In 
\cite{ACGSI,ACGSII}, we write this lattice
as $N_{\omega}$; here we use the notation $\Lambda_{\omega}$ to fit the
notation of \cite{Theta}.} 
to $\omega$.
Morphisms are maps of cones induced by maps of the corresponding lattices.
We write $\omega_{\ZZ}=\omega\cap \Lambda_{\omega}$ for the set of integral
points of $\omega$.

We consider graphs $G$, with sets of vertices $V(G)$, edges $E(G)$ and legs
$L(G)$. In what follows, we will frequently confuse $G$ with its topological
realisation $|G|$. Legs will correspond to marked or punctured points of
punctured curves, and are rays in the marked case and compact line segments in
the punctured case. We view a compact leg as having only one vertex. An abstract
tropical curve over $\omega\in \mathbf{Cones}$ is data $(G, {\mathbf g}, \ell)$
where ${\mathbf g}:V(G)\rightarrow\NN$ is a genus function and
$\ell:E(G)\rightarrow \Hom(\omega_{\ZZ},\NN)\setminus\{0\}$ determines edge
lengths. Since this paper will only deal with genus $0$ curves, we omit
$\mathbf{g}$ entirely from the sequel. 

Associated to the data
$(G,\ell)$ is a generalized cone complex (a diagram in the
category of $\mathbf{Cones}$ with all morphisms being inclusions
of faces induced by a morphism of corresponding lattices) 
$\Gamma(G,\ell)$ along with a morphism of cone complexes
$\Gamma(G,\ell)\rightarrow \omega$ with fibre over $s\in\Int(\omega)$
being a tropical curve, i.e., a metric graph, with underlying graph
$G$ and affine edge length of $E\in E(G)$ being $\ell(E)(s)\in \RR_{\ge 0}$. 
Associated to each vertex $v\in V(G)$
of $G$ is a copy $\omega_v$ of $\omega$
in $\Gamma(G,\ell)$. Associated to each edge or leg $E\in E(G) \cup L(G)$
is a cone $\omega_E \in \Gamma(G,\ell)$ with
$\omega_E\subseteq \omega\times\RR_{\ge 0}$ and the map to $\omega$
given by projection onto the first coordinate. This projection
fibres $\omega_E$ in compact intervals or rays over $\omega$ 
(rays for legs representing a marked point). 

A \emph{family of tropical maps} 
to $\Sigma(X)$ over $\omega\in \mathbf{Cones}$ is a morphism of cone complexes
\[
h:\Gamma(G,\ell)\rightarrow \Sigma(X).
\]
If $s\in\Int(\omega)$, we may view $G$ as the fibre of
$\Gamma(G,\ell)\rightarrow \omega$ over $s$ as a metric graph, and
write
\[
h_s:G\rightarrow \Sigma(X)
\]
for the corresponding tropical map with domain $G$.
The \emph{type} of such a family consists of the data
$\tau:=(G,\bsigma,\mathbf{u})$ where
\[
\bsigma:V(G)\cup E(G)\cup L(G)\rightarrow \Sigma(X)
\]
associates to $x\in V(G)\cup E(G)\cup L(G)$ the minimal
cone of $\Sigma(X)$ containing $h(\omega_x)$. Further,
$\mathbf{u}$ associates to each (oriented) edge or leg $E\in E(G)\cup L(G)$
the corresponding \emph{contact order} $\mathbf{u}(E)\in \Lambda_{\bsigma(E)}$,
the image of the tangent vector $(0,1)\in \Lambda_{\omega_E}=
\Lambda_{\omega}\oplus\ZZ$ under the map $h$.

As we shall only consider tropicalizations of 
pre-stable punctured curves (see \cite[Def.\ 2.5]{ACGSII},
following \cite[Prop.\ 2.21]{ACGSII} we may assume that for $L\in L(G)$ 
with adjacent vertex $v\in V(G)$
giving $\omega_L,\omega_v\in \Gamma(G,\ell)$, we have
\begin{equation}
\label{eq:leg}
h(\omega_L)=(h(\omega_v)+\RR_{\ge 0}{\bf u}(L))\cap \bsigma(L)
\subseteq \Lambda_{\bsigma(L),\RR}.
\end{equation}

A \emph{decorated type} is data $\btau=(G,\bsigma,\mathbf{u},\mathbf{A})$
where $\mathbf{A}:V(G)\rightarrow H_2(X)$ associates a curve class to
each vertex of $G$. The \emph{total curve class} of ${\bf A}$
is $A=\sum_{v\in V(G)} {\bf A}(v)$.

We also have a notion of a contraction morphism of types
$\phi:\tau\rightarrow \tau'$, see \cite[Def.\ 2.24]{ACGSI}. This
is a contraction of edges on the underlying graphs, and the
additional data satisfies some relations as follows. 
If $x\in V(G)\cup E(G)\cup L(G)$, then $\bsigma'(\phi(x))
\subseteq \bsigma(x)$ (if $x$ is an edge, it may be contracted
to a vertex by $\phi$). Further, if $E\in E(G)\cup L(G)$ 
then $\mathbf{u}(E)=\mathbf{u}'(\phi(E))$ under the inclusion
$\Lambda_{\bsigma'(\phi(E))}\subseteq \Lambda_{\bsigma(E)}$, provided that $E$ is
not an edge contracted by $\phi$.

We also recall the notion of \emph{global contact order},
\cite[Def.\ 2.29]{ACGSII}. If $\sigma\in\Sigma(X)$, 
we set 
\[
\foC_{\sigma}(X):=\colim_{\sigma\subseteq\sigma'} \Lambda_{\sigma'},
\]
where the colimit is taken in the category of sets,
and define 
\[
\foC(X):=\coprod_{\sigma\in\Sigma(X)} \foC_{\sigma}(X).
\]
A \emph{global type of tropical or punctured map} is then 
$\bar\tau=(G,\bsigma,\overline{\mathbf{u}})$, where for each
$E\in E(G)\cup L(G)$, $\overline{\mathbf{u}}(E)\in \foC_{\bsigma(E)}(X)$.
Note that a type $\btau=(G,\bsigma,\mathbf{u})$
determines a global type by replacing $\mathbf{u}(E)$ with
its image under the natural map $\Lambda_{\bsigma(E)}\rightarrow
\foC_{\bsigma(E)}(X)$.

We say a type (or global type) $\tau$ is \emph{realizable} if there exists a 
family of tropical maps to $\Sigma(X)$ of type (or global type) $\tau$.
We also say $\btau=(\tau,{\bf A})$ is realizable if $\tau$ is
realizable. We note that if a global type $\bar\tau$ is realizable, there is
only one type $\tau$ of tropical map realizing it (see 
\cite[Lem.\ 3.5]{ACGSII}). Most of the time in this paper, we will only
deal with realizable types, and hence we will then not distinguish
between a type and a global type in this case.

If a type $\tau$ is realizable, then there is a universal family of
tropical maps of type $\tau$, parameterized by an object of
$\mathbf{Cones}$. Hopefully without confusion, we will generally write 
this cone as $\tau$. Hence we have a cone complex $\Gamma(G,\ell)$ equipped
with a map to $\tau$ and a map of cone complexes $h=h_{\tau}:\Gamma(G,\ell)
\rightarrow \Sigma(X)$. Generally we write $h$ rather than $h_{\tau}$
when unambiguous. Note that for each $x\in E(G)\cup L(G)\cup V(G)$, 
we thus obtain $\tau_x \in \Gamma(G,\ell)$ the corresponding cone.

\medskip
We write $\shX$ for the Artin fan of $X$, see
\cite{ACMW17}, as well as \cite[\S2.2]{ACGSI} for a summary.
We emphasize here that we always mean the absolute Artin fan,
denoted as $\shA_X$ in \cite{ACGSII}, even when we are working in the
relative situation.

We refer to \cite[Defs.\ 2.10, 2.13, 2.14]{ACGSII} for the notion of a family
$\pi:C^{\circ}\rightarrow W$ of punctured curves and pre-stable or 
stable punctured maps $f:C^{\circ}/W\rightarrow X$ or
$f:C^{\circ}/W\rightarrow \shX$. For the most part in this paper, we
work with punctured maps with target $X$ defined over $\Spec\kk$,
and in the relative case, only briefly work with punctured maps
defined over $S$.

Given a punctured map with domain $C^\circ\rightarrow W$ and
$W=\Spec(Q\rightarrow\kappa)$ for $\kappa$ an algebraically closed
field and target $X$ or $\shX$, we obtain by
functoriality of tropicalizations a family of tropical maps
\begin{equation}
\label{eq:trop diag}
\xymatrix@C=30pt
{
\Sigma(C)=\Gamma(G,\ell)\ar[d]\ar[r]& \Sigma(X)\\
\Sigma(W)=\omega=Q^{\vee}_{\RR} &
}
\end{equation}
parameterized by $W$.
The \emph{type} of the punctured map is then the type
$\tau=(G,\bsigma,\mathbf{u})$
of this family of tropical maps. We recall that the punctured
map $f:C^{\circ}/W\rightarrow X$ is \emph{basic} if \eqref{eq:trop diag}
is the universal family of tropical maps of type $\tau$.

Given a decorated global type $\bar\btau$, \cite[Def.\ 3.7]{ACGSII}
defines the notion of a marking of a punctured map by $\bar\btau$.
In particular, this gives rise to moduli spaces
$\scrM(X,\bar\btau)$ (resp.\ $\foM(\shX,\bar\btau)$) 
of stable (resp.\ pre-stable) $\bar\btau$-marked punctured maps 
to $X$ (resp.\ $\shX$). When we are in the relative situation and
wish to work with punctured maps to $X$ defined over $S$, then
we write $\scrM(X/S,\bar\btau)$ for the corresponding moduli space.
We note that while curve classes in
$\shX$ are meaningless, the decoration on $\bar\btau$ affects the
notion of isomorphism in the category $\foM(\shX,\bar\btau)$.
By \cite[Thm.\ 3.10]{ACGSII},
these are algebraic stacks, $\scrM(X,\bar\btau)$ is Deligne-Mumford
and both stacks carry natural log structures, namely the basic
log structure.
Further, there is a natural morphism 
\begin{equation}
\label{def:varepsilon}
\varepsilon:\scrM(X,\bar\btau) \rightarrow \foM(\shX,\bar\btau)
\end{equation}
given by composing a punctured map $C^{\circ}\rightarrow X$ with the
canonical map $X\rightarrow\shX$. \cite[\S4]{ACGSII} then gives a perfect
relative obstruction theory for $\varepsilon$. Most importantly, while the
stacks $\foM(\shX,\bar\btau)$ may be quite poorly behaved globally, in fact they
have a simple local structure coming from the fact that they are idealized log
smooth over $S$, see \cite[Thm.\ 3.24, Rem.\ 3.25]{ACGSII}.

We sometimes work with moduli spaces $\foM(\shX,\bar\tau)$ of
pre-stable $\bar\tau$-marked punctured maps to $\shX$, forgetting the
decorations. The forgetful morphism $\foM(\shX,\bar\btau)\rightarrow
\foM(\shX,\bar\tau)$ is strict \'etale, and we may also write
$\varepsilon:\scrM(X,\bar\btau)\rightarrow \foM(\shX,\bar\tau)$
for the composition of \eqref{def:varepsilon} with this forgetful
morphism.

If there is a contraction morphism between decorated global types
$\phi:\bar\btau\rightarrow \bar\btau'$, we obtain a forgetful
map $\scrM(X,\bar\btau)\rightarrow \scrM(X,\bar\btau')$.
This gives rise to a stratified description of these moduli spaces,
see \cite[Rem.\ 3.29]{ACGSII}.

There is one special choice of decorated type which imposes the minimal
number of conditions on the curve. We always write $\beta$
for a \emph{class of punctured map}, which is a type where the
underlying graph $G$ has only one vertex and no edges, and
$\bsigma$ takes the vertex and every leg to $\{0\}\in\Sigma(X)$.
Put another way, a class of punctured map is the data 
$\beta=(\overline{\mathbf{u}},A)$
where $\overline{\mathbf{u}}$ is a collection of
global contact orders $\{\bar u_i\}\subseteq \foC_0(X)$ and $A\in H_2(X)$
is a curve class. Given a decorated global type $\bar\btau$,
there is always a unique choice of punctured curve class $\beta$
with a contraction morphism $\bar\btau\rightarrow \beta$.
Here $A$ is the total curve class of ${\bf A}$ if 
$\bar\btau=(G,\bsigma,\bar{\mathbf{u}}, \mathbf{A})$. In this case
we say the decorated type $\bar\btau$ is \emph{of class} $\beta$, and there
is a canonical morphism $\scrM(X,\bar\btau)\rightarrow \scrM(X,\beta)$,
and an analogous morphism for maps to Artin fans. The moduli
space $\scrM(X,\beta)$ is the moduli space of punctured maps where
only the curve class and the contact orders of the marked and punctured
points are fixed.

Suppose now given $\tau$ a realizable type equipped with a contraction morphism
$\tau\rightarrow \tau'$ to a (not necessarily realizable) global type $\tau'$,
and let $\btau'$ be a decoration of $\tau'$. As mentioned earlier, we do not
distinguish between the type $\tau$ and its corresponding global type. We denote
by $\foM_{\tau}(\shX,\tau')$ the reduced closed substack of $\foM(\shX,\tau')$
whose underlying closed subspace is the closure of the locus of geometric points
whose corresponding punctured map is of type $\tau$. We may then define
\[
\scrM_{\tau}(X,\btau'):=\scrM(X,\btau')\times_{\foM(\shX,\tau')}
\foM_{\tau}(\shX,\tau').
\]
While $\foM_{\tau}(\shX,\tau')$, assuming non-empty, always carries points
corresponding to curves of type $\tau$, this may not be the case for
$\scrM_{\tau}(X,\btau')$. Nevertheless, the type $\tau''$ of any curve over a
geometric point in $\foM_{\tau}(\shX,\tau')$ always has a (not necessarily
unique) contraction map $\tau''\rightarrow \tau$. We note that provided
$\foM_{\tau}(\shX,\tau')$ is non-empty, there is then a canonical morphism
\[
\foM(\shX,\tau)\rightarrow \foM_{\tau}(\shX,\tau')
\]
of degree $|\Aut(\tau/\tau')|$. Here $\Aut(\tau/\tau')$ denotes the
group of automorphisms of the global type $\tau$ compatible
with the contraction map $\tau\rightarrow\tau'$, i.e., 
automorphisms of the underlying graph $G$ preserving
$\bsigma$ and $\bar{\bf u}$ and the contraction map.

If $\bar\btau=(G,\bsigma,\bar{\bf u},\mathbf{A})$ denotes a 
choice of global type, and 
$I\subseteq
E(G)\cup L(G)$ is a collection of edges and legs, then we write
\[
\foM^{\ev}(\shX,\bar\btau)=\foM^{\ev(I)}(\shX,\bar\btau)
:= \foM(\shX,\bar\btau)\times_{\ul{\shX}^I} \ul{X}^I.
\]
Here $\ul{\shX}^I$ denotes the product of $\#I$ copies of $\ul{\shX}$
over $\ul{S}$, and similarly $\ul{X}^I$; the morphism
$\foM(\shX,\bar\btau)\rightarrow \ul{\shX}^I$ is given by evaluation
at the nodes and punctured points indexed by elements of $I$,
and $\ul{X}^I\rightarrow \ul{\shX}^I$ is induced by the canonical smooth
map $\ul{X}\rightarrow\ul{\shX}$. The map $\varepsilon$ then factors
as 
\[
\xymatrix@C=30pt
{
\scrM(X,\bar\btau)\ar[r]^{\varepsilon^{\ev}} & \foM^{\ev}(\shX,\bar\btau)
\ar[r] & \foM(\shX,\bar\btau).
}
\]
The second morphism is smooth, while $\varepsilon^{\ev}$ also possesses
a relative obstruction theory compatible with the morphism
$\varepsilon$ of \eqref{def:varepsilon}, see 
\cite[\S4.2]{ACGSII}.

\subsection{Some tropical lemmas}

We first observe the following result, which explains why the choice
of affine structure on $B$ given in \S\ref{subsec:the affine structure}
is the correct one. 

\begin{lemma}
\label{lem:balancing}
Let $f:C^{\circ}/W\rightarrow X$ be a stable punctured map to $X$,
with $W=\Spec(Q\rightarrow\kappa)$ a geometric log point. For
$s\in \Int(Q^{\vee}_{\RR})$, let $h_s:G\rightarrow \Sigma(X)$ be
the corresponding tropical map. If $v\in V(G)$ satisfies
$h_s(v)\in B_0$, then $h_s$ satisfies the balancing condition
at $v$. More precisely, if $E_1,\ldots,E_n$ are the legs
and edges adjacent to $v$, oriented away from $v$,
then the contact orders $\mathbf{u}(E_i)$ may be interpreted
as elements of $\Lambda_{h_s(v)}$, the stalk of the
local system $\Lambda$ of Definition~\ref{def:Lambda} at
$h_s(v)$. In this group, the balancing condition
\begin{equation}
\label{eq:balancing}
\sum_{i=1}^m \mathbf{u}(E_i)=0
\end{equation}
is satisfied.
Further, in the relative case, the composition $\Sigma(g)\circ h_s:
G\rightarrow \Sigma(S)=\RR_{\ge 0}$ is a balanced tropical map.
\end{lemma}

\begin{proof}
By definition of $\Delta$, $h_s(v)\in B_0$ implies
that $h_s(v)\in\Int(\sigma)\in\P$ for $\sigma$ of codimension $0$
or codimension $1$. If $\sigma\in\P_{\max}$, then
necessarily $\bsigma(E_i)=\sigma$ and $\mathbf{u}(E_i)\in \Lambda_{\sigma}$ is
a tangent vector to $\sigma$, and hence can be viewed
as an element of $\Lambda_{h_s(v)}$. Necessarily
$f$ contracts the corresponding component $C_v$ of $C$
to the zero-dimensional stratum $X_{\sigma}$, and hence
the balancing condition holds from \cite[Prop.\ 2.25]{ACGSII}, 
keeping in mind that $\tau_x$ given in the statement of that proposition
vanishes because $C_v$ is contracted.

If instead $\sigma\in\P^{[n-1]}$, note that by Assumptions~\ref{ass:absolute},
(2), $\sigma$ is not contained in any cone
of $\Sigma(X)$ which is not also in $\P$. 
Noting also that as $\bsigma(E_i)$ necessarily contains $\sigma$,
we must have $\bsigma(E_i)\in \P$. In particular,
$\mathbf{u}(E_i)\in \Lambda_{\bsigma(E_i)}$
may then be viewed, via parallel transport in $\Lambda$, as a tangent
vector in $\Lambda_{h_s(v)}$.  

We may now split the punctured map $f$ at all nodes of
$C$ contained in $C_v$, and obtain by restriction a punctured
log map $f_v:C^{\circ}_v\rightarrow X$ by \cite[Prop.\ 5.2]{ACGSII}.
Note that the dual graph of $C^{\circ}_v$ consists of a single vertex 
with legs $E_1,\ldots,E_m$. Further, the contact order of $f_v$
at the puncture corresponding to $E_i$ agrees with the
contact order $\mathbf{u}(E_i)$ for the map $f$ of the corresponding edge or 
leg of $G$, oriented away from $v$. 
It is thus sufficient to show balancing for the
tropical map induced by $f_v$.

Note $f_v$ factors through the strict
morphism $X_{\sigma}\hookrightarrow X$. Moreover, by Lemma 
\ref{lem:toric dim one}, $X_{\sigma}$ is isomorphic to a stratum
of the toric variety $X_{\Sigma_{\sigma}}$, and hence we obtain
a punctured map $f_v:C_v^{\circ}\rightarrow X_{\Sigma_{\sigma}}$.
However, tropicalizations of punctured maps to toric varieties
are always balanced when viewed as maps to the corresponding
fans, see \cite[Rem.\ 2.26]{ACGSII}. 
The claimed balancing then follows from the construction of the
affine chart $\psi_{\sigma}$.

Finally, in the relative case, necessarily the underlying map
$\ul{g}\circ\ul{f}$ is constant on irreducible components of $C$
since $S$ is assumed to be affine. Thus balancing of
$\Sigma(g)\circ h_s$ holds again from \cite[Prop.\ 2.25]{ACGSII}.
\end{proof}

Note that if $\tau=(G,\bsigma,\mathbf{u})$ is a type of tropical map
to $\Sigma(X)$, the balancing condition of Lemma~\ref{lem:balancing}
still makes sense at vertices mapping to $B_0$. Indeed, for $v\in V(G)$ with $\bsigma(v)$ a codimension
zero or one cone of $\P$, we may choose any point $x\in \Int(\bsigma(v))$
and \eqref{eq:balancing} makes sense inside $\Lambda_x$. Further, in the
relative case, by composing with $\Sigma(g)$, $\tau$ yields a type of 
tropical map to $\Sigma(S)=\RR_{\ge 0}$, and again it makes sense
to ask that such a type be balanced.

Thus we define:

\begin{definition}
\label{def:balanced type}
A type $\tau$ of tropical map to $\Sigma(X)$ is \emph{balanced} if:
\begin{enumerate}
\item
For each $v\in V(G)$ with $\bsigma(v)\in\P$ a codimension zero or one
cone, the balancing condition \eqref{eq:balancing} holds at $v$.
\item In the relative case, $\tau$ induces a type of a balanced tropical
map to $\Sigma(S)$.
\end{enumerate}
\end{definition}

The following observation shows that certain degree data of maps are
determined by tropical data:

\begin{lemma}
\label{lem:admissible degrees}
Let $\btau=(\tau,\mathbf{A})$ be a decorated type,
and suppose $\scrM(X,\btau)$ is non-empty. Then:
\begin{enumerate}
\item For each $v\in V(G)$
with $\bsigma(v)\in\P_{\max}$, we have
$\mathbf{A}(v)=0$.
\item Assume $v\in V(G)$ with 
$\rho=\bsigma(v)\in \P^{[n-1]}_{\inte}$
contained in $\sigma\in\P_{\max}$. Let $E_1,\ldots,E_r$ be the edges
adjacent to $v$ with $\bsigma(E_i)=\sigma$ for $1\le i\le r$. Let
$\delta:\Lambda_{\sigma}\rightarrow \ZZ$ be the quotient map of
$\Lambda_{\sigma}$ by $\Lambda_{\rho}$ positive on tangent vectors pointing from
$\rho$ into $\sigma$. Then $\mathbf{A}(v)=d [X_{\rho}]$, where 
\[
d=\sum_{i=1}^r \delta(\mathbf{u}(E_i)).
\]
We note that this number is independent of the choice of $\sigma$ containing
$\rho$ by the balancing condition.
\end{enumerate}
\end{lemma}

\begin{proof}
By assumption there exists a punctured map 
$f:C^{\circ}/W\rightarrow X$ over
a geometric log point $W$ marked by $\btau$. Thus for
a vertex $v\in V(G)$, we obtain via splitting as in the proof of
Lemma~\ref{lem:balancing} a subcurve $C^{\circ}_v$ and a punctured
map $f_v:C^{\circ}_v\rightarrow X$.

In case (1), $\mathbf{A}(v)=0$ is obvious because $f_v$ maps
$C_{v}^{\circ}$ to a zero-dimensional stratum.

In case (2), ${\bf A}(v)$ must be a multiple of the class $[X_{\rho}]$.
If $\RR_{\ge 0} D_i^*$ is the ray of $\sigma$ not contained in
$\rho$, then this multiple may be 
determined by intersecting the curve class
of $f_v$ with the divisor $D_i$. The claim now 
follows immediately from 
\cite[Cor.\ 1.14]{Assoc}.
\end{proof}

\begin{definition}
\label{def:spine}
Let $G$ be a graph of genus $0$.
The \emph{spine} of $G$ is the smallest connected subgraph $G'\subseteq G$
containing all legs of $G$.
\end{definition}

The following is the key tropical argument of the paper. In item (1)
we consider the type of tropical map which will contribute to wall
structures. These are tropical maps where the domain has only one
(non-contracted) leg. Item (1) shows that this leg can
sweep out at most a codimension one polyhedral cone: in the
case of codimension
one these will play the role of a wall in the canonical wall structure.
Item (2), on the other hand, considers types which will correspond
to broken lines with respect to the canonical wall structure, and
the main point of (2) is to show that the spine of $G$, in this case,
is mapped into $B_0\subseteq |\Sigma(X)|$. This will be key for the
logarithmic/broken line correspondence theorem, 
Theorem~\ref{thm:main correspondence theorem}.

\begin{lemma}
\label{lem:key tropical lemma}
Fix a balanced realizable type 
$\tau=(G, \bsigma, \mathbf{u})$ of a genus zero tropical map to
$\Sigma(X)$ with a distinguished leg $L_{\out}\in L(G)$ and
$\mathbf{u}(L_{\out})\not=0$. Let $h:\Gamma(G,\ell) \rightarrow \Sigma(X)$ 
be the corresponding universal
family of tropical maps, defined over the cone $\tau$.
For $s\in \tau$,
we write $h_s:G\rightarrow \Sigma(X)$ for the induced map.
\begin{enumerate}
\item Suppose $G$ has only one leg, $L_{\out}$, and let $\tau_{\out}
\in \Gamma(G,\ell)$ be the corresponding cone. Suppose further that
$\bsigma(L_{\out})\in\P$. Then $\dim h(\tau_{\out})\le n-1=\dim X-1$.
\item Suppose $G$ has precisely two legs, $L_{\inc}$ and $L_{\out}$, with 
$\bsigma(L_{\inc}),\bsigma(L_{\out})\in \P$. Suppose further that
$\dim\tau=n-1$ and $\dim h(\tau_{\out})=n$.
Then with $G'$ the spine of $G$ and
$s\in\Int(\tau)$, $h_s(G')\subseteq B\subseteq |\Sigma(X)|$, and
$h_s(G')$ only intersects codimension zero and one cones of $\P$, except
possibly for the non-vertex endpoints of $L_{\inc}$ and $L_{\out}$.
Further, $\dim h(\tau_v)=n-1$ for every vertex $v$ of the spine $G'$.
\end{enumerate}
\end{lemma}

\begin{proof}
We prove both items simultaneously. We may assume in the first
case that to the contrary $\dim h(\tau_{\out})=n$, and thus
in both cases we have $\dim h(\tau_{\out})=n$.
Further, in the first case we may also assume that $\dim \tau=n-1$.
Indeed, if $\dim \tau>n-1$, there must be an $(n-1)$-dimensional
face $\tau'$ of $\tau$, necessarily corresponding to a type
$\tau'$ which is a contraction of $\tau$, such that $\dim h(\tau'_{\out})
=n$. Thus we may replace $\tau$ by $\tau'$ in this case to 
ultimately achieve a contradiction.

Write $\sigma_v:=h(\tau_v)$.
As $\dim\tau=n-1$, $\dim \sigma_v\le n-1$.

We shall inductively find a sequence of distinct edges and legs
$L_{\out}=E_1,\ldots,E_p$ and vertices $v_1,\ldots,v_{p-1}$ of
$G$ such that $E_1=L_{\out}$ and $v_i$ is a vertex of $E_i$ and
$E_{i+1}$. Further, this sequence will satisfy the following
inductive properties, for $s\in \Int(\tau)$:
\begin{enumerate}
\item The images of the edges $h_s(E_i)$ are all contained in $B$
and only intersect codimension $0$ and $1$ cones of $\Sigma(X)$,
except possibly for the non-vertex endpoint of $L_{\out}$ and,
if $E_p=L_{\inc}$, also the non-vertex endpoint of $L_{\inc}$.
\item $\mathbf{u}(E_i)$ is not tangent to $\sigma_{v_i}$.
\item $\dim\sigma_{v_i}=n-1$. 
\end{enumerate}
We will be able to continue the induction provided $E_p$ is an edge.
Thus in the first case of the lemma, the induction process would continue
forever, a contradiction. In the second case, eventually $E_p=L_{\inc}$,
and $G'$ is the union of $E_1,\ldots,E_p$, giving the desired result.

For the base case, we take $E_1=L_{\out}$, $v_1$ the unique vertex of
$E_1$. Note from \eqref{eq:leg} that 
\begin{equation}
\label{eq:h tau out}
h(\tau_{\out})=(\sigma_{v_1}+\RR_{\ge 0}\mathbf{u}(L_{\out}))\cap
\bsigma(L_{\out}).
\end{equation}
As $\dim h(\tau_{\out})=n$ by assumption, necessarily $\bsigma(L_{\out})
\in\P_{\max}$ and $\dim\sigma_{v_1}= n-1$.
Thus $\bsigma(v_1)\in\P$ is either codimension $0$ or $1$, and in any
case $\sigma_{v_1}$ intersects the interior of $\bsigma(v_1)$. 
Further, $\bsigma(v_1)$ is a face of $\bsigma(L_{\out})$.
It is then clear that for $s\in \Int(\tau)$,
$h_s(v_1)\in\Int(\sigma_{v_1})$ and 
$h_s(L_{\out})$ is contained in $B$ and
only intersects codimension zero and one cones of $\Sigma(X)$ (except
possibly for the non-vertex endpoint of $L_{\out}$).
Further, if $\mathbf{u}(L_{\out})$ were tangent to $\sigma_{v_1}$,
then \eqref{eq:h tau out} implies $\dim h(\tau_{\out})=n-1$, a contradiction.
Thus inductive conditions (1)--(3) are satisfied.

We next observe that given a sequence of edges
$L_{\out}=E_1,\ldots,E_p$ and vertices $v_i$ of $E_i$ and $E_{i+1}$,
item (2) for $2 \le i < p$ holds regardless of the details of
the construction  of this sequence of edges. Indeed, suppose that
$\mathbf{u}(E_i)$ is tangent to $\sigma_{v_i}$. Split the graph
$G$ by detaching $E_i$ from $v_i$ to obtain two connected components
$G_1, G_2$ such that $E_i$ is a leg of $G_1$. Hence $E_1$ is also
contained in $G_1$. Now, with $s\in\Int(\tau)$, necessarily
$h_s(v_i)+\epsilon \mathbf{u}(E_i)\in \sigma_{v_i}$ for
$\epsilon$ such that $|\epsilon|$ is sufficiently small. Thus
there is an $s(\epsilon)\in \Int(\tau)$, depending on $\epsilon$, with
the property that $h_{s(\epsilon)}(v_i)=h_s(v_i)+\epsilon
\mathbf{u}(E_i)$. In particular, by changing the length of the
edge $E_i$, we may glue $h_s|_{G_1}$ to $h_{s(\epsilon)}|_{G_2}$
to obtain a tropical map $h_{s'}:G\rightarrow \Sigma(X)$ of type $\tau$ 
which does not coincide
with $h_s$, but for which $h_{s'}(E_1)=h_s(E_1)$.
Since $h_s(E_1)$ already varies in an $(n-1)$-dimensional family,
necessarily $\dim\tau$ must be at least
$n$, contradicting the assumption on the dimension.

Now assume further that $E_1,\ldots,E_p$ satisfy the inductive
conditions (1)--(3). We wish to construct $E_{p+1}$. By assumption
(1),
$\sigma_{v_p}$ is contained in a codimension one or codimension zero
cone $\sigma=\bsigma(v_p)$ of $\P$, and thus it follows that for $s\in \Int(\tau)$, 
$h_s(v_p)\in B_0$. As $\tau$ is a balanced type, the balancing
condition thus holds at $v_p$. Since $\mathbf{u}(E_i)$ is
not tangent to $\sigma_{v_p}$, there must be at least one other
edge $E$ adjacent to $v_p$ with $\mathbf{u}(E)$ not tangent to
$\sigma_{v_p}$ in order for balancing to hold. Choose one such edge
to be $E_{p+1}$, and let $v_{p+1}$ be the other vertex of $E_{p+1}$
if $E_{p+1}$ is an edge. 

We check the inductive conditions (1) and (3). For (1), note
that as $s\in\tau$ varies, $h_s(E_{p+1})$ varies in an $(n-1)$-dimensional
family, and as $h_s(E_{p+1})$ is not tangent to $\sigma_{v_p}$
by choice of $E_{p+1}$, $h_s(E_{p+1})$ fills out an $n$-dimensional 
subcone $\tau'$ of some
$\sigma'\in\Sigma(X)$ containing $\sigma$. By Assumptions~\ref{ass:absolute},(2),
$\sigma'\in\P$. Further, for any $s\in \Int(\tau)$,
$h_s(E_{p+1})$ may only intersect faces of $\sigma'$ of codimension at most
one, except possibly the non-vertex endpoint of
$E_{p+1}$ if $E_{p+1}=L_{\inc}$. Indeed, this is clear from 
\eqref{eq:leg} in the latter case.
Otherwise, $E_{p+1}$ has another vertex $v_{p+1}$, and if
$h_s(E_{p+1})$ meets a face $\sigma''\subseteq\sigma'$ of codimension
at least $2$, then $h_s(v_{p+1})\in\sigma''$ and $\bsigma(v_{p+1})
\subseteq\sigma''$. However, then $h_s(E_{p+1})$ can't vary in an
$(n-1)$-dimensional family. Thus inductive condition (1) follows,
and (3) is also clear for $\sigma_{v_{p+1}}$.
\end{proof}

\begin{remark}
\label{rem:spine unique}
In fact, the proof of the lemma tells us a bit more. In case (2),
with notation as in the proof, all edges of $G$ adjacent to $v_i$
except for $E_i$ and $E_{i+1}$ are tangent to $\sigma_{v_i}$. 
Indeed, if not, there would be a choice for the edge $E_{i+1}$
and the sequence of edges would not be unique. However, the edges
$E_1,\ldots,E_p$ must be the unique sequence of edges from
$L_{\out}$ to $L_{\inc}$ in the spine $G'$.
\end{remark}

\section{The canonical wall structure and logarithmic broken lines}
\label{sec:wall structures and log broken lines}

We continue with a pair $(X,D)$ satisfying Assumptions~\ref{ass:absolute}
or \ref{ass:relative} in the absolute and relative cases. We recall
also we have fixed data:
\begin{enumerate}
\item A group of degree data $H_2(X)$.
\item A saturated finitely generated monoid $Q\subseteq H_2(X)$ such that
$Q\cap (-Q)=H_2(X)_{\tors}$ which contains the classes of all
effective curves on $X$. We write the monomial maximal ideal of $Q$
\[
\fom:=Q\setminus Q^{\times}.
\]
\end{enumerate}
Throughout this section, we also fix a monoid ideal 
$I\subseteq Q$ such that $\sqrt{I}=\fom$. Equivalently, we require
$Q\setminus I$ to be finite.

\subsection{Recall of monomials on $B$}
\label{subsec:recall of monomials}
We have constructed $(B,\P,\varphi)$ 
with $\varphi$ a $Q^{\gp}_{\RR}$-valued 
MPL function as given in Construction~\ref{const:phi canonical}.
This choice of function then yields a local system $\shP$
\cite[Def.\ 1.15]{Theta} on $B_0=B\setminus\Delta$
fitting into an exact sequence
\[
0\longrightarrow \ul{Q}^{\gp} \longrightarrow \shP \longrightarrow \Lambda
\longrightarrow 0,
\]
where $\ul{Q}^{\gp}$ denotes the constant sheaf with stalk $Q^{\gp}$
on $B_0$.\footnote{In the case that $Q^{\gp}$ has torsion, 
\cite[Def.\ 1.15]{Theta} is not suitable. Rather, one may define $\shP$
using the explicit description of parallel transport given below
using the given kinks $\kappa_{\rho}\in Q^{\gp}$. In the formalism
of \cite{Theta}, such torsion may be accommodated by incorporating it
into the parameter ring $A$, see \cite[Rem.~5.17]{Theta}. However,
we do not need such involved notions here.}
We write the map $\shP_x\rightarrow \Lambda_x$ as 
$p\mapsto \bar p$. 
Further, for each $x\in B_0$, \cite[Def.\ 1.16]{Theta} gives a
submonoid $\shP^+_x\subseteq \shP_x$ of exponents of monomials
defined at $x$.

For our purposes, rather than reviewing the definition of
$\shP$, it is easier to give explicit descriptions of the monoids
$\shP^+_x$ and the effects of parallel transport under these
descriptions. 

For $\sigma\in\P_{\max}$, $x\in\Int(\sigma)$, we have
\begin{equation}
\label{eq:P+x interior}
\shP^+_x= \Lambda_x \times Q.
\end{equation}

For $\rho\in\P^{[n-1]}_{\partial}$, $x\in\Int(\rho)$, we have
\begin{equation}
\label{eq:P+x boundary}
\shP^+_x=\Lambda_{\rho\sigma}\times Q 
\end{equation}
where $\Lambda_{\rho\sigma}$ is the monoid of integral tangent
vectors contained in the tangent wedge $T_{\rho}\sigma$ of $\sigma$
along the face $\rho$.
If $\sigma\in\P_{\max}$ contains $\rho$, parallel transport in the
local system $\shP$ from $x$ to $y\in\Int(\sigma)$ induces the inclusion
$\shP^+_x\hookrightarrow \shP^+_y$ given by $(\lambda,q)\mapsto (\lambda,q)$.

For $\rho\in\P^{[n-1]}_{\inte}$, $x\in \Int(\rho)$, we have
\begin{equation}
\label{eq:P+x alternate}
\shP^+_x=(\Lambda_{\rho}\oplus \NN Z_+\oplus \NN Z_-\oplus Q)/
\langle Z_++Z_-=\kappa_{\rho}\rangle.
\end{equation}
This abstract description requires an ordering
$\sigma,\sigma'\in\P_{\max}$ of the maximal cells containing
$\rho$ and a choice of vector $\xi\in\Lambda_x$ pointing into
$\sigma$ and representing a generator of $\Lambda_{\sigma}/\Lambda_{\rho}$.
Then for $y\in \Int(\sigma)$, $y'\in\Int(\sigma')$, parallel transport
in the local system $\shP$ yields inclusions 
\begin{align}
\label{eq:transport1}
\begin{split}
\fot_{\rho\sigma}:\shP_x^+\hookrightarrow  {} & \shP_y^+\\
(\lambda_{\rho},a Z_+,b Z_-,q)\mapsto {} & 
\big(\lambda_{\rho}+(a-b)\xi, q + b \kappa_{\rho})
\end{split}
\end{align}
and
\begin{align}
\label{eq:transport2}
\begin{split}
\fot_{\rho\sigma'}:\shP_x^+\hookrightarrow  {} & \shP_{y'}^+\\
(\lambda_{\rho},a Z_+,b Z_-,q)\mapsto  {} & 
\big(\lambda_{\rho}+(a-b)\xi, q + a \kappa_{\rho})
\end{split}
\end{align}
respectively. See the discussion of \cite[\S2.2]{Theta}.

\medskip

Given a choice of monoid ideal $I\subseteq Q$, we also introduce
the monoid ideal $I_x\subseteq \shP^+_x$ when $x\in \Int(\sigma)$, $\sigma
\in\P_{\max}$ defined in the description \eqref{eq:P+x interior} as
\begin{equation}
\label{eq:Ix def}
I_x:= \Lambda_x \times I.
\end{equation}

\begin{notation}
\label{not:monomial notation}
For $x\in\Int(\sigma)$, $\sigma\in\P_{\max}$, we will often write a monomial
in $\kk[\shP^+_x]=\kk[Q][\Lambda_x]$ either as $z^m$ for $m\in\shP^+_x$
or as $t^qz^{\bar m}$
for $(\bar m, q)\in\Lambda_x\oplus Q$ via \eqref{eq:P+x interior}.

Similarly, if $x\in\Int(\rho)$ with $\rho\in \scrP^{[n-1]}$, 
we have a canonically defined submonoid
$\Lambda_{\rho}\oplus Q\subseteq \shP^+_x$ via
\eqref{eq:P+x boundary} or
\eqref{eq:P+x alternate}, hence defining a subring
$\kk[Q][\Lambda_{\rho}]\subseteq \kk[\shP^+_x]$. We again write
monomials in this ring as $t^qz^{\bar m}$ for $q\in Q, \bar m
\in\Lambda_{\rho}$. 
\end{notation}

We will frequently need to use parallel transport in
$\shP^+$ from a cell $\sigma\in\P_{\max}$ to a cell
$\sigma'\in\P_{\max}$, either with $\sigma=\sigma'$ or
$\sigma\cap\sigma'=\rho\in \P^{[n-1]}$. Take any $x\in \Int(\sigma)$,
$x'\in \Int(\sigma')$. If $\sigma=\sigma'$, we may define
\[
\fot_{\sigma,\sigma'}:\shP^+_x\rightarrow \shP^+_{x'}
\]
to be given by parallel transport along a path contained in $\Int(\sigma)$;
under the representation \eqref{eq:P+x interior}, this map is the identity.
On the other hand, if $\sigma\cap\sigma'\in \P^{[n-1]}$,
parallel transport along a path contained in $\Star(\rho)$ gives a map
\[
\fot_{\sigma,\sigma'}:\shP^+_x\rightarrow \shP_{x'}.
\]
Noting that at the level of groups, $\fot_{\sigma,\sigma'}
=\fot_{\rho\sigma'}\circ\fot_{\rho\sigma}^{-1}$, it follows that if 
$m\in \shP^+_x$ and $\bar m \in T_\rho\sigma$, then
$\fot_{\sigma,\sigma'}(m)\in \shP^+_{x'}$.
Thus, either in this case or the case $\sigma=\sigma'$,
we may write $\fot_{\sigma,\sigma'}(m)\in \shP^+_{x'}$ whenever
\[
m\in \shP^+_x, \quad \bar m \in T_{\sigma\cap\sigma'}\sigma.
\]

If we let $R'_I$ denote the subring of
$(\kk[Q]/I)[\Lambda_{\sigma}]$ generated by monomials 
of the form $t^q z^{\bar m}$ for $\bar m\in T_{\sigma\cap\sigma'}\sigma$, we
obtain a ring homomorphism
\begin{equation}
\label{eq:fot ring hom}
\fot_{\sigma,\sigma'}:R'_I \rightarrow (\kk[Q]/I)[\Lambda_{\sigma'}]
\end{equation}

\subsection{The canonical wall structure}
\subsubsection{Recall of wall structures}
We recall the notion of \emph{walls} and \emph{wall structures} 
from \cite[Def.~2.11]{Theta}:

\begin{definition}
A \emph{wall} on $(B,\P)$ is a codimension one rational polyhedral subset
$\fop\not\subseteq\partial B$ of some $\sigma\in\P_{\max}$,
along with an element
\[
f_{\fop}=\sum_{m\in \shP^+_x,\bar m\in\Lambda_{\fop}}c_m z^m\in \kk[\shP^+_x],
\]
for $x\in\Int(\fop)$. Identifying $\shP_y$ with $\shP_x$ by
parallel transport inside $\sigma\setminus\Delta$,
we require that $m\in \shP^+_y$
for all $y\in \fop\setminus\Delta$ when $c_m\not=0$. 
We further require that $f_{\fop}\equiv 1\mod \fom$.
\end{definition}

\begin{definition}
A \emph{wall structure} $\scrS$ on $(B,\P)$ is a finite set of walls.
\end{definition}

\begin{remark}
\label{rem:differences}
The above definitions differ in a couple of ways from that of
\cite[Def.~2.11]{Theta}. First, we are less permissive with wall
functions, insisting that $f_{\fop}\equiv 1\mod\fom$. Second,
we are more permissive with the notion of wall structure. In
\cite{Theta}, we insist walls form the codimension one cells of
a rational polyhedral decomposition of $B$ refining $\P$. This
is imposed there to make it easier to describe gluing. However,
a wall structure in the above more liberal sense is equivalent
(in the sense of Definition~\ref{def:wall structure stuff} below)
to one in the sense of \cite[Def.~2.11]{Theta},
and we ignore this issue in this section and the next, 
only returning to the convention of \cite{Theta} in \S\ref{sec:main theorem}.
\end{remark}

\begin{definition}
\label{def:wall structure stuff}
For a wall structure $\scrS$, we define
\begin{align*}
|\scrS| := {} & \bigcup_{\fop\in \scrS} \fop\cup \bigcup_{\rho\in
\P^{[n-1]}} \rho,\\
\Sing(\scrS) := {} & \Delta\cup \bigcup_{\fop\in\scrS} \partial\fop
\cup \bigcup_{\fop,\fop'\in \scrS}(\fop\cap \fop')
\end{align*}
where the last union is over all pairs of walls $\fop,\fop'$ with
$\fop\cap\fop'$ codimension at least two.

If $x\in B\setminus\Sing(\scrS)$, we define
\begin{equation}
\label{eq:fx def}
f_x:= \prod_{x\in\fop\in \scrS} f_{\fop}.
\end{equation}
We say two wall structures are \emph{equivalent 
(modulo $I$)} if
$f_x=f'_x\mod I$ for all $x\in B\setminus (\Sing(\scrS)\cup \Sing(\scrS'))$.
Generally we omit
mention of $I$ if clear from context.
\end{definition}

\subsubsection{The construction}

\begin{definition}
\label{def:wall type}
A \emph{wall type} is a type $\tau=(G,\bsigma,{\mathbf u})$ 
of tropical map to $\Sigma(X)$ such that:
\begin{enumerate}
\item
$G$ is a genus zero graph with $L(G)=\{L_{\out}\}$ with
$\bsigma(L_{\out})\in\P$ and $u_{\tau}:={\bf u}(L_{\out})\not=0$.
\item $\tau$ is realizable and balanced.
\item Let $h:\Gamma(G,\ell)\rightarrow \Sigma(X)$ be the corresponding
universal family of tropical maps, and $\tau_{\out}\in \Gamma(G,\ell)$
the cone corresponding to $L_{\out}$. Then $\dim\tau=n-2$ and
$\dim h(\tau_{\out})=n-1$. Further, $h(\tau_{\out})\not\subseteq
\partial B$.
\end{enumerate}
A \emph{decorated wall type} is a decorated type $\btau=(\tau,{\bf A})$
with $\tau$ a wall type.
\end{definition}

Before using wall types to define the invariants we use
in the canonical wall structure, we first make an observation
in the relative case needed to show properness of the relevant
moduli spaces. For most of the paper, we only work with the absolute
moduli space $\scrM(X,\tau)$, but it turns out that in the relative
case, we may also work with the relative moduli space $\scrM(X/S,\tau)$
when $\tau$ is a wall type. To make this precise, we note that a realizable
type $\tau$ for a punctured map to $X$ is a type for $X/S$ if 
the universal tropical map over $\tau$ to $\Sigma(X)$ fits into a commutative
diagram
\[
\xymatrix@C=30pt
{
\Gamma(G,\ell) \ar[r]^{h_{\tau}}\ar[d]&\Sigma(X)\ar[d]^{\Sigma(g)}\\
\tau\ar[r] & \Sigma(S)
}
\]

\begin{proposition}
\label{prop:relative case over S}
In the relative case, let $\tau$ be a wall type for $X$, and
let $\beta$ be the class of punctured map determined by the data
$u_{\tau}$ and $A\in Q\setminus I$ a non-zero curve class.
Then $\tau$ is a type of punctured map to
$X/S$, and $\scrM(X,\beta)=
\scrM(X/S,\beta)$.
\end{proposition}

\begin{proof}
We first show that the type $\tau=(G,\bsigma,\mathbf{u})$ is in fact a type for
$X/S$. Since $G$ is connected, it is sufficient to show 
that for each $E\in E(G)\cup L(G)$, $\Sigma(g)_*({\bf u}(E))=0$.
Recall that as a wall type, $\tau$ is realizable and balanced.
Since $\tau$ is realizable, we obtain a family of tropical maps
$h_s:G\rightarrow \Sigma(X)$, $s\in \Int(\tau)$. Composing with
$\Sigma(g)$ gives a tropical map $\Sigma(g)\circ h_s:G\rightarrow
\Sigma(S)=\RR_{\ge 0}$, which is balanced by Definition~\ref{def:balanced type}.
But it is then immediate that this map must be constant,
as any tropical map to $\RR_{\ge 0}$ satisfying the balancing condition
and with only one leg must be constant. Thus $\tau$ is a type
defined over $S$.

The equalities of moduli spaces now follow from
\cite[Prop.\ 5.11]{ACGSII}.
\end{proof}

\begin{construction}
\label{const:wall construction}
Fix a wall type $\tau $ and a non-zero curve class $A\in Q\setminus I$.
Let $\beta$ be the class of punctured map determined by the data
$u_{\tau}$ and $A$. Then we obtain a reduced closed stratum
$\foM_{\tau}(\shX,\beta)\subseteq \foM(\shX,\beta)$,
and a moduli space $\scrM_{\tau}(X,\beta)$ along with a morphism
\[
\varepsilon:\scrM_{\tau}(X,\beta)\rightarrow \foM_{\tau}(\shX,\beta).
\]

\begin{lemma}
\label{lem:wall virtual dim}
$\scrM_{\tau}(X,\beta)$ is proper over $\Spec\kk$ and carries
a virtual fundamental class of dimension $0$.
\end{lemma}

\begin{proof}
In the absolute case, $\scrM_{\tau}(X,\beta)$ is closed substack
of $\scrM(X,\beta)$, which is proper over $\Spec\kk$
by \cite[Cor.\ 3.17]{ACGSII}. 

We next consider the relative case.
There is a morphism $\shX\rightarrow \shS = [\AA^1/\GG_m]$, induced
by $g$, where $\shS$ is the Artin fan of $S$. Write
$\shX_0:=\shX\times_{\shS} [0/\GG_m]$; this is a closed substack of $\shX$. 
As $\dim h(\tau_{\out})=n-1$ and
$h(\tau_{\out})\not\subseteq\partial B$, it follows that
$\Sigma(g):h(\tau_{\out}) \rightarrow \Sigma(S)=\RR_{\ge 0}$ is
surjective by Proposition~\ref{prop:gtrop affine submersion}.
By Proposition~\ref{prop:relative case over S}, 
$\Sigma(g)\circ h_s$ is constant for each $s\in\tau$, and hence
for any vertex $v\in V(G)$, $\Sigma(g):h(\tau_v)\rightarrow \RR_{\ge 0}$
is also surjective. From this it follows that any
punctured map $C^{\circ}\rightarrow \shX$ in $\foM_{\tau}(\shX,\beta)$
has image lying set-theoretically in $|\shX_0|$. Since $\foM_{\tau}(\shX,
\beta)$ is reduced by construction, any map must thus factor through
$\shX_0$ log-scheme-theoretically, and hence 
any punctured map in $\scrM_{\tau}(X,\beta)$ must factor through
$X_0$ log-scheme-theoretically.
Thus $\scrM_{\tau}(X,\beta)$
is a closed substack of $\scrM(X\times_S 0,\beta)$,
which is again proper over $\Spec\kk$.

We now calculate the virtual dimension. By \cite[Prop.\ 3.28]{ACGSII}, 
if $\bar\tau$ is the global type induced by $\tau$, then $\foM(\shX,
\bar\tau)$ is pure-dimensional and reduced of dimension
\[
-3+|L(G)|-\dim\tau=-3+1-(n-2)=-n.
\]
The same is then true for $\foM_{\tau}(\shX,\beta)$, as
the forgetful map $\foM(\shX,\bar\tau)\rightarrow \foM_{\tau}(\shX,\beta)$
is finite of generic degree $|\Aut(\bar\tau)|$.
The virtual relative dimension of
$\scrM_{\tau}(X,\beta)$ over $\foM_{\tau}(\shX,\beta)$ at
a punctured map $f:C^{\circ}\rightarrow X$ over a geometric point
is $\chi(f^*\Theta_{X})=A\cdot c_1(\Theta_{X})+n$
by Riemann-Roch.
Recall that $c_1(\Theta_X)=-(K_X+D)\equiv_{\QQ}-\sum_i a_i D_i$
by assumption. 
Further, for any generator $D_i^*$ of $\bsigma(L_{\out})$,
it follows that $a_i=0$ as $\bsigma(L_{\out})\in\P$. Thus by 
\cite[Cor.\ 1.14]{Assoc}, $A\cdot D_i=0$ whenever $a_i\not=0$.
Thus the total virtual dimension is $0$ as claimed.
\end{proof}

We now define
\begin{equation}
\label{eq:Ntau def}
W_{\tau,A}:=\deg [\scrM_{\tau}(X,\beta)]^{\virt}.
\end{equation}
In addition, $h|_{\tau_{\out}}:\tau_{\out}\rightarrow \sigma$
induces a morphism
\[
h_*: \Lambda_{\tau_{\out}}\rightarrow \Lambda_{\sigma},
\]
and we define
\begin{equation}
\label{eq:ktau def}
k_{\tau}:=|\coker(h_*)_{\tors}|=
|\Lambda_{h(\tau_{\out})}/h_*(\Lambda_{\tau_{\out}})|.
\end{equation}
Finally, set
\begin{equation}
\label{eq:wall def}
\boxed{\fop_{\tau,A}:=\big(h(\tau_{\out}), \exp(k_{\tau} W_{\tau,A} t^A 
z^{-u_{\tau}})\big).}
\end{equation}
Here we view $t^A z^{-u_{\tau}}$ as a monomial in $\kk[\shP^+_x]$
for $x\in \Int(h(\tau_{\out}))$
as in Notation~\ref{not:monomial notation}. To view the
exponential as a finite sum, note that
$\kk[Q][\Lambda_{h(\tau_{\out})}]\subseteq \kk[\shP^+_x]$, and 
we may truncate the infinite sum by removing all monomials
which are zero in $(\kk[Q]/I)[\Lambda_{h(\tau_{\out})}]$.

We then define:

\begin{definition}
\label{def:undecorated structure} 
\[
\scrS_{\can}^{\mathrm{undec}}:=\{\fop_{\tau,A}\,|\, \hbox{$\tau$ an
isomorphism class of wall type, $A\in Q\setminus
I$, $W_{\tau,A}\not=0$}\}.
\]
\end{definition}

We note the superscript ``$\mathrm{undec}$'' refers to the \emph{undecorated}
wall structure, in distinction with the \emph{decorated} wall structure
we will define in Construction~\ref{const:decorated wall structure},
which will be equivalent to the above wall structure but which will
be more useful in the proof of consistency.
\end{construction}

\begin{proposition}
\label{prop:canonical is wall}
$\scrS^{\mathrm{undec}}_{\can}$ is a wall structure.
\end{proposition}

\begin{proof}
We need to verify that (1) $\fop_{\tau,A}$ is always a wall and (2) 
$\scrS^{\mathrm{undec}}_{\can}$ is finite. 

For the first item, fix $\tau,A$. Write $u:=u_{\tau}$. It is obvious that 
$h(\tau_{\out})$ is a rational polyhedral cone of codimension one by
assumption. We need to check that the parallel transport of $(-u,A)$
to $\shP_y$ lies in $\shP^+_y$ for each point
$y\in \fop_{\tau,A}\setminus\Delta$.
Since $\Delta$ is the union of all codimension two cones of $\P$, this
is only an issue if $y\in\rho\subseteq\sigma$ where $\rho$ is
codimension one and $\fop_{\tau,A}\not\subseteq\rho$. In this case,
$\dim\bsigma(L_{\out})=n$. Let $v\in V(G)$ be the vertex adjacent to $L_{\out}$,
$\tau_v\in \Gamma(G,\ell)$ the corresponding cone. 
We divide the analysis into three cases, depending on the relationship
between $y$ and $h(\tau_v)$.

{\bf Case 1:} $y\not\in h(\tau_{v})$. 
 From \eqref{eq:leg}, necessarily $-u\in T_{\rho}\sigma$. 
If $y\in\partial B$,
then it is immediate from \eqref{eq:P+x boundary} that
$(-u,A)\in \shP^+_y$.
If instead $y\not\in\partial B$, we may choose
$\xi\in\Lambda_y$ pointing into $\sigma$ as in the description of $\shP^+_y$
of \eqref{eq:P+x alternate}, and then write $-u=u'+a\xi$ for some $a>0$
and $u'\in\Lambda_{\rho}$.
Thus by \eqref{eq:transport1}, $(-u,A)$ is identified with the
element $(u',aZ_+,0,A)$ of $\shP^+_y$. 

{\bf Case 2:} $y \in h(\tau_v)\cap \partial B$.
We are thus in the relative case. 
By Proposition~\ref{prop:relative case over S}, $\beta$ must
be defined over $S$, which in particular means that $u$ is
tangent to the fibres of $g_{\trop}$, and hence $u$ is tangent
to $\partial B$. However, by \eqref{eq:P+x boundary}, it then
follows that $(-u,A)\in \shP^+_y$.

{\bf Case 3:} $y\in h(\tau_v)\setminus \partial B$. Here we assume that
$W_{\tau,A}\not=0$, (as otherwise such a wall would not be included in
$\scrS^{\mathrm{undec}}_{\can}$). Thus there is necessarily a punctured map
$f:C^{\circ}/W \rightarrow X$ over a geometric log point with type $\tau'$
equipped with a contraction to $\tau$. Thus we may mark $f$ with $\tau$, and
this leads to a decorated type $\btau=(\tau,{\bf A})$ with ${\bf A}(v)$ given by
the curve class of the map $f_v:C^{\circ}_v\rightarrow X$, where $C^{\circ}_v$
is the subcurve of $C^{\circ}$ corresponding to $v\in V(G)$, as in the proof of
Lemma~\ref{lem:balancing}. In this case, $u$ points into $\sigma$, and in the
notation $\delta$ of Lemma~\ref{lem:admissible degrees}, (2), ${\bf
A}(v)=d[X_{\rho}]$ with $d \ge \delta(u)$. Thus the total curve class $A$
satisfies $A=\delta(u) [X_{\rho}]+A'$ for some $A'\in Q$.

We may now use the description \eqref{eq:P+x alternate} along with
\eqref{eq:transport1} to test whether $(-u,A)\in\shP^+_x$ lies in the image of
$\fot_{\rho\sigma}:\shP^+_y \rightarrow\shP^+_x$. Choosing $\xi\in\Lambda_y$ as
before, we may now write $u=u'+\delta(u)\xi$ for some $u'\in\Lambda_{\rho}$.
Thus $(-u,A)$ equals $\fot_{\rho\sigma}$ applied to
\[
(-u',0,\delta(u) Z_-, A-\delta(u) \kappa_{\rho})
=(-u',0, \delta(u) Z_-, A')\in \shP^+_y.
\]

\medskip

We have now covered all possible cases for the location of $y$,
and hence $\fop_{\tau,A}$ is a wall.

\medskip

To show $\scrS^{\mathrm{undec}}_{\can}$ is finite, we first observe that as
$Q\setminus I$ is finite by assumption on $Q$ and $I$, there are only a finite
number of choices for $A$. For determining $\tau$, there are only a finite
number of possibilities for $\bsigma(L_{\out})\in\P$. Given a choice of $A$ and
wall type $\tau$ with given $\bsigma(L_{\out})$, $W_{\tau,A}\not=0$ implies that
$\scrM_{\tau}(X,\beta)$ is non-empty, and then \cite[Cor.~1.14]{Assoc} shows
$u_{\tau}$ is determined by the choice of $A$ and $\bsigma(L_{\out})$. If
$\beta$ is the punctured curve class determined by $A$ and a given $u_{\tau}$,
then, since $\scrM(X,\beta)$ is finite type, there are only a finite number of
types of tropical maps $\tau$ appearing in the tropicalizations of curves in
$\scrM(X,\beta)$, i.e., there are only a finite number of $\tau$ such that
$\scrM_{\tau}(X,\beta)$ is non-empty.
\end{proof}

The main result of the paper, to be proved in \S\ref{sec:main theorem}, can then
be stated:

\begin{theorem}
$\scrS^{\mathrm{undec}}_{\can}$ is a consistent wall structure in the sense of
\cite[Def.\ 3.9]{Theta}.
\end{theorem}

The above definition of the canonical wall structure is conceptually
the simplest and most useful in practice (see \cite{HDTV} for some explicit
examples). However, for the proofs of this paper, it is convenient
to replace $\scrS^{\mathrm{undec}}_{\can}$ with an equivalent
wall structure $\scrS_{\can}$ using decorated types as follows.

\begin{construction}
\label{const:decorated wall structure}
Fix a decorated wall type $\btau=(\tau,{\bf A})$, and $A=\sum_{v\in V(G)}
{\bf A}(v)$ the total curve class.
As $\tau$ is realizable, we may view it equivalently as a global type.
Hence we obtain a morphism of moduli spaces $\varepsilon:\scrM(X,\btau)
\rightarrow \foM(\shX,\btau)$. From the proof of
Lemma~\ref{lem:wall virtual dim}, one sees that $\foM(\shX,\btau)$ is
pure-dimensional and $[\scrM(X,\btau)]^{\virt}$ is a zero dimensional
cycle. Hence we may define
\[
W_{\btau}:={\deg [\scrM(X,\btau)]^{\virt}\over |\Aut(\btau)|}.
\]
We may then define a wall
\begin{equation}
\label{eq:refined wall def}
\boxed{\fop_{\btau}:=
\big(h(\tau_{\out}), \exp(k_{\tau} W_{\btau}t^A z^{-u_{\tau}})\big)}
\end{equation}
and
\begin{equation}
\label{eq:Scan refined def}
\scrS_{\can}:=\left\{\fop_{\btau}\,\Big|\,
\substack{\hbox{$\btau$ an isomorphism class of 
decorated wall}\\\hbox{type with total curve class lying in 
$Q\setminus I$}}\right\}.
\end{equation}
We note here we do not exclude walls with $W_{\btau}=0$,
so this wall structure may include an infinite set of 
trivial walls, i.e., with attached function $1$. So technically
this is not a wall structure, but if we remove all trivial walls,
it becomes a finite set as in Proposition~\ref{prop:canonical is wall}.
However, for bookkeeping purposes, it will prove useful to
include these trivial walls.

To see that this gives a wall structure equivalent (in the
sense of Definition~\ref{def:wall structure stuff}) to the previous definition,
fix a wall type $\tau$ and a curve class $A\in Q\setminus I$
such that $W_{\tau,A}\not=0$; this data will
give one wall in the earlier definition of $\scrS_{\can}$. With
$\beta=(\{u_{\tau}\},A)$ the associated class of punctured map, we have the
canonical morphism $\foM(\shX,\tau)\rightarrow \foM_{\tau}(\shX,\beta)$, which
is surjective and finite of degree $|\Aut(\tau)|$. Further, we have a
Cartesian diagram in all categories \cite[Prop.~5.19]{ACGSII}
\[
\xymatrix@C=30pt
{
\coprod_{\btau=(\tau,{\bf A})} \scrM(X,\btau)\ar[r]^>>>>>{j'}
\ar[d]_{\varepsilon_{\tau}}&
\scrM(X,\beta)\ar[d]^{\varepsilon}\\
\foM(\shX,\tau)\ar[r]_j&
\foM(\shX,\beta)
}
\]
Here ${\bf A}$ runs over all decorations of $\tau$ with total curve class $A$.
As $j$ is a map of degree $|\Aut(\tau)|$ onto $\foM_{\tau}(\shX,\beta)$,
we see that
\begin{align}
\label{eq:N tau A decomp}
W_{\tau,A} = {\deg \varepsilon^![\foM_{\tau}(\shX,\beta)]}
= {\deg \varepsilon^!j_*[\foM(\shX,\tau)]\over |\Aut(\tau)|}
= {\deg j'_*\varepsilon_{\tau}^![\foM(\shX,\tau)]\over |\Aut(\tau)|}
= \sum_{\btau=(\tau,{\bf A})} {\deg [\scrM(X,\btau)]^{\virt}\over |\Aut(\tau)|}.
\end{align}
Here the third equality holds by push-pull of \cite[Thm.~4.1]{Man}, and
the last summation runs over all choices of decorations ${\bf A}$
of $\tau$ with total curve class $A$. Note however that
$\Aut(\tau)$ acts on the set of all choices of decoration ${\bf A}$, 
with orbits of this action giving isomorphism classes of decorated
types $\btau$.
The stabilizer of a given $\btau$ is the subgroup $\Aut(\btau)
\subseteq \Aut(\tau)$, and hence the orbit containing $\btau$ is of size
$|\Aut(\tau)|/|\Aut(\btau)|$. Thus the final expression of 
\eqref{eq:N tau A decomp} can be now expressed as a sum over isomorphism
classes of decorations $\btau$ of $\tau$ with total class $A$:
\begin{equation}
\label{eq:N tau A decomp2}
W_{\tau,A}
=\sum_{\btau}{\deg [\scrM(X,\btau)]^{\virt}\over |\Aut(\tau)|}
\cdot {|\Aut(\tau)|\over |\Aut(\btau)|}
=\sum_{\btau}{\deg [\scrM(X,\btau)]^{\virt}\over |\Aut(\btau)|}.
\end{equation}
 From this, we see the collection of walls in the new definition
of $\scrS_{\can}$ coming from all $\btau$ with underlying type $\tau$
and total curve class $A$ is equivalent to the corresponding
single wall in the original definition of $\scrS_{\can}$.
\end{construction}

\begin{example}
The construction of $\scrS_{\can}$ agrees (up to equivalence) 
with the construction of
the canonical scattering diagram of \cite{GHK} when $\dim X=2$ and
$K_X+D=0$, so that $B=\Sigma(X)$. In other words, 
the construction given here is a generalization
of the construction of \cite{GHK}.

To see this, we analyze decorated wall types $\btau$ in dimension 
two. Since we require $\dim\tau=\dim X-2$, in fact such types
will be rigid. Thus there is a unique tropical map $h:G\rightarrow
B$ of type $\tau$, and all vertices of $G$ are mapped to $0\in B$;
otherwise, rescaling $B$ provides a non-trivial deformation of $h$.
Further, we require $\dim h(\tau_{\out})=1$. So
$h(\tau_{\out})$ is a ray in $B$ with endpoint $0$. 
Hence there is no choice but for $u_{\tau}\in B(\ZZ)\setminus
\{0\}$ and $h(\tau_{\out})=\RR_{\ge 0} u_{\tau}$. 
Further, any edge of $G$ contracted
by $h$ has length a free parameter of the tropical curve, and hence
again by rigidity, there are no edges. Thus the only possibility for
$\btau$ is that $G$ has one vertex, with an attached curve class
$A$, and one leg, $L_{\out}$, with ${\bf u}(L_{\out})=u_{\tau}\in
B(\ZZ)$. Note further that $k_{\tau}$ is then just the index of
$u_{\tau}$, i.e., the degree of divisibility of $u_{\tau}$ in $B(\ZZ)$.
Of course $|\Aut(\btau)|=1$. 

To make a comparison with the setup of \cite{GHK}, it is then convenient,
for $u \in B(\ZZ)$ primitive, to define
\[
W_{A,u,k}:=W_{\btau},
\]
where $\btau$ is the decorated wall type as described above with
curve class $A$ and $u_{\tau}=ku$. In \cite{GHK}, only one wall
for each ray $\RR_{\ge 0}u$ of rational slope occurs, and hence it
is then more useful to write the canonical wall structure in the
equivalent form
\[
\scrS'_{\can}:=\left\{\Big(\RR_{\ge 0}u, \exp\Big(\sum_{A,k} k 
W_{A,u,k}t^Az^{-ku}\Big)\Big)\,\Big|\,\hbox{$u\in B(\ZZ)\setminus \{0\}$ primitive}
\right\}.
\]
Here the sum is over all positive integers $k$ and all curve classes
$A$. However, if we write $u=aD_i^*+bD_j^*$ for some non-negative
$a,b$ with $a+b>0$, then by the balancing condition as expressed in
\cite[Cor.~1.14]{Assoc}, we may in fact sum over only those curve
classes with $A\cdot D_i=a$, $A\cdot D_j=b$, $A\cdot D_k=0$ for $k\not=i,j$.

In \cite{GHK}, a wall function is associated to each ray $\fod =\RR_{\ge 0} u$
of rational slope using invariants defined as follows. Assuming $\fod$ does not
coincide with a ray of $\P$, one performs a toric (log \'etale) blow-up
$\pi:\widetilde X\rightarrow X$ by refining $(B,\P)$ along the given ray $\fod$.
Otherwise, if $\fod$ coincides with a ray of $\P$, we may take $\pi$ to be the
identity. Let $\widetilde A$ be a curve class on $\widetilde X$ with
the following properties, depending on whether or not $\pi$ is the identity.
If $\pi$ is not the identity, then we require
that $\widetilde A$ has trivial intersection number with all
boundary components of $\widetilde X$ except for the exceptional divisor $E$ of
$\pi$. If $\pi$ is the identity, let $E$ be the
component of $D$ corresponding to the ray $\fod$. We then require
that $\widetilde A$ has zero intersection number with each component of
$D$ except for $E$. Then \cite{GHK}, following \cite[\S 4]{GPS},
defines a
number $N_{\widetilde A}$. This number is defined as a \emph{relative}
Gromov-Witten invariant for the non-compact pair $(\widetilde
X^{\circ},E^{\circ})$ where $\widetilde X^{\circ}\subseteq \widetilde X$ is
obtained by removing the closure of $\pi^{-1}(D)\setminus E$ from $\widetilde
X$, and $E^{\circ}=E\cap \widetilde X^{\circ}$. This relative invariant counts
rational curves of class $\widetilde A$ with one marked point, with contact
order $k_{\widetilde A}:=\widetilde A\cdot E$ with $E^{\circ}$. 

To compare $N_{\widetilde A}$ with the type of number considered in this paper,
note that the data of $\widetilde A$ and the contact order $k_{\widetilde A}$
also specifies a type $\widetilde\beta$ of logarithmic map to the pair
$(\widetilde X,\pi^{-1}(D))$, and we may compare the logarithmic Gromov-Witten
invariant $\deg [\scrM(\widetilde X,\widetilde\beta)]^{\virt}$ with
$N_{\widetilde A}$. In fact, it follows from \cite{AMW} that these two numbers
will agree provided that every stable log map in the moduli space
$\scrM(\widetilde X, \widetilde\beta)$ factors through $\widetilde X^{\circ}$.
However, it is an elementary exercise in tropical geometry to show that if given
a stable log map $f:C/W\rightarrow \widetilde X$ of type $\widetilde\beta$ with
$W$ a log point, any tropical map in the corresponding family of tropical maps
has image $\fod$. This may be proved using an argument similar to the argument
given in Lemma~\ref{lem:key tropical lemma}, but a similar tropical argument in
two dimensions has already appeared in the proof of \cite[Lem.\ 12]{Bousseau},
which may also be viewed as a tropical interpretation of \cite[Prop.\ 4.2]{GPS}.
On the other hand, any stable log map of type $\widetilde \beta$ which does not
factor through $\widetilde X^{\circ}$ will necessarily have a tropicalization
whose image does not coincide with $\fod$, as follows immediately from the
construction of tropicalization.

Finally, let $A=\pi_*\widetilde A$. Note that in the case $\pi$ is not the
identity, the curve class $\widetilde A$ is uniquely determined by $A$ and the
constraint that $\widetilde A$ has zero intersection number with components of
the strict transform of $D$. Then if $u$ is the primitive generator of $\fod\cap
B(\ZZ)$, we may interpret $k_{\widetilde A}u\in B(\ZZ)$ as a contact order, and
the class $A$ and the contact order $k_{\widetilde A}u$ determines a type of log
map $\beta$ to $X$. It then follows from the birational invariance of log
Gromov-Witten theory of \cite[Thm.~1.1.1]{AW} that $\deg [\scrM(\widetilde
X,\widetilde\beta)]^{\virt}=\deg [\scrM(X, \beta)]^{\virt}$. In particular,
$N_{\widetilde A}=W_{A,u,k_{\widetilde A}}$.

Conversely, given the data of $A$, $u$ and $k$ giving a wall type $\tau$ with
$\scrM(X,\tau)$ non-empty, we obtain via the above blow-up procedure
a curve class $\widetilde A$. Thus we have a one-to-one correspondence between
the set of data contributing to $\scrS'_{\can}$ and the data contributing to the
canonical scattering diagram of \cite{GHK}. Further, we have just observed an
equality of the invariants. The equivalence of $\scrS'_{\can}$ with the
canonical scattering diagram of \cite{GHK} now follows by inspection of the
definition of the canonical scattering diagram in \cite{GHK}.
\end{example}

\begin{example}
Returning to Examples~\ref{ex:running example 1} and 
\ref{ex:running example 2}, we first introduce notation for
the relevant curve classes. Let $e_i$ be the class of a fibre
of $\pi|_{E_i}:E_i \rightarrow Z_i$. Let $f$ be the class of
a curve of the form $\{a\}\times \PP^1 \times \{b\}$ disjoint from
the centers $Z_1,Z_2$.

One can show that the only decorated wall types $\btau$ with
$W_{\btau}\not=0$ are of the following 5 types. In each case,
the underlying graph $G$ of $\btau$ has one vertex $v$, no edges,
and of course a unique leg $L_{\out}$.
\begin{enumerate}
\item For some $k>0$, we have $u_{\tau}=k D_{2,0}^*$, $\bsigma(L_{\out})=
\fop_1:=\RR_{\ge 0} D_{2,0}^*+\RR_{\ge 0} D_{1,0}^*$,
and ${\bf A}(v)= k e_1$.
\item For some $k>0$, we have $u_{\tau}=k D_{2,0}^* $, $\bsigma(L_{\out})=
\fop_2:=\RR_{\ge 0} D_{2,0}^*+\RR_{\ge 0} D_{1,\infty}^*$,
and ${\bf A}(v)= k e_1$.
\item For some $k>0$, we have $u_{\tau}=k(E_2^*-D_{1,\infty}^*)$,
$\bsigma(L_{\out})= \fop_3:=\RR_{\ge} E_2^*+\RR_{\ge 0} D_{1,\infty}^*$,
and ${\bf A}(v)= k(f-e_1-e_2)$.
\item For some $k>0$, we have $u_{\tau}=k D_{2,\infty}^*$,
$\bsigma(L_{\out})=\fop_4:=\RR_{\ge 0} E_2^*+\RR_{\ge 0}D_{2,\infty}^*$,
and ${\bf A}(v)= k(f-e_1)$.
\item For some $k>0$, we have $u_{\tau}=k D_{2,\infty}^*$,
$\bsigma(L_{\out})=\fop_5:=\RR_{\ge 0} D_{1,0}^*+\RR_{\ge 0}D_{2,\infty}^*$,
and ${\bf A}(v)= k(f-e_1)$.
\end{enumerate}
In all cases, $h(\tau_{\out})=\bsigma(L_{\out})$, so walls only have five
possible supports. Further, in every case, $k_{\tau}=k$ and
$W_{\btau}=(-1)^{k-1}/k^2$. After
passing to an equivalent scattering diagram with only one wall with a
given support, we see, for example, that we have a wall
\[
\left(\fop_1, \exp\Big(\sum_{k>0} k {(-1)^{k-1}\over k^2} t^{k e_1}
z^{-kD_{2,0}^*}\Big)\right)=\left(\fop_1, 1+t^{e_1}z^{-D_{2,0}^*}\right).
\]
Similarly, we have four other walls:
\[
\left(\fop_2,1+t^{e_1} z^{-D_{2,0}^*}\right),
\left(\fop_3,1+t^{f-e_1-e_2} z^{D_{1,\infty}^*-E_2^*}\right),
\left(\fop_4,1+t^{f-e_1} z^{-D_{2,\infty}^*}\right),
\left(\fop_5,1+t^{f-e_1} z^{-D_{2,\infty}^*}\right).
\]
Together, these walls cover the affine plane contained in
$B$ which is the union of all two-dimensional cones of $\P$ not
containing $D_{3,0}^*$ or $D_{3,\infty}^*$.
We omit a derivation of these results. Showing (1)--(5) are the only
possibilities is not difficult using \cite[Cor.~1.14]{Assoc} to encode
balancing requirements for punctured maps. A direct calculation of
$W_{\btau}$ has not been carried out, but presumably these multiple
cover calculations can be carried out as in the two-dimensional
case of \cite[Prop.~5.2]{GPS}. Instead, the above formulas for the
wall functions are proved in a more general context in \cite{HDTV}. In
fact, the formulas for these wall functions follows from the consistency
of $\scrS_{\can}$ proved in the present paper.

The walls $\fop_i$ for $i\not=3$ arise from zero-dimensional strata
in one-dimensional moduli spaces of ordinary (non-punctured) stable
log maps. For example, the inclusion of
every fibre of $\pi|_{E_1}:E_1\rightarrow Z_1$ into $X$ may be viewed as a
stable log map. The walls $\fop_1$ and $\fop_2$ capture the curves in
this family which are degenerate with respect to the log structure
on $X$, i.e., fall into $D_{1,0}$ or $D_{1,\infty}$. For $k>1$, we count
multiple covers of these fibres.

The same holds for $\fop_4$ and $\fop_5$, with the curves in question
being, in general, strict transforms of curves of the form
$\{a\}\times \PP^1 \times \{1\}\subseteq \overline{X}$. Note as these
curves intersect $Z_1$ at one point, the class of the strict transform
is indeed $f-e_1$. However, as $a\rightarrow \infty$, this curve
degenerates to a union $C=C_1\cup C_2$ of two irreducible components,
of class $f-e_1-e_2$ (the strict transform $C_1$ of $\{\infty\}\times\PP^1\times
\{1\}$) and of class $e_2$ (the curve $C_2=\pi^{-1}(\infty,\infty,1)$).
It is this degenerate curve and its multiple covers which contribute
to $\fop_4$. Meanwhile the strict transform of
$\{0\}\times\PP^1\times\{1\}$ and its multiple covers
contribute to the wall $\fop_5$.

Finally, $\fop_3$ arises not from a family of curves with one marked point
of positive contact order with the boundary, but from the punctured
log map whose image is $C_1$. This involves a negative contact order with
the divisor $D_{1,\infty}$, which contains $C_1$. This curve is rigid,
and if we did not include a wall for this curve, we would not get
a consistent scattering diagram. This shows how it is essential,
unlike in the case that $\dim X=2$, to take
into account punctured curves rather than just marked curves, as
without $\fop_3$, $\scrS_{\can}$ would not be consistent.
\end{example}

\begin{construction}[The torus action]
As with the canonical wall structure in two dimensions in \cite[\S5]{GHK},
there is a natural torus action on the mirror family constructed
using $\scrS_{\can}$ in all dimensions. We refer the reader to 
\cite[\S4.4]{Theta} for the general setup for torus actions in the
context of families of varieties built via wall structures, and do not
review the notation here.

In this case, there is an action of a torus with character lattice
\[
\Gamma:=\Div_D(X),
\]
i.e., the free abelian group generated by boundary divisors.
In the language of \cite[\S4.4]{Theta}, the base ring $A$ is taken
here to be the ground field $\kk$. To specify the torus action,
we must specify the maps $\delta_Q$ and $\delta_B$ of the diagram
\cite[(4.8)]{Theta}. Note that each irreducible component $D_i$ 
of $D$ defines an $\RR$-valued PL function on $B$; indeed, by 
construction $B\subseteq \Div_D(X)^*_{\RR}$, and the linear functional
$D_i$ on this vector space restricts to a piecewise linear function
on $B$. We continue to denote this PL function as $D_i$. We may then define
\[
\delta_Q:Q\lra\Gamma,\quad\quad \delta_Q(A)= \sum_{i}
(D_i\cdot A) D_i
\]
for $A\in Q$,
and 
\[
\delta_B:\PL(B)^* \lra \Gamma,\quad\quad \delta_B(\beta)=
\sum_{i}\beta(D_i) D_i
\]
for $\beta\in \PL(B)^*$. Commutativity of \cite[(4.8)]{Theta} then
follows by direct computation. Indeed, note that in our current context, 
$Q_0$ is the free monoid 
with generators $e_{\rho}$, $e_{\rho}\in \P^{[n-1]}$, and $h:Q_0\rightarrow
Q$ is defined by
$h(e_{\rho})=[X_{\rho}]$, the kink of our choice of MPL function
$\varphi$ along $\rho$. Then 
\[
\delta_Q(h(e_{\rho}))= \sum_i (D_i\cdot X_{\rho}) D_i.
\]
On the other hand, $g:Q_0 \rightarrow \PL(B)^*$ takes $e_{\rho}$
to the linear functional on $\PL(B)$ taking $\psi\in\PL(B)$ to the
kink of $\psi$ along $\rho$. It is easy to check from toric geometry 
that the kink of
the PL function induced by $D_i$ along $\rho$ is $D_i \cdot X_{\rho}$
(noting that if $D_i$ is not good, it induces the zero function on $B$).
The claimed equality $\delta_Q\circ h = \delta_B\circ g$ then follows.

Thus, by \cite[Thm.~4.17]{Theta}, it follows there is an induced
$\Spec\kk[\Gamma]$-action on the flat family $\check\foX\rightarrow
\Spec(\kk[Q]/I)$ constructed from the wall structure $\scrS_{\can}$
provided that $\scrS_{\can}$ is homogeneous in the sense of
\cite[Def.~4.16]{Theta}. For this, it is enough to check that
if $\btau$ is a decorated wall type with $W_{\btau}\not=0$ and
total curve class $A$, then $\deg_{\Gamma} t^Az^{-u_{\tau}}=0$.
Here $\deg_{\Gamma} t^A = \delta_Q(A)=\sum_i (D_i\cdot A) D_i$.
If we write $u_{\tau}=\sum u_i D_i^*$ as a tangent vector to 
$\bsigma(L_{\out})$ then $\deg_{\Gamma}(z^{-u_{\tau}})$ is
computed from \cite[(4.9)]{Theta} as $-\sum_i u_i D_i$.
It then follows from \cite[Cor.\ 1.14]{Assoc} that necessarily
$u_i=D_i\cdot A$, and hence $\deg_{\Gamma}t^Az^{-u_{\tau}}=0$ as
desired.
\end{construction}

\subsubsection{The relative case}
\begin{construction}
\label{const:asymptotic}
Suppose we are in the situation of Proposition 
\ref{prop:relative case maximal fibres}.
It is useful (see \cite{HDTV}) to compare the wall structures
for $(X,D)$ and $(X_s,D_s)$, which we do as follows. 

Recall from that proposition that after extending the 
affine structure on $B$ across the interior of cells $\omega\in
\P^{[n-2]}_{\partial}$, $\partial B$ may be viewed as the
polyhedral affine pseudomanifold corresponding to the pair $(X_s,D_s)$.
As such, write $(\partial B)_0$ for the complement in $\partial B$
of the union of cones of $\P^{[n-3]}_{\partial}$.
Let $\Lambda_{\partial B}$ denote the sheaf on $(\partial B)_0$
of integral tangent vectors. We may view
$\Lambda_{\partial B}$ as a subsheaf of $\Lambda|_{\partial B}$, where the
latter sheaf is now extended across the interiors of cells
$\omega\in\P^{[n-2]}_{\partial}$. 

Further,
Proposition~\ref{prop:relative case maximal fibres}, (3) then implies
that if $\shP_{\partial B}$ is defined using the MPL function
$\varphi|_{\partial B}$ on $\partial B$, we obtain a natural
inclusion $\shP_{\partial B} \subseteq \shP|_{\partial B}$.
Again, the latter sheaf is viewed as extending across the interiors of
$\omega\in\P^{[n-2]}_{\partial}$. Note that $\shP_{\partial B}$
consists of those sections $m$ of $\shP|_{\partial B}$ such that
$\bar m$ is tangent to $\partial B$.

On the other hand, we also have the MPL function $\varphi_{(X_s,D_s)}$
taking values in $H_2(X_s)\otimes_{\ZZ}\RR$. This leads to a sheaf
$\shP_{(X_s,D_s)}$ on $(\partial B)_0$, and the map $\iota:H_2(X_s)\rightarrow
H_2(X)$ then induces a map
\[
\iota_*:\shP_{(X_s,D_s)} \rightarrow \shP_{\partial B}.
\]
This allows us to define a map, for $x\in (\partial B)_0$,
\[
\iota_*:\kk[\shP_{(X_s,D_s),x}^+]\rightarrow \kk[\shP_{\partial B,x}^+],
\]
where $\shP_{\partial B,x}^+:=\shP_{\partial B,x}\cap \shP^+_x$.
Finally, we may define
\[
\iota(\scrS_{\can,s}):=
\big\{(\fop,\iota_*(f_{\fop}))\,|\, (\fop,f_{\fop})\in \scrS_{(X_s,D_s)}\big\}.
\]
Like $\scrS_{\can,s}$, this is a wall structure on $\partial B$.

On the other hand, we may define the
\emph{asymptotic wall structure} of the wall structure $\scrS_{\can}$
on $B$ by
\begin{equation}
\label{eq:asymp def}
\scrS^{\mathrm{as}}_{\can}:=
\{(\fop\cap \partial B,f_{\fop})\,|\,
\hbox{$(\fop, f_{\fop})\in \scrS$ with $\dim\fop\cap \partial B=n-2$}\}.
\end{equation}
We note that for each such wall, $f_{\fop}\in \kk[\shP^+_{\partial B,x}]
\subseteq \kk[\shP^+_x]$. Indeed, this follows from 
Proposition~\ref{prop:relative case over S}, which implies that for
a wall type $\tau$ for $X$, $u_{\tau}$ is tangent to fibres
of $g_{\trop}$, and in particular tangent to $\partial B$.
Hence $\scrS^{\mathrm{as}}_{\can}$ 
may be viewed as a wall structure on $\partial B$.

We then have:
\end{construction}

\begin{proposition}
Suppose that all good divisors contained in $g^{-1}(0)$ have
multiplicity one in $g^{-1}(0)$. Then
the wall structures $\iota(\scrS_{\can,s})$ and $\scrS^{\mathrm{as}}_{\can}$
are equivalent.
\end{proposition}

\begin{proof}
Note $\iota(\scrS_{\can,s})$ is equal to the canonical wall structure
for $(X_s,D_s)$ constructed using the curve class group $H_2(X)$ rather 
than $H_2(X_s)$. 
In the proof, we generally use $\btau'$ to denote a decorated wall type for
$(X,D)$ and $\btau$ for a decorated wall type for $(X_s,D_s)$, with
curve class decoration ${\bf A}$ taking values in $H_2(X)$.

Given a decorated wall type $\btau'$, the corresponding cone $\tau'$ 
comes with a structure
map $p:\tau'\rightarrow \RR_{\ge 0}$ coming from the fact the type
is defined over $S$. Note $p$ is surjective as $u_{\tau}$
is tangent to fibres of $g_{\trop}$ by 
Proposition~\ref{prop:relative case over S} and 
$h_{\tau'}(\tau'_{\out})\not\subseteq\partial B$
by Definition~\ref{def:wall type},(3).
Then $p^{-1}(0)$ is a proper face of $\tau'$, and hence
corresponds to a contraction morphism $\phi:\btau'\rightarrow\btau$
of decorated types.
The type $\tau$ is realizable provided that the length of $L_{\out}
\in L(G')$
isn't zero on the face $p^{-1}(0)$. Thus $\tau$ realizable and
$\dim\tau=n-3$ is equivalent to 
$\dim \fop_{\btau'}\cap \partial B = n-2$. From 
this it is immediate that $\tau$ is a wall type for $(X_s,D_s)$
if and only if
$\dim \fop_{\btau'}\cap \partial B = n-2$.

Consider the sets
\begin{align*}
S':= {} & \{\btau'\,|\,\hbox{$\btau'$ is a decorated wall type for $X$
with $\dim \fop_{\btau'}\cap\partial B=n-2$}\},\\
S:= {} & \{\btau \,|\,\hbox{$\btau$ is a decorated wall type for $X_s$}\}.
\end{align*}
The discussion of the previous paragraph has constructed a map
$\Phi:S'\rightarrow S$. To show the desired equivalence of wall
structures, it is
enough to show that for $\btau\in S$ we have
\begin{equation}
\label{eq:desired degen formula}
k_{\tau} W_{\btau} = \sum_{\btau'\in \Phi^{-1}(\btau)}
k_{\tau'} W_{\btau'}.
\end{equation}

To do so, we use the decomposition setup of 
\cite[Thm.\ 5.21, Def.\ 5.22, Thm.\ 5.23]{ACGSII}.
Note that for $0\not=s\in S$, $s$ is a trivial log point, so
in fact $\scrM(X_s,\btau)=\scrM(X_s/s,\btau)$.
Thus in particular we have
\[
W_{\btau}=
{\deg [\scrM(X_s/s,\btau)]^{\virt}\over |\Aut(\btau)|}
= {\deg [\scrM(X_0/0,\btau)]^{\virt}\over |\Aut(\btau)|},
\]
the first equality by definition and the second equality by
\cite[Thm.\ 5.23,(1)]{ACGSII}. Further, by 
\cite[Thm.\ 5.23,(2)]{ACGSII},
we have
\begin{equation}
\label{eq:decomp formula}
\deg[\scrM(X_0/0,\btau)]^{\virt} = \sum_{\btau'} {m_{\btau'}
\over |\Aut(\btau'/\btau)|} \deg [\scrM(X_0/0,\btau')]^{\virt},
\end{equation}
where the sum is over codimension one degenerations 
$\btau'$ of $\btau$ in the sense
of \cite[Def.~5.22,(2)]{ACGSII}, and $m_{\btau'}$ is the order of the
cokernel of $p_*:\Lambda_{\tau'}\rightarrow\ZZ$. It is
immediate from that definition that if $\scrM(X_0/0,\btau')\not=\emptyset$ then
such a $\btau'$ is a decorated wall type for $X$. (We need to assume
non-emptiness to guarantee that $\tau'$ is balanced.) 
It also follows from Proposition~\ref{prop:relative case over S}
and the proof of Lemma~\ref{lem:wall virtual dim}
that $\scrM(X_0/0,\btau')=\scrM(X/S,\btau')=\scrM(X,\btau')$. Thus we may
rewrite \eqref{eq:decomp formula} as
\[
k_{\tau}{\deg[\scrM(X_0/0,\btau)]^{\virt}
\over |\Aut(\btau)|} = \sum_{\btau'} {m_{\btau'} k_{\tau}
\over |\Aut(\btau'/\btau)||\Aut(\btau)|} \deg [\scrM(X,\btau')]^{\virt}.
\]
Certainly, $|\Aut(\btau'/\btau)||\Aut(\btau)|=|\Aut(\btau')|$.
Further, $\Lambda_{\tau'_{\out}}=\Lambda_{\tau_{\out}}\oplus\ZZ$ and
$p_*(\Lambda_{\tau'_{\out}})=m_{\tau'}\ZZ$.
Next $\Lambda_{\bsigma'(L_{\out})}=\Lambda_{\bsigma(L_{\out})}\oplus\ZZ$,
with the induced map $g_{\trop,*}:\Lambda_{\bsigma'(L_{\out})}
\rightarrow \ZZ$ being surjective because of the assumption that
$g^{-1}(0)$ is reduced.
Thus an elementary diagram chase gives a short exact sequence
\[
0\rightarrow
\coker(\Lambda_{\tau_{\out}}\rightarrow \Lambda_{\bsigma(L_{\out})})
\rightarrow
\coker(\Lambda_{\tau'_{\out}}\rightarrow \Lambda_{\bsigma'(L_{\out})})
\rightarrow \ZZ/m_{\tau'}\ZZ \rightarrow 0,
\]
so $m_{\tau'}k_{\tau}= k_{\tau'}$, giving
\eqref{eq:desired degen formula}.
\end{proof}

\subsection{Logarithmic broken lines and theta functions}
\label{sec:log broken lines}

Wall types defined in the previous subsection allowed
us to define the canonical wall structure. Consistency
of a wall structure is defined partly via \emph{broken lines}
(reviewed in Definition~\ref{def:broken line}). A key point of
our proof of consistency is a correspondence result that 
associates counts of broken lines to certain punctured invariants.
In this subsection, we shall define these punctured invariants.
We proceed quite analogously with the notion of wall types.

\begin{definition}
\label{def:broken line type}
A \emph{(non-trivial) broken line type} is a 
type $\tau=(G,\bsigma,{\mathbf u})$ 
of tropical map to $\Sigma(X)$ such that:
\begin{enumerate}
\item
$G$ is a genus zero graph with $L(G)=\{L_{\inc},L_{\out}\}$ with
$\bsigma(L_{\out})\in\P$ and 
\[
u_{\tau}:={\bf u}(L_{\out})\not=0, \quad
p_{\tau}:={\bf u}(L_{\inc})\in \bsigma(L_{\inc})\setminus\{0\}
\]
(so that $L_{\inc}$ represents a marked rather than punctured point).
\item $\tau$ is realizable and balanced.
\item Let $h:\Gamma(G,\ell)\rightarrow \Sigma(X)$ be the corresponding
universal family of tropical maps, and let $\tau_{\out}\in \Gamma(G,\ell)$
be the cone corresponding to $L_{\out}$. Then $\dim\tau=n-1$ and
$\dim h(\tau_{\out})=n$. 
\end{enumerate}

We also consider the possibility of a 
\emph{trivial broken line type} $\tau$,
which is not an actual type: the underlying graph $G$ consists
of just one leg $L_{\inc}=L_{\out}$ and no vertices, with the convention
that $u_{\tau}=-p_{\tau}$. Note this does not correspond to
an actual punctured curve.\footnote{Such a broken line type will correspond, as discussed in
\S\ref{sec:correspondence}, to a broken line contained entirely in 
$\bsigma(L_{\inc})$
and which does not bend.}

A \emph{decorated broken line type} is a decorated type 
$\btau=(\tau,\mathbf{A})$ with $\tau$ a broken line type.
In the trivial case, as there are no vertices, this does not
involve any extra information, and in this case, the total curve class
is taken to be $0$.

A \emph{degenerate broken line type} is a type $\tau$ which satisfies
conditions (1) and (2) above and instead of (3),
\begin{itemize}
\item[$(3')$] $\dim \tau=n-2$ and $\dim h(\tau_{\out})=n-1$.
\end{itemize}
\end{definition}

\begin{lemma}
\label{lem:virt dim broken lines}
Let $\btau$ be a non-trivial decorated broken line type. Then
$\scrM(X,\btau)$ is
proper over $\Spec\kk$ and carries a virtual fundamental class of
dimension zero.
\end{lemma}

\begin{proof}
The argument is essentially identical to that of the proof of
Lemma~\ref{lem:wall virtual dim}, and we leave it to the reader
to make the necessary modifications, except for one point.
In that proof, we appealed to Proposition~\ref{prop:relative case over S}
to argue that in the relative case, for each $v\in V(G)$, we have
$\Sigma(g):h(\tau_v)\rightarrow \RR_{\ge 0}$ surjective. Crucially,
a broken line type is not in general a type of tropical map to
$X/S$, so a different argument is necessary. However, note that by
the assumption that $\tau$ is a balanced type,
for any $s\in \Int(\tau)$, $\Sigma(g)\circ h_s:G\rightarrow \RR_{\ge 0}$
is a balanced tropical map. Necessarily $\Sigma(g)_*(\mathbf{u}(L_{\inc}))
=\Sigma(g)(p_{\tau})\in \RR_{\ge 0}$, and hence by balancing,
$\Sigma(g)_*(u_{\tau})\in \RR_{\le 0}$. This implies, again by
balancing, that if $v_{\inc}$, $v_{\out}$ are the vertices adjacent
to $L_{\inc}$, $L_{\out}$ respectively, then $a=\Sigma(g)\circ h_s(v_{\out})
\le \Sigma(g)\circ h_s(v_{\inc})=b$, and for any other vertex
$v\in V(G)$, $\Sigma(g)\circ h_s(v)\in [a,b]$. Further, since
$\dim h(\tau_{\out})=n$, necessarily $h(\tau_{\out})\not\subseteq\partial
B$, and thus $\Sigma(g):h(\tau_{\out})\rightarrow \RR_{\ge 0}$ is
surjective, and so $\Sigma(g):h(\tau_{v_{\out}})\rightarrow \RR_{\ge 0}$
is also surjective. Putting this together, we see that 
$\Sigma(g):h(\tau_v)\rightarrow \RR_{\ge 0}$ is surjective for
any vertex $v\in V(G)$. This is sufficient to complete the argument
of the proof of Proposition~\ref{lem:wall virtual dim} for properness
of the moduli spaces involved.
\end{proof}

We now define, for $\btau$ a non-trivial decorated broken line type,
\begin{equation}
N_{\btau}:={\deg [\scrM(X,\btau)]^{\virt}\over |\Aut(\btau)|}.
\end{equation}
We also have a map $h_*:\Lambda_{\tau_{\out}}\rightarrow 
\Lambda_{\bsigma(L_{\out})}$, necessarily of finite index, and
define
\begin{equation}
\label{eq:ktau broken line def}
k_{\tau}:= |\coker h_*|=
|\Lambda_{\bsigma(L_{\out})}/h_*(\Lambda_{\tau_{\out}})|.
\end{equation}
For a trivial decorated broken line type $\btau$, we set
\[
N_{\btau} := 1, \quad k_{\btau}:=1.
\]

\begin{definition}
\label{def:log theta function}
Fix $p\in B(\ZZ)\setminus \{0\}$, $\sigma\in\P_{\max}$,
$x\in\Int(\sigma)$ not contained in any rationally defined hyperplane 
in $\sigma$. We then define
\[
\vartheta^{\log}_p(x)= \sum_{\btau} k_{\tau}N_{\btau} t^A z^{-u_{\tau}}
\in \kk[\shP^+_x]/I_x
\]
where the sum is over all isomorphism classes of
decorated broken line types $\btau$ with
$p_{\tau}=p$, $x\in h(\tau_{\out})$. Here 
$A$ is the total curve class of $\btau$, and we require $A\in Q\setminus I$.
\end{definition} 

\begin{lemma}
\label{lem:theta finite}
For $p\in B(\ZZ)\setminus \{0\}$,
the number of decorated broken line types $\btau$ with $p_{\tau}=p$,
total curve class in $Q\setminus I$, and $N_{\btau}\not=0$ is finite. In 
particular, the sum defining $\vartheta^{\log}_p(x)$ is finite.
\end{lemma}

\begin{proof}
As $Q\setminus I$ is finite, there are only a finite number of total
curve classes. Further, there are only a finite number of possibilities
for $\bsigma(L_{\out})$. For each curve class $A$ and choice
of $\bsigma(L_{\out})$, it follows from 
\cite[Cor.\ 1.14]{Assoc} that $u_{\tau}$ is determined by $p$
if $\scrM(X,\btau)$ is non-empty. If $\beta$ is the punctured curve
class given by contact orders $p$ and $u_{\tau}$ and curve
class $A$, then $\scrM(X,\beta)$ is of finite type,
and there are only a finite number of decorated types $\btau$
appearing as the types of tropicalizations of punctured maps
in $\scrM(X,\beta)$. Thus there are only a finite number of possible
choices of $\btau$ with given total curve class $A$ and
$\bsigma(L_{\out})$ satisfying $\mathbf{u}(L_{\inc})=p$
and $\mathbf{u}(L_{\out})=u_{\tau}$.
\end{proof}


\section{A correspondence theorem for broken lines}
\label{sec:correspondence}

\subsection{Broken lines and statement of the correspondence theorem}
\label{subsec:broken lines}

We first recall the definition of \emph{broken lines}
\cite[Def.\ 3.3]{Theta} for a wall structure $\scrS$ on $B$. 
We continue to work with
the specific choice of $(B,\P)$ derived from $(X,D)$, a monoid
$Q\subset H_2(X)$ and monoid ideal $I\subset Q$ with $Q\setminus I$
finite as usual.

\begin{definition}
Let $\gamma:(a,b)\rightarrow B_0$ be a path with
$t\in (a,b)$ such that $\gamma(t)\in |\scrS|\setminus \Sing(\scrS)$.
Let $n\in \check\Lambda_{\gamma(t)}$
be a primitive cotangent vector annihilating the tangent space to
$|\scrS|$ at $\gamma(t)$. Thinking of $n$ as defining a linear function
locally near $\gamma(t)$ which is zero along $|\scrS|$, we assume
further that $n$ is positive on 
$\gamma(t-\epsilon,t)$ and negative on $\gamma(t,t+\epsilon)$ for some
$\epsilon>0$. Suppose $\gamma(t-\epsilon)\in\sigma$, $\gamma(t+\epsilon)
\in\sigma'$ with $\sigma,\sigma'\in\P_{\max}$.
Given an expression $az^{m}$ with $a\in \kk[Q]/I$, $m\in
\Lambda_{\sigma}$, $\langle m,n\rangle >0$, 
we define a \emph{result of
transport along $\gamma$} of $az^m$ to be $a_iz^{m_i}$ 
chosen as follows. Write
\[
f_{\gamma(t)}^{\langle n,m\rangle} \cdot
\fot_{\sigma,\sigma'}(az^m)
= \sum_i a_i z^{m_i}
\]
inside $(\kk[Q]/I)[\Lambda_{\sigma'}]$,
with the $m_i\in\Lambda_{\sigma'}$ mutually distinct.
We then take $a_i z^{m_i}$ to be one of the terms in this sum.
Here $\fot$ is as defined in \eqref{eq:fot ring hom}
and $f_{\gamma(t)}$ is as defined in \eqref{eq:fx def}, which
may be viewed via transport as an element of
$(\kk[Q]/I)[\Lambda_{\sigma'}]$.
\end{definition}

\begin{definition}
\label{def:broken line}
A \emph{broken line} for a wall structure $\scrS$ on $(B,\P)$
is a proper continuous map
\[
\beta:(-\infty,0]\rightarrow B_0
\]
with image disjoint from $\Sing(\scrS)$, along with a sequence
$-\infty=t_0<t_1<\cdots<t_r=0$ for some $r\ge 1$ with $\beta(t_i)
\in |\scrS|$ for $1\le i\le r-1$, and for each $i=1,\ldots,r$ an
expression $a_iz^{m_i}$ with $a_i\in \kk[Q]/I$, $m_i\in \Lambda_{\beta(t)}$
for any $t\in (t_{i-1},t_i)$, subject to the following conditions:
\begin{enumerate}
\item $\beta|_{(t_{i-1},t_i)}$ is a non-constant affine map with
image disjoint from $|\scrS|$, hence contained in the interior
of some $\sigma\in\P_{\max}$, and
$\beta'(t)=-m_i$ for all $t\in (t_{i-1},t_i)$.
\item For each $i=1,\ldots,r-1$ the expression 
$a_{i+1}z^{m_{i+1}}$ is a result of transport of $a_iz^{m_i}$
along $\beta|_{(t_{i-1},t_{i+1})}$.
\item $a_1=1$.
\end{enumerate}
The \emph{asymptotic monomial} of $\beta$ is $m_1$. Note that
since $\beta((-\infty,t_1])$ is an affine half-line contained in
a cone $\sigma\in\P$, we may identify $m_1$ with an integral point
of $\sigma$, and hence view $m_1\in B(\ZZ)$.

Given a broken line $\beta$, we write $a_{\beta}=a_r$, $m_{\beta}=m_r$.
\end{definition}

\begin{remark}
This only differs in one way from \cite[Def.\ 3.3]{Theta}, namely 
the requirement that $a_1=1$.
In \cite{Theta}, broken lines satisfying this additional condition are
called \emph{normalized} broken lines.
As we never use any other type of broken line here, we include
this in the definition.
\end{remark}

Broken lines define theta functions:

\begin{definition}
\label{def:broken line theta function}
Fix $p\in B(\ZZ)\setminus \{0\}$, $\sigma\in\P_{\max}$, $x\in\Int(\sigma)$
not contained in any rationally defined hyperplane in $\sigma$.
We then define
\[
\vartheta_p(x)=\sum_{\beta} a_{\beta}z^{m_{\beta}}\in \kk[\shP_x^+]/I_x,
\]
where the sum is over all broken lines $\beta$ with asymptotic
monomial $p$ and $\beta(0)=x$.
\end{definition}

We note this sum is finite, e.g., by \cite[Lem.\ 3.7]{Theta}.

The main result of this section may now be stated as:

\begin{theorem}
\label{thm:theta correspondence}
For $p\in B(\ZZ)\setminus \{0\}$, $\sigma\in\P_{\max}$, $x\in\Int(\sigma)$
not contained in any rationally defined hyperplane, we have
\[
\vartheta_p(x)= \vartheta^{\log}_p(x).
\]
\end{theorem}

Before proving this theorem, we first recast it as a tropical correspondence
theorem between broken lines and punctured log curves.
To do so, we need to deal with a certain amount of book-keeping. To
this end, we first refine the notion of broken line. This
refinement keeps track of which wall types cause the bending of the
broken line.

\begin{definition}
\label{def:decorated broken line}
A \emph{decorated broken line} for $\scrS_{\can}$ is a proper
continuous map 
\[
\beta:(-\infty,0]\rightarrow B_0
\]
with image disjoint from $\Sing(\scrS_{\can})$, along with a sequence
$-\infty=t_0<t_1<\cdots<t_r=0$ for some $r\ge 1$ with $\beta(t_i)
\in |\scrS_{\can}|$ for $1\le i\le r-1$, along with additional data:
\begin{itemize}
\item[(i)]
For each $i=1,\ldots,r$ an
expression $a_iz^{m_i}$ with $a_i\in \kk[Q]/I$, $m_i\in \Lambda_{\beta(t)}$
for any $t\in (t_{i-1},t_i)$. 
\item[(ii)] For each $i=1,\ldots,r-1$, let $\scrS_i:=\{\fop_{\btau}\in
\scrS_{\can}\,|\, \beta(t_i)\in \fop_{\btau}\}$.
Then we are given in addition a function
$\mu_i:\scrS_i\rightarrow \NN$. 
\end{itemize}
This data is subject to the following conditions:
\begin{enumerate}
\item $\beta|_{(t_{i-1},t_i)}$ is a non-constant affine map, and
$\beta'(t)=-m_i$ for all $t\in (t_{i-1},t_i)$.
\item Define the \emph{support} of $\mu_i$ to 
be the subset of $\scrS_i$
on which $\mu_i$ takes non-zero value. Then the support of $\mu_i$
must be finite, and non-empty if $\beta(t_i)$ lies in the interior
of a maximal cone of $\P$. Further, if $\beta(t)\in\rho\in\P^{[n-1]}$,
then $t=t_i$ for some $i$.
\item For each $i=1,\ldots,r-1$, if $\beta(t_i-\epsilon)\in\sigma$,
$\beta(t_i+\epsilon)\in\sigma'$, we have
\begin{equation}
\label{eq:parallel transport for correspondence}
a_{i+1}z^{m_{i+1}} = \fot_{\sigma,\sigma'}(a_i z^{m_i})
\prod_{\fop_{\btau}\in \scrS_i} {(\langle n_i, m_i\rangle k_{\tau}
W_{\btau}t^{A} z^{-u_{\tau}})^{\mu_i(\fop_{\btau})}\over 
\mu_i(\fop_{\btau})!}.
\end{equation}
Here $n_i \in \check\Lambda_{\beta(t_i)}$ is primitive,
vanishing on the tangent space to $|\scrS_{\can}|$ at
$\beta(t_i)$, and positive on $m_i$.
\item $a_1=1$.
\end{enumerate}
\end{definition}

\begin{remark}
There are a couple of points of distinction between the above
definition and Definition~\ref{def:broken line}. First, it is possible
that $\beta(t)\in |\scrS_{\can}|$ with $t\not=t_i$ for
any $i$ provided $\beta(t)$ lies in the interior of a maximal cell.
Second, the data of the $\mu_i$ completely determines the change
of monomial. To explain the expression 
\eqref{eq:parallel transport for correspondence},
note that by the description \eqref{eq:Scan refined def} of $\scrS_{\can}$,
\[
f_{\beta(t_i)}^{\langle n_i,m_i\rangle}
=
\exp\left( \sum_{\fop_{\btau}\in \scrS_i} \langle n_i,m_i\rangle k_{\tau}
W_{\btau}t^A z^{-u_{\tau}}\right)=
\sum_{\mu_i}
\prod_{\fop_{\btau}\in \scrS_i} {(\langle n_i, m_i\rangle k_{\tau}
W_{\btau}t^{A} z^{-u_{\tau}})^{\mu_i(\fop_{\btau})}\over 
\mu_i(\fop_{\btau})!}
\]
where the sum is over all functions $\mu_i:\scrS_i\rightarrow\NN$
with finite support. Thus
\eqref{eq:parallel transport for correspondence} arises from one
of the summands, and if $\beta(t_i)$ falls in the interior
of a maximal cell, we do not allow the trivial summand $1$ (with $\mu_i
\equiv 0$). This would correspond, in Definition~\ref{def:broken line},
to a situation where $a_iz^{m_i}=a_{i+1}z^{m_{i+1}}$.

It is then clear that we have a refined description
\[
\vartheta_p(x)=\sum_{\beta} a_{\beta}z^{m_{\beta}} \in \kk[\shP^+_x]/I_x
\]
as a sum over decorated broken lines with asymptotic monomial $p
\in B(\ZZ)$ and $\beta(0)=x$. Each term in the summation
of Definition~\ref{def:broken line theta function} is now
split up into a sum over a number of decorated broken lines.
\end{remark}

\begin{construction}
\label{const:broken line to type}
We may now associate to a decorated broken line $\beta$ an
isomorphism class of a decorated broken line type $\btau_{\beta}=
(G_{\beta},\bsigma_{\beta},{\bf u}_{\beta}, {\bf A}_{\beta})$.
This is done as follows. 

First, we adopt the following notation.
For each $\fop_{\btau}\in \scrS_i$,
write $(G_{\btau},\bsigma_{\btau},{\bf u}_{\btau})$
and ${\bf A}_{\btau}$ for the data giving the decorated type
$\btau$. Let $L_{\out,\btau}\in L(G_{\btau})$ be the unique leg, adjacent
to a unique $v_{\out,\btau}\in V(G_{\btau})$.

%

The graph $G_{\beta}$ has a spine comprising
of the legs $L_{\inc}$ and $L_{\out}$
and the edges $E_1,\ldots,E_{r-2}$ and their vertices
$v_1,\ldots,v_{r-1}$, with the vertices of $E_i$ being $v_i$, $v_{i+1}$
and the vertices of $L_{\inc}$ and $L_{\out}$ being $v_1$
and $v_{r-1}$ respectively. At the vertex $v_i$ we glue 
$\mu_i(\fop_{\btau})$ copies of $G_{\btau}$ to $v_i$ along the unique
leg of $G_{\btau}$.

Next, we define $\bsigma_{\beta}$ as follows.
We define $\bsigma_{\beta}(v_i)$ or $\bsigma_{\beta}(E_i)$ to be the minimal
cone of $\P$ containing $\beta(t_i)$ or $\beta((t_i,t_{i+1}))$ 
respectively, and $\bsigma_{\beta}(L_{\inc})$ and $\bsigma_{\beta}(L_{\out})$
to be the minimal cones containing
$\beta((-\infty,t_1))$ or $\beta((t_{r-1},t_r))$ respectively. 
Further, for each $i$, $\fop_{\btau}\in\scrS_i$, we may view 
each of the $\mu_i(\fop_{\btau})$ copies
of $G_{\btau}$ 
as a subgraph of $G_{\beta}$, and we take $\bsigma_{\beta}$ to
agree with $\bsigma_{\btau}$ on each copy of $G_{\btau}$.

To define $\mathbf{u}_{\beta}$, we again take it to agree with 
$\mathbf{u}_{\btau}$ on any copy of $G_{\btau}$, 
and take $\mathbf{u}_{\beta}(E_i)=-m_{i+1}$ 
with $E_i$ oriented from $v_i$ to $v_{i+1}$. We also set
$\mathbf{u}_{\beta}(L_{\inc})=m_1$, $\mathbf{u}_{\beta}(L_{\out})=-m_r$ (using
the standard convention that legs are oriented away from their vertex).

This defines the global type $\tau_{\beta}$. For the decoration 
${\bf A}_{\beta}$, it similarly agrees with ${\bf A}_{\btau}$
on each copy of $G_{\btau}$. 
If $v=v_i$,
we take ${\mathbf A}_{\beta}(v_i)=0$ if $\beta(t_i)$ lies in the interior
of a maximal cell of $\P$, but if $\beta(t_i)\in\Int(\rho)$
with $\rho \in \P^{[n-1]}$, then we take
${\mathbf A}_{\beta}(v_i)=\langle n_i,m_i\rangle [X_{\rho}]\in H_2(X)$.
\end{construction}

\begin{lemma}
Let $\beta$ be a decorated broken line with $x=\beta(0)\in\Int(\sigma)$,
$\sigma\in\P_{\max}$, $x$ not contained in any rationally defined
hyperplane. Then $\btau_{\beta}$ is a decorated broken line type.
\end{lemma}

\begin{proof}
Condition (1) of Definition~\ref{def:broken line type} is
immediate. 

For condition (2), we first show $\tau_{\beta}$ is realizable, i.e., 
there is a tropical map $h:G_{\beta}\rightarrow \Sigma(X)$ of type 
$\tau_{\beta}$. On the spine $G'_{\beta}$ of $G_{\beta}$,
consisting of vertices $v_1,\ldots,
v_{r-1}$, edges $E_1,\ldots, E_{r-2}$ and legs $L_{\inc}, L_{\out}$,
we define $h$ to agree with $\beta$, extending $\beta$
linearly on $L_{\out}$ if necessary. 
For each $i$ and for each $\fop_{\btau}\in \scrS_i$, $\btau$
is a wall type, and hence is realizable with an $n-2$-dimensional
universal family with output leg tracing out the wall $\fop_{\btau}$.
Further, $\beta(t_i)\in\Int(\fop_{\btau})$.
Thus there is a unique tropical map 
$h_{\btau}:G_{\btau}\rightarrow \Sigma(X)$ of type $\tau$ with $\beta(t_i)
\in h_{\btau}(L_{\out,\btau})$. Thus we may define $h$ restricted
to any copy of $G_{\btau}$ in $G_{\beta}$ to be given by $h_{\btau}$,
appropriately truncating the edge $L_{\out,\btau}$ as necessary.
This shows realizability.

Next, we observe that $\tau_{\beta}$ is balanced. Because this is
true for each wall type $\tau$ appearing in $\tau_{\beta}$, this
only needs to be checked at the vertices $v_i$, but such balancing
follows from \eqref{eq:parallel transport for correspondence}.
Indeed, by considering exponents,
that equation tells us that
\[
m_{i+1}=m_i-\sum_{\fop_{\btau}\in\scrS_i}\mu_i(\fop_{\btau}) u_{\tau},
\]
which is the balancing condition at $v_i$ given the specified contact
orders in Construction~\ref{const:broken line to type} of edges
adjacent to $v_i$.
In the relative case, we also need to verify that $\Sigma(g)\circ h$
is balanced as a map to $\Sigma(S)=\RR_{\ge 0}$. However, this again
holds at all vertices of $G_{\beta}$ except for the $v_i$'s because
the same is assumed true of the wall types $\tau$. At a vertex $v_i$,
the result follows from the balancing in $B$ and the fact (Lemma 
\ref{prop:gtrop affine submersion}) that $g_{\trop}$ is an affine submersion.

For condition (3), 
observe that, by elementary convex geometry,
$\beta$ varies in an $n$-dimensional family: each $\beta(t_i)$
is constrained to live in an $n-1$-dimensional subspace, and the
location of $\beta(t_1)$ and the $m_i$ completely determine the
map $\beta$, up to the location of $\beta(0)$, which gives an
extra parameter. In particular, the dimension of the universal
family of tropical maps of type $\tau_{\beta}$ is then clearly
$n-1$, and $h(\tau_{\beta,\out})$ is $n$-dimensional. 
This shows $\tau_{\beta}$ is a broken line type.
\end{proof}

We note that the broken line type $\btau_{\beta}$ had a very
specific form for the curve classes associated to the vertices $v_i$.
We codify this as follows:

\begin{definition}
Let $\btau=(G,\bsigma,{\mathbf u},{\mathbf A})$ 
be a decorated broken line type with
spine $G'\subseteq G$ (see Definition~\ref{def:spine}), 
the latter having vertices
$v_1,\ldots,v_{r-1}$ and edges $E_i$ with vertices $v_i, v_{i+1}$.
By Lemma~\ref{lem:key tropical lemma}, $\bsigma(v_i)\in\P^{[n]}\cup
\P^{[n-1]}$. We say $\btau$ is \emph{admissible}
if $\mathbf{A}(v_i)=0$ whenever $\bsigma(v_i)\in \P^{[n]}$,
and $\mathbf{A}(v_i) = |\langle n_i,\mathbf{u}(E_i)\rangle| 
[X_{\bsigma(v_i)}]$ whenever $\bsigma(v_i)\in \P^{[n-1]}$,
where $n_i\in\check\Lambda_x$ for $x\in \Int(\bsigma(v_i))$ is a choice
of primitive normal vector to $\bsigma(v_i)$.
\end{definition}

By construction, if $\beta$ is a decorated broken line, then
$\btau_{\beta}$ is admissible. The following is immediate from
Lemma~\ref{lem:admissible degrees}:

\begin{lemma}
Let $\btau$ be a decorated broken line type. Then $N_{\btau}\not=0$
implies that $\btau$ is admissible.
\end{lemma}

We now reverse the procedure of 
Construction~\ref{const:broken line to type}
and go from decorated broken line types
to decorated broken lines.

\begin{construction}
\label{const:type to broken line}
Let $\btau=(G,\bsigma,{\bf u},{\bf A})$ be a decorated broken line type,
$h$ the universal tropical map parameterized by the cone $\tau$, and
$x\in \Int(h(\tau_{\out}))$.
We may construct a broken line $\beta_{\btau,x}$ as follows. There 
exists a unique $s\in \Int(\tau)$ such that
$x\in h_s(L_{\out})$. Let $G'$ be the spine
of $G$. 
Label the vertices of $G'$ as $v_1,\ldots,v_{r-1}$
with $L_{\inc}$ adjacent to $v_1$, $L_{\out}$ adjacent to $v_{r-1}$,
and edges $E_1,\ldots, E_{r-2}$, with $E_i$ having vertices $v_i$ and $v_{i+1}$.
Shrink $L_{\out}$ so that the non-vertex endpoint of $L_{\out}$
maps to $x$ under $h_s$.
We may then choose an identification $\psi:\RR_{\le 0}\rightarrow G'$ 
with the following property. Let $t_i \in \RR_{\le 0}$ satisfy
$\psi(t_i)=v_i$, $1\le i \le r-1$, take $t_r=0$,
and set $\beta_{\btau,x}=h_s|_{G'}\circ\psi$. We require that
condition (1) of Definition
\ref{def:decorated broken line} is satisfied for $\beta_{\btau,x}$ with
$m_i = - \mathbf{u}(E_i)$, with $E_i$ oriented from
$v_i$ to $v_{i+1}$. Note by construction that $\beta_{\btau,x}(0)=x$.

We now specify the decorations, notably the functions $\mu_i$.
Let $G_{i1},\ldots,G_{is_i}$
be the closures of the connected components of $G\setminus \{v_i\}$
not containing $L_{\inc}$ or $L_{\out}$. These are the graphs
attached to $v_i$ but only intersecting $G'$ at $v_i$. We view
the edge of $G_{ij}$ adjacent to the vertex $v_i$ as a leg 
$L_{ij,\out}$ of
$G_{ij}$, so that $v_i$ is not a vertex of $G_{ij}$.
This gives rise to decorated types $\btau_{i1},\ldots,\btau_{is_i}$
by restriction of $\bsigma,\mathbf{u},\mathbf{A}$ to each graph $G_{ij}$. 

We claim that each $\btau_{ij}$ is a wall type.
Indeed, condition (1) of Definition~\ref{def:wall type}
is immediate except for the statement that $\mathbf{u}(L_{ij,\out})\not=0$, 
which we rule out below. For condition (2), note $\btau_{ij}$
is realizable and balanced since $\btau$ is realizable and balanced.

For condition (3), let $h_{ij}:\Gamma(G_{ij},\ell_{ij})
\rightarrow \Sigma(X)$ be the universal tropical map of type
$\btau_{ij}$, defined over the cone $\tau_{ij}$. 
First note that necessarily $h(\tau_{v_i})\subseteq 
h_{ij}(\tau_{ij,\out})$, and by
Lemma~\ref{lem:key tropical lemma}, (2), $\dim h(\tau_{v_i})=n-1$. 
Thus $\dim h_{ij}(\tau_{ij,\out})\ge n-1$. However
by Lemma~\ref{lem:key tropical lemma}, (1), still assuming that
$\mathbf{u}(L_{ij,\out})\not=0$, $\dim h_{ij}(\tau_{ij,\out})\le n-1$
and thus this dimension is $n-1$, as desired.

Suppose now that $\dim \tau_{ij}>n-2$. 
Thus for any fixed $s'\in \Int(\tau)$, 
there is a positive dimensional subset $\omega \subseteq
\tau_{ij}$ such that for $s''\in\omega$, $h_{s'}(v_i)\in 
h_{ij,s''}(L_{ij,\out})$,
where $L_{ij,\out}$ is the unique leg of $G_{ij}$. Thus 
for each point $y\in h(\tau_{v_i})$, there is a positive dimensional
family of tropical maps of type $\tau$ taking $v_i$ to $y$.
Again since $\dim h(\tau_{v_i})=n-1$, this shows $\dim \tau\ge n$,
a contradiction.

We now eliminate the case that ${\bf u}(L_{ij,\out})=0$. Indeed, 
if this is the case and $v$ is the unique vertex of $L_{ij,\out}$
in $G_{ij}$, then $h_{ij}(\tau_{ij,v})=h_{ij}(\tau_{ij,L_{ij,\out}})$,
and this is again $n-1$-dimensional as above. Thus $\dim\tau_{ij}\ge n-1$.
On the other hand, if $E_{ij}$ is the edge of $G$ corresponding to
the leg $L_{ij,\out}$ of $G_{ij}$, we have ${\bf u}(E_{ij})=
{\bf u}(L_{ij,\out})=0$, so the length
of $E_{ij}$ is arbitrary, and thus we obtain again that $\dim\tau
\ge n$. Thus we conclude that $\btau_{ij}$ is a wall type.

We may now define $\mu_i:\scrS_i\rightarrow \NN$ by 
defining, for $\fop_{\btau'}\in \scrS_i$,
$\mu_i(\fop_{\btau'})=|\{j\,|\, \btau_{ij}\simeq\btau'\}|$.

Having defined the $\mu_i$, \eqref{eq:parallel transport for correspondence}
defines the monomials $a_iz^{m_i}$. There is one thing to check:
as we have already defined $m_i$ in terms of ${\bf u}$, we need
to check that \eqref{eq:parallel transport for correspondence} yields
the same values for $m_i$. However, this is immediate inductively
by balancing at the vertices $v_i$.

Hence we have constructed a decorated broken line $\beta_{\btau}$.
\end{construction}

Constructions~\ref{const:broken line to type} and 
\ref{const:type to broken line} now immediately give:

\begin{proposition}
\label{prop:one to one}
Given $x\in B_0$ not
contained in any rationally defined hyperplane
and $p\in B(\ZZ)\setminus \{0\}$,
there is a one-to-one correspondence
\begin{align*}
&\{\beta\,|\,\hbox{
$\beta$ is a decorated broken line with endpoint $x$ and asymptotic
monomial $p$}\}\\ 
\leftrightarrow& \left\{\btau\,\Big|\,\substack{\hbox{$\btau$ is an isomorphism 
class of admissible decorated}\\ \hbox{broken line 
types with $x\in \Int(h(\tau_{\out}))$ and $p_{\tau}=p$
}}\right\}
\end{align*}
given by $\beta\mapsto \btau_{\beta}$, $\btau\mapsto \beta_{\btau,x}$.
\end{proposition}

The key result, to be proved via gluing of punctured curves in
the next subsection, is then:

\begin{theorem}
\label{thm:main correspondence theorem}
Let $\btau$ be an admissible decorated broken line type, and let
$\beta_{\btau,x}$ be the corresponding broken line with
endpoint $x\in \Int(h(\tau_{\out}))$. Then we have
\[
a_{\beta_{\btau,x}}z^{m_{\beta_{\btau,x}}}=
k_{\tau} N_{\btau} t^A z^{-u_{\tau}},
\]
where $A$ is the total degree of $\btau$.
\end{theorem}

\begin{proof}[Proof of Theorem~\ref{thm:theta correspondence}
given Theorem~\ref{thm:main correspondence theorem}.]
This follows immediately from the definitions
Definition~\ref{def:broken line theta function} and 
\ref{def:log theta function} of the two types of theta functions, 
Proposition~\ref{prop:one to one},
and Theorem~\ref{thm:main correspondence theorem}.
\end{proof}


\subsection{Proof of the correspondence theorem, Theorem~\ref{thm:main correspondence theorem}}

The proof is by induction on the length of $\beta_{\btau,x}$, that is, the number
$r\ge 1$ in Definition~\ref{def:broken line}. The base case $r=1$ is the case of
a straight broken line ending at some point $x$ in the interior of a maximal
cell. The corresponding broken line type is then a trivial one $\btau_\inc$
without a vertex and trivial curve class $A=0$. Thus $N_\btau=1$, $k_\btau=1$
by definition and the claimed equality holds trivially.

For the inductive step, let $\btau=(\tau,\bf A)$, $\tau=(G,\bsigma,\bf u)$ be an
admissible decorated broken line type as in the theorem, with associated broken
line $\beta_{\btau,x}$ of length $r$. Let $v\in V(G)$ be the vertex adjacent to
$L_\out$, $\tau_v\in\Gamma(G,\ell)$ the corresponding cone and
$\sigma_\fop=\bsigma(v)$ the smallest cell containing $h(\tau_v)$. Here we
denote by $\fop$ the set-theoretic intersection of all walls or codimension one
cells of $\P$ containing $h(\tau_v)$. Thus $\fop$ is a polyhedral subset of $B$
of dimension $n-1$. Note that $\dim\sigma_\fop\in\{n-1,n\}$, and we refer to the
two cases as codimension one and codimension zero, respectively. We work in an
affine chart of $B_0$ containing $\Int\sigma_\fop$, with $\Lambda$ denoting the
lattice of integral tangent vector fields \and by $\Lambda_\fop\subset\Lambda$
the corank one lattice of vectors tangent to $h(\tau_v)$. Let $E_\inc\in E(G)$
be the unique edge adjacent to $v$ and belonging to the spine of $\tau$, and
$E_1,\ldots,E_l$ the remaining edges adjacent to $v$, all oriented toward $v$.
The case $l=0$ means that no such edges are present and then $E_1,\ldots, E_l$
denotes the empty sequence. Denote further $u_\inc=\bu(E_\inc)$,
$u_\out=\bu(L_\out)=u_\tau$, $u_i=\bu(E_i)$, and $\sigma_\inc=\bsigma(E_\inc)$
the maximal cell containing $h_s(E_\inc)$ for all $s\in \tau$.

We now split $G$ at all the edges $E_\inc, E_1,\ldots,E_l$ adjacent to $v$. See
\cite[\S5.1]{ACGSII} for a formal treatment of splitting. Splitting an edge $E$
with vertices $v_1,v_2$ leads to a pair of legs that we denote $(E,v_1)$,
$(E,v_2)$. By Construction~\ref{const:type to broken line} the types obtained
from $\btau$ after splitting are as follows.
\begin{enumerate}
\item
A decorated broken line type $\btau_\inc=(G_\inc, \bsigma_\inc, \bu_\inc,
\bA_\inc)$, with the unique leg $(E_\inc,v_\inc)$ obtained from
splitting and $\bu_\inc(E_\inc,v_\inc)=u_\inc$.
\item
Decorated wall types $\btau_i=(G_i,\bsigma_i,\bu_i,\bA_i)$, $i=1,\ldots,l$, with
legs $(E_i,v_i)\in L(G_i)$ obtained from splitting, and $\bu_i(E_i,v_i)=u_i$.
\item
$\btau_0=(G_0,\bsigma_0,\bu_0,\bA_0)$, the type of a decorated punctured map
with only one vertex $v$ with $\bsigma_0(v)=\sigma_\fop=\bsigma(v)$, no edges,
and legs $(E_\inc,v)$, $(E_1, v),\ldots,(E_l,v)$, $(E_\out, v)=L_\out$ with
$\bA_0(v)= \bA(v)$ and
\[
\bu_0(E_\inc,v)=-u_\inc,\quad \bu_0(E_\out,v)=u_\out,\quad
\bu_0(E_i,v)=-u_i,\ i=1,\ldots,l.
\]
\end{enumerate}
As usual the corresponding undecorated types are denoted $\tau_\inc$ and
$\tau_i$, $i=0,1,\ldots,l$. Grouping together identical $\btau_i$ we may relabel
and assume the $\btau_i$ are pairwise non-isomorphic, but each $\btau_i$ to
occur $\mu_i$-times in the splitting for some $\mu_i\in\NN$.

The propagation rule~\eqref{eq:parallel transport for correspondence} for
monomials in the definition of decorated broken lines
(Definition~\ref{def:decorated broken line}) together with the induction
hypothesis shows that the claimed equality
$a_{\beta_{\btau}}z^{m_{\beta_{\btau}}}= k_{\tau} N_{\btau} t^A z^{-u_{\out}}$
is equivalent to the following
\begin{equation}
\label{Eqn: scattering equation}
k_\tau N_\btau t^A z^{-u_\out} = k_{\tau_\inc} N_{\btau_\inc} t^{A_\inc}
t^{d[X_{\sigma_\fop}]} z^{-u_\inc}\cdot\prod_{i=1}^l \frac{\big(d k_{\tau_i}
W_{\btau_i} t^{A_i}z^{-u_{\tau_i}}\big)^{\mu_i}}{\mu_i!}.
\end{equation}
Note that $[X_{\sigma_\fop}]=0$ in the codimension zero case and
$[X_{\sigma_\fop}]$ is the class of the curve corresponding to the $(n-1)$-cell
containing $\fop$ in the codimension one case. Furthermore,
\begin{equation}
\label{Eqn: constant d}
d=|\delta(u_\inc)|
\end{equation}
for $\delta:\Lambda\arr\ZZ$ the quotient by $\Lambda_\fop$, and $k_{\tau_\inc}$,
$k_{\tau_i}$ are defined in \eqref{eq:ktau broken line def} and \eqref{eq:ktau
def}, respectively.

To prove this equality, we employ the numerical splitting formula for punctured
Gromov-Witten invariants proved by Yixian Wu \cite{Wu} and recalled in the
appendix. Assumption~\ref{Ass: toric} of toric gluing strata is trivially
fulfilled in the codimension zero case, while in codimension one it follows from
Lemma~\ref{lem:toric dim one}. In the following discussion, we freely use the
notation introduced in the appendix, but note that we do not split all edges of
$\tau$, but just the subset $E(G,v)\subseteq E(G)$ of $l+1$ edges adjacent to
$v$. Thus for any $E\in E(G,v)$ the vector space $(\Lambda_E)_\RR$ holding the
$E$-component of the displacement vector equals $\Lambda_\RR$. For the time
being we also revert to the original labelling of the $\btau_i$ with possibly
several identical wall types. The numerator $\mu_i!$ in \eqref{Eqn: scattering
equation} arises only at the very end when taking the automorphism group of the
glued decorated type $\btau$ into account.

To invoke Theorem~\ref{Thm: gluing theorem} we need to choose a general
displacement vector $\nu=(\nu_E)_{E\in E(G,v)}$ and then determine the
corresponding set $\Delta(\nu)$ of decorated transverse types
\[
(\bomega_\inc,\bomega_0,\bomega_1,\ldots,\bomega_l)
\]
for $\nu$, with $\bomega_\bullet$ a decorated type marked by $\tau_\bullet$.
Lemma~\ref{Lem: general types scattering broken lines} below shows that for a
certain choice of $\nu$ there is a unique decorated transverse type with
$\bomega_\inc=\btau_\inc$, $\bomega_i=\btau_i$ for all $i$, and with
\begin{equation}
\label{Eqn: omega_0}
\bomega_0=(\hat G_0,\hat\bsigma_0,\hat\bu_0,\hat\bA_0)=(\omega_0,\hat\bA_0)
\end{equation}
a certain maximal decorated global type marked by $\btau_0$.
\begin{figure}
\begin{center}
\input{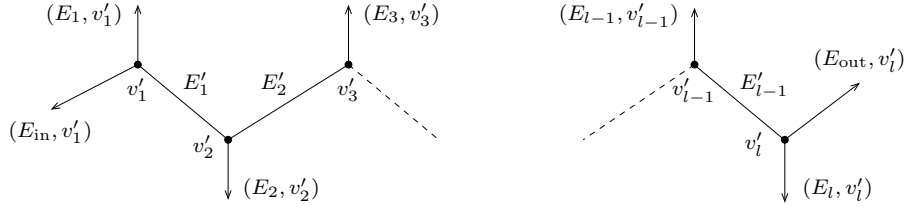}
\caption{\small The tropical type $\omega_0$ (codimension zero case).}
\label{Fig: omega_0}
\end{center}
\end{figure}
If $l=0$ we take $\bomega_0=\btau_0$. To define $\bomega_0$ for $l\ge 1$, assume
first $\sigma_\fop=\sigma_\inc$, that is, we are in the situation of codimension
zero. Take for $\hat G_0$ the graph with $l$ trivalent vertices
$v'_1,\ldots,v'_l$ connected in the given order by $l-1$ edges
$E'_1,\ldots,E'_{l-1}$ to form a chain, and with legs distributed as
\[
(E_\inc,v'_1),\quad(E_\out,v'_l),\quad (E_i,v'_i),\, i=1,\ldots,l,
\]
see Figure \ref{Fig: omega_0}.
The notation $E_\inc$, $E_i$, $E_\out$ indicates that we identify these legs
with the legs of $\tau_0$ via the contraction morphism $\omega_0\arr\tau_0$
contracting all edges of $\omega_0$. All strata are given by $\sigma_\inc$ and
the contact orders for each edge are determined inductively from the contact
orders $\bu_0$ of $\tau_0$ by enforcing the balancing condition
Lemma~\ref{lem:balancing}. Necessarily $\hat\bA_0=0$.

If $\sigma_\fop$ is of codimension one, we further split $v'_l$ into two
vertices $v'_l$ and $v'_\out$ connected by another edge $E'_l$, with $\hat
\bsigma_0(v'_\out)=\sigma_\fop$, and the legs at $v'_l$ and $v'_\out$
distributed as $(E_l,v'_l)$, $(E_\out,v'_\out)$. For $l=1$ the leg
$(E_\inc,v'_1)=(E_\inc,v'_l)$ is also adjacent to $v'_l$. The other strata are
forced to be one of the two maximal cells $\sigma_\inc=\bsigma(E_\inc,v)$ or
$\bsigma(L_\out)$ containing $\sigma_\fop$, namely $\hat\bsigma_0
(E_\out,v'_\out) =\bsigma(L_\out)$ and all other strata equal to $\sigma_\inc$.
The curve classes are determined by the contraction morphism as
\[
\hat\bA(v'_\out)=\bA(v)=d\cdot [X_{\sigma_\fop}]
\]
and all other curves classes trivial. Note that in any
case the curve classes are determined by the contraction morphism. It is
therefore enough to drop the decoration in the following classification of
transverse types according to Definition~\ref{Def: Delta(nu)}.

\begin{lemma}
\label{Lem: general types scattering broken lines}
There exists a displacement vector $\nu=(\nu_E)_E\in\prod_{E\in E(G)}
\Lambda_\RR$ that is general for $\tau$ such that any $\nu$-transverse type
$\omega$ for $\tau$ (Definition~\ref{Def: Delta(nu)},2) with
$\scrM(X,\omega'_0)\neq \emptyset$, for $\omega'_0$ the part of $\omega$ marked
by $\tau_0$, is isomorphic to
\[
\omega=(\tau_\inc,\tau_1,\ldots,\tau_l,\omega_0),
\]
where $\omega_0$ is as described in \eqref{Eqn: omega_0}.
\end{lemma}

\begin{proof}
Take the quotient $\delta:\Lambda\arr\ZZ$ by
$\Lambda_\fop$ to evaluate non-negatively on vectors pointing from $\fop$ to the
incoming direction, and let $m\in\Lambda$ be a vector with $\delta(m)=1$. Then in
particular $\Lambda=\Lambda_\fop\oplus\ZZ m$. Choose $\nu_1,\ldots,\nu_l\in\RR$
with
\[
\nu_1>\ldots> \nu_l>0,
\]
and define
\[
\nu_{E_\inc}=0,\quad
\nu_{E_i}= \nu_i\cdot m,\quad i=1,\ldots,l.
\]
Thus $\nu_{E_i}=(0,\nu_i)$ when identifying $\Lambda$ with
$\Lambda_\fop\oplus\ZZ$. Following the interpretation in Remark~\ref{Rem:
interpretation of Delta(nu)} of $\Delta(\nu)$, we have to show that any
$\nu$-broken tropical punctured map marked by $\tau$ fulfilling
$\scrM(X,\omega'_0)\neq\emptyset$ has components of types
\[
\omega_\inc=\tau_\inc,\ \omega_1=\tau_1,\ldots,\omega_l=\tau_l,\ \omega_0.
\]
We first consider the component marked by $\tau_0$. Denote this component by
$\omega'_0=(G'_0,\bsigma'_0,\bu'_0)$ and by
\[
(E_\inc,v'_\inc),\ (E_\out,v'_\out),\ (E_1,v'_1),\ldots,(E_l,v'_l)
\]
the legs, with some of the vertices $v'_\inc,v'_1,\ldots,v'_l,v'_\out$ possibly
coinciding. Denote further by $E'_1,\ldots,E'_r\in E(G'_0)$ the unique sequence
of edges connecting $v'_\inc$ to $v'_\out$, oriented away from $v'_\inc$. Note
also that $d=-\delta(u_\inc)$ for $d$ as introduced in \eqref{Eqn: constant d}
and that $E_\inc\in E(G)$ is oriented away from $v$. Now since
$\scrM(X,\omega'_0) \neq\emptyset$, Lemma~\ref{lem:balancing} shows that
$\omega'_0$ fulfills the balancing condition. Thus $u_1,\ldots,
u_l\in\Lambda_\fop=\ker(\delta)$ implies that the images under $\delta$ of the
contact orders of all edges are determined as follows:
\[
-\delta(\bu'_0(E))=\begin{cases} d,&E=E'_1,\ldots,E'_r\\
0,&\text{otherwise}.\end{cases}
\]
Moreover, by the perturbed matching condition~\eqref{Eqn: perturbed matching},
the fact that the outgoing vertex for each $\tau_i$ maps to $\fop$ and the
choice of $\nu_1,\ldots,\nu_l$, it follows that $\delta$ is constant on the
subtree connecting the leg $(E_i,v_i)$ to the chain of edges $E'_1,\ldots,E'_r$,
with value $\nu_i$. Since all $\nu_i$ are disjoint, it follows that all of these
subtrees are disjoint. Note that there are $r+1$ vertices on the
path connecting $v'_{\inc}$ and $v'_{\out}$ (including these latter two
vertices). Thus necessarily $r+1 \ge l$. However, if $\dim\sigma_{\fop}=n-1$,
$\bsigma_0'(E_{\inc},v'_{\inc})$ and $\bsigma_0'(E_{\out},v'_{\out})$ are distinct
maximal cones containing $\sigma_{\fop}$, and thus there must be
at least one vertex $v_j'$ on this path contained in $\sigma_{\fop}$. By the
choice of displacement vectors, none of the subtrees connecting
an edge $(E_i,v_i)$ to the chain of edges $E'_1,\ldots,E'_r$ is connected
to $v_j'$. Thus in this case, $r \ge l$.

Now by the dimension formula \eqref{Eqn: dimension formula splitting} we require
\begin{equation}
\label{Eqn: dimension comparison}
\textstyle
\dim\widetilde\omega_\inc+\dim\widetilde\omega'_0
+\sum_{i=1}^l\dim\widetilde \omega_i =
\dim \widetilde\tau+ \sum_E\rk\Lambda= (n+l)+(l+1)\cdot n= ln+2n+l.
\end{equation}
Note that if $\dim \sigma_{\fop}=n$, then $\dim\omega_0' \ge n+r$,
as there is an $n$-dimensional choice of location for one vertex,
and an additional choice of length for the edges $E_1',\ldots,E_r'$.
Thus, in this first case, $\dim\widetilde\omega_0' \ge n+r+l+1$.
On the other hand, if $\dim\sigma_{\fop}=n-1$, then
$\dim\omega_0'\ge n-1+r$, as there is an $(n-1)$-dimensional choice
of location for the vertex $v_j'$. Thus, in this second case,
$\dim\widetilde\omega_0' \ge n+r+l$.
The left-hand side of \eqref{Eqn: dimension comparison} now
has dimension at least
\[
\begin{cases}
n+(n+l+r+1)+ l\cdot(n-1)=ln+2n+r+1, & \dim \sigma_{\fop} =n\\
n+(n+l+r)+ l\cdot(n-1)=ln+2n+r,&\dim \sigma_{\fop}=n-1.
\end{cases}
\]
Note in both cases, this quantity is then $\ge ln+2n+l$,
with equality with the right-hand side of 
\eqref{Eqn: dimension comparison} precisely if $r+1=l$ or $r=l$ in the
two cases, 
$\omega_\inc=\tau_\inc$, $\omega_i=\tau_i$, $i=1,\ldots,l$, and the mentioned
subtrees in $\omega'_0$ trivial. Comparing with the right-hand side of
\eqref{Eqn: dimension comparison} shows that we indeed have equality. In
particular, $G'_0$ agrees with the graph $G_0$ in \eqref{Eqn: omega_0}. Then
also $\bsigma'_0=\hat\bsigma_0$ since each vertex, except $v'_\out$ in the
codimension one case, maps to the maximal cell $\bsigma(E_\inc)$. Furthermore,
the contact orders $\bu'_0$ agree with $\hat\bu_0$ due to the balancing
condition.

Taken together this shows $\omega'_0=\omega_0$, and hence
$(\tau_\inc,\tau_1,\ldots,\tau_l,\omega_0)$ is indeed the only transverse type
for $\nu$.
\end{proof}

It remains to compute the multiplicity $m(\bomega)=m(\omega)$ occurring in
Theorem~\ref{Thm: gluing theorem}.

\begin{lemma}
\label{Lem: multiplicity computation}
The multiplicity according to Definition~\ref{Def: Delta(nu)},(3)
for the single element $\omega=(\tau_\inc,\tau_1,\ldots,\tau_l,\omega_0)$ of
$\Delta(\nu)$ in Lemma~\ref{Lem: general types scattering broken lines} equals
\[
m(\omega)=
\begin{cases}
k_\tau^{-1}k_{\tau_\inc}\prod_{i=1}^l(dk_{\tau_i}),& \dim\sigma_\fop=n\\
d k_\tau^{-1}k_{\tau_\inc}\prod_{i=1}^l(dk_{\tau_i}),& \dim\sigma_\fop=n-1.
\end{cases}
\]
\end{lemma}

\begin{proof}
The multiplicity in question is the index of $\im(\Phi^\gp)$, where
\[
\textstyle
\Phi:(\widetilde\tau_\inc)_\ZZ \times(\widetilde\omega_0)_\ZZ\times \prod_{i=1}^l (\widetilde\tau_i)_\ZZ\arr \Lambda\times\Lambda^l
\]
is the map describing the matching at the gluing edges, denoted
$\varepsilon_{\omega_0}$ in \eqref{Eqn: prod_E omega_E -> Lambda_E}. In
particular, $\Phi^{-1}(0)=\widetilde\tau$ is the enlarged basic cone for
the glued type $\tau$. For the case of a trivial incoming broken line type
define $\widetilde\tau_\inc$ as the maximal cone containing the incoming vector
and $\Phi$ on this factor as the inclusion to the first copy of $\Lambda$.

We first treat the case $\dim\sigma_\fop=n$. In this case, by the description of
the type $\omega_0$, we can write
\begin{equation}
\label{Eqn: tilde omega_0^gp}
(\widetilde\omega_0^\gp)_\ZZ = \Lambda\times\ZZ^l\times\ZZ\times\ZZ^{l-1}=
\Lambda\times\ZZ^l\times\ZZ^l,
\end{equation}
with the $\ZZ$-factors holding the lengths
\[
\ell_1,\ldots,\ell_l,\ell_\inc,\ell'_1,\ldots,\ell'_{l-1}
\]
of the legs $(E_1,v'_1),\ldots, (E_l,v'_l)$, $(E_\inc,v'_1)$ and edges
$E'_1,\ldots,E'_{l-1}$, in this order, while the $\Lambda$-factor records the
image of the point on the incoming leg. For uniformity of notation write
$E'_0=(E_\inc,v'_1)$, $u'_0= u_\inc$, $\ell'_0=\ell_\inc$ and
\[
u'_i= \hat\bu_0(E'_i)=u_\inc+u_1+\cdots+u_i,\quad i=1,\ldots,l-1.
\]
Thus the image of $v'_i\in V(\hat G_0)$ under the tropical punctured map of type
$\omega_0$ defined by
\[
(m,\ell_1,\ldots,\ell_l,\ell'_0,\ldots,\ell'_{l-1})\in \Lambda\times\ZZ^l\times\ZZ^l
\]
equals
\[
m+\ell'_0 u'_0+\ell'_1 u'_1+\ldots+ \ell'_{i-1} u'_{i-1}.
\]
Denote further by $\ev_\inc: (\widetilde\tau_\inc^\gp)_{\ZZ}\to\Lambda$ and
$\ev_i: (\widetilde\tau_i^\gp)_{\ZZ}\to \Lambda_\fop\subset\Lambda$ the linear
extensions of the evaluation maps at the point on the respective unique leg.

We first compute the cokernel $Q$ of the restriction of $\Phi^\gp$ to
\[
(\widetilde\tau_\inc)^\gp \times(\{0\}\times\ZZ^l)\times\prod_{i=1}^l (\widetilde\tau_i)^\gp.
\]
The further restriction to the subspace spanned by the last copy of $\ZZ$ in \eqref{Eqn: tilde omega_0^gp} and $(\widetilde\tau_l)^\gp$ leads to the map
\[
\varepsilon_l: \ZZ\times(\widetilde\tau_l^\gp)_{\ZZ}\arr \Lambda,\quad
(\ell'_{l-1}, h_l)\longmapsto \ev_l(h_l)-\ell'_{l-1}\cdot u'_{l-1}
\]
to the last copy of $\Lambda$ in the codomain of $\Phi^\gp$.
Now since $|\delta(u'_{l-1})|=d$ we can fit $\varepsilon_l$ 
into the diagram with exact rows and injective columns
\begin{equation}
\label{Eqn: index in Lambda versus in Lambda_frp}
\vcenter{
\xymatrix{
0\ar[r]&(\widetilde\tau_l^\gp)_{\ZZ}\ar[r]\ar[d]^{\ev_l}&
\ZZ\times(\widetilde\tau_l^\gp)_{\ZZ}\ar[r]\ar[d]^{\varepsilon_l}&
\ZZ\ar[r]\ar[d]^{\cdot d}&0\\
0\ar[r]&\Lambda_\fop\ar[r]&\Lambda\ar[r]&\ZZ\ar[r]&0.
}}
\end{equation}
The cokernel of $\ev_l$ having order $k_{\tau_l}$, the snake lemma implies
$|\coker(\varepsilon_l)|=d\cdot k_{\tau_l}$. An induction on $l$ thus shows that
$|Q|$ equals $\prod_{i=1}^l \big(d k_{\tau_i}\big)$ times the order of the
cokernel of the case $l=0$, which is $|\coker(\ev_\inc)|= k_{\tau_\inc}$. We have thus shown
\begin{equation}
\label{Eqn: |Q|}
\textstyle
\big|Q\big| = k_{\tau_\inc}\prod_{i=1}^l(dk_{\tau_i}).
\end{equation}

To finish the computation of $\big|\coker\Phi^\gp\big|$, consider the following
diagram with exact rows and columns and with the obvious maps.
\[
\small
\xymatrix{
&&0\ar[d]&0\ar[d]&\ar@{-->}[d]\\
&0\ar[r]\ar[d]& (\widetilde\tau^\gp_\inc)_{\ZZ} \times \{0\}\times \ZZ^l\times
\prod_{i=1}^l (\widetilde\tau^\gp_i)_{\ZZ}\ar[r]\ar[r]\ar[d]
&\Lambda\times\Lambda^l\ar[r]\ar[d]& Q\ar[r]\ar[d]& 0\\
0\ar[r]&\widetilde\tau^\gp_{\ZZ}\ar[r]\ar[d]^\alpha
&(\widetilde\tau^\gp_\inc)_{\ZZ} \times(\widetilde\omega^\gp_0)_{\ZZ}
\times \prod_{i=1}^l (\widetilde\tau^\gp_i)_{\ZZ}\ar[r]^(.7){\Phi^\gp}\ar[d]
&\Lambda\times\Lambda^l\ar[r]\ar[d]& \coker(\Phi^\gp)\ar[r]\ar[d]& 0\\
0\ar[r]&\Lambda\times\ZZ^l\ar[r]\ar@{--}[d]&\Lambda\times\ZZ^l\ar[r]\ar[d]
&0\ar[r]\ar[d]&0\\
&&0&0}
\]
The snake lemma yields the exact sequence
\begin{equation}
\label{Eqn: coker Phi^gp sequence}
0 \arr\widetilde\tau^\gp_{\ZZ}\stackrel{\alpha}{\longrightarrow}
\Lambda\times\ZZ^l\arr Q\arr \coker(\Phi^\gp)\arr 0.
\end{equation}
Note also that $\widetilde\tau^\gp_{\ZZ}=\tau^\gp_{\ZZ}\times\ZZ\times\ZZ^l$ with the $\ZZ$-factors from the additional points at the $l+1$ gluing edges,
and the map $\alpha$ is a product of the injection of lattices
\[
\bar\alpha:\tau^\gp_{\ZZ}\times\ZZ\arr\Lambda,\quad
(h,\lambda)\longmapsto V(h)-\lambda\cdot u_\inc
\]
and $\id_{\ZZ^l}$. Here $V:\tau^\gp_{\ZZ}
\arr\Lambda_\fop$ is the evaluation at the
outgoing vertex. Replacing $u_\inc$ in this formula by $u_\out$ yields
$\ev_\out:\tau^\gp_{\ZZ}\times\ZZ\arr \Lambda$ with index of the image defining
$k_\tau$. But $|\delta(u_\out)|=|\delta(u_\inc)|=d$. By considering two diagrams
similar to \eqref{Eqn: index in Lambda versus in Lambda_frp}, it follows that
\[
|\coker(\alpha)| =|\coker(\bar\alpha)|=|\coker(\ev_\out)|=k_\tau.
\]
Together with \eqref{Eqn: |Q|} the exact sequence \eqref{Eqn: coker Phi^gp
sequence} thus yields
\[
\textstyle
m(\omega)= \big|\coker(\Phi^\gp)\big|=
k_\tau^{-1}\cdot |Q|= k_\tau^{-1}k_{\tau_\inc}\prod_{i=1}^l (d k_{\tau_i}),
\]
as claimed.

If $\dim\sigma_\fop=n-1$ we have an additional vertex $v'_\out$ mapping to
$\sigma_\fop$ and an additional edge $E'_l$ connecting $v'_l$ to $v'_\out$, with contact order $u'_l$. Thus
in this case we have
\[
(\widetilde\omega_0)^\gp_\ZZ= \Lambda_\fop\times\ZZ^l\times\ZZ\times\ZZ^l,
\]
with $\Lambda_\fop$ recording the position of $v'_\out$ in $\fop$.
We can now exhibit $(\widetilde\omega_0)^\gp_\ZZ$ as a sublattice of $\Lambda\times\ZZ^l\times\ZZ^l$ from~\eqref{Eqn: tilde omega_0^gp}
by the map
\begin{eqnarray*}
\Psi:\Lambda_\fop\times\ZZ^l\times\ZZ^{l+1}&\lra& \Lambda\times\ZZ^l\times\ZZ^l,\\
(m,\ell_1,\ldots,\ell_l,\ell'_0,\ldots,\ell'_l)&\longmapsto&
\textstyle
\big(m- \sum_{i=0}^l\ell'_i u'_i,\ell_1,\ldots,\ell_l,\ell'_0,..,\ell'_{l-1}\big).
\end{eqnarray*}
Here $u'_0=u_\inc$, $\ell'_0=\ell_\inc$ are defined analogously to the
codimension zero case. Since the images of $u'_i$ in $\Lambda/\Lambda_\fop\simeq
\ZZ$ generate $d\ZZ$, we see that $\Psi$ is an inclusion of lattices of index
$d$. The stated formula in codimension one now follows by observing that
$\Phi^\gp$ is the composition of $\Phi^\gp$ in the codimension zero case with
$\id_{(\tilde\tau_\inc)_\ZZ}\times \Psi \times \id_{\prod_i
(\tilde\tau_i)_\ZZ}$.
\end{proof}
\medskip

Finally we need to compute one simple punctured invariant.

\begin{lemma}
\label{Lem: vertex contribution}
\[
\deg \big[\scrM(X,\bomega_0)\big|^\virt=\begin{cases}1,& \codim\sigma_\fop=n\\
1/d,&\codim\sigma_\fop=n-1.\end{cases}
\]
\end{lemma}

\begin{proof}
The graph $\hat G_0$ given by $\bomega_0$ is a chain with $l$ or $l+1$ trivalent
vertices of genus~$0$, depending if $\dim\sigma_\fop=n$ or
$\dim\sigma_\fop=n-1$. In any case, $l$ vertices are marked by $\sigma_\inc$,
the maximal cell containing the image of $E_\inc$. Thus if $\dim\sigma_\fop=n$,
the underlying stable maps in $\scrM(X,\bomega_0)$ are contracted trees of $l$
copies of $\PP^1$ with three special points each, with $l-1$ pairs of such
special points identified at the nodes. With the labelling of the $l+2$
remaining points, this is a rigid curve with trivial automorphism
group. Moreover, the pull-back of $\Theta_X$
is trivial. Hence $\scrM(X,\bomega_0)$ is a reduced point, giving the result in
the case $\dim\Lambda_\fop=n$.

In the codimension one case, $\fop$ is contained in the $(n-1)$-cell
$\rho=\sigma_\fop$ with stratum $X_\rho\subset X$ isomorphic to $\PP^1$ and
$\cM_X|_{X_\rho}$ admitting a lift of the torus action. In particular, it holds
$\Theta_X|_{X_\rho} \simeq\cO_{\PP^1}^n$. The incoming and outgoing legs of
$G_0$ are labelled by the maximal cells $\sigma_\inc,
\bsigma(L_\out)\in\Sigma(X)$ containing $\rho$. By the description of
$\bomega_0$, only the outgoing vertex maps to $\rho$, all others map to
$\sigma_\inc$. Thus the stable maps underlying objects in $\scrM(X,\bomega_0)$
are contracted trees of $l$ copies of $\PP^1$ with three special points each as
before, joined at the outgoing point with another copy of $\PP^1$ mapping to
$X_\rho\simeq\PP^1$. The tropical picture shows that the map $\PP^1\to\PP^1$ is
totally branched at the two zero-dimensional strata $X_{\sigma_\inc},
X_{\bsigma(L_\out)}\subset X_\rho$ with branching order $d=|\delta(u_\inc)|$ the
constant from \eqref{Eqn: constant d}. Hence the associated curve class is
$d\cdot [X_\rho]$ (as follows also from Lemma~\ref{lem:admissible degrees}).
Thus there is a unique underlying stable map over $\ul{W}=\Spec\kk$.
Arguing similarly as in \cite[Claim~3.22]{Assoc}, there is then
a unique enhancement, up to isomorphism, of this
underlying stable map to a basic punctured map $f:C^{\circ}/W\rightarrow X$ 
of type $\bomega_0$
with $\ul{W}=\Spec\kk$.
Moreover, by the triviality of $\Theta_X|_{X_\rho}$ this is again a
logarithmically unobstructed and rigid situation, now with automorphism group
$\mu_d$. Hence $\scrM(X,\bomega_0)$ is the quotient of a reduced point by
$\mu_d$, proving the stated result.
\end{proof}

\begin{proof}[Proof of \eqref{Eqn: scattering equation}]
For applying the gluing formula we now use the labelling with pairwise different $\tau_i$ and multiplicities $\mu_i$. Let
\[
\delta: \scrM(X,\btau)\arr \scrM(X,\btau_\inc)\times\scrM(X,\btau_0)\times
\prod_{i=1}^l\scrM(X,\btau_i)^{\mu_i}
\]
be the splitting map from \eqref{Eqn: splitting map} for $\btau$. Applying
Theorem~\ref{Thm: gluing theorem} with the general displacement vector $\nu$
from Lemma~\ref{Lem: general types scattering broken lines} and noting that $\Aut(\bomega/\btau)=\{1\}$ yields
\begin{equation}
\label{Eqn: vfc form of broken line bending}
\delta_*\big[\scrM(X,\btau)\big]^\virt=
m(\omega)\cdot \big[\scrM(X,\btau_\inc)^\virt\big]\times
(j_{\bomega_0})_*[\scrM(X,\bomega_0)]^\virt\times\prod_{i=1}^l \big[\scrM(X,\btau_i)^{\mu_i}\big]^\virt.
\end{equation}
Next observe
\[
\Aut(\btau)\simeq \Aut(\btau_\inc)\times\prod_{i=1}^l\big(\Aut(\btau_i)\rtimes S_{\mu_i}),
\]
with the symmetric group $S_{\mu_i}$ acting by permutation of the $\mu_i$ copies
of $\btau_i$ attached at $v$. Taking degrees in \eqref{Eqn: vfc form of broken
line bending} and applying Lemma~\ref{Lem: vertex contribution}, plugging in
$m(\omega)$ from Lemma~\ref{Lem: multiplicity computation}, multiplying by
$k_{\tau}$, and dividing by $|\Aut(\btau)|$ now leads to
\[
\begin{aligned}
k_\tau N_\btau
&= k_{\tau_\inc}\cdot{\textstyle \prod_{i=1}^l (d k_{\tau_i})^{\mu_i}}
\cdot \frac{\deg[\scrM(X,\btau_\inc)]^\virt}{|\Aut(\btau_\inc)|}
\cdot \prod_{i=1}^l\frac{\deg[\scrM(X,\btau_i)^{\mu_i}]^\virt}{
|\Aut(\btau_i)\rtimes S_{\mu_i}|}\\
&= k_{\tau_\inc} N_{\tau_\inc} \cdot\prod_{i=1}^l
\frac{(d k_{\tau_i}W_{\btau_i})^{\mu_i}}{\mu_i!}.
\end{aligned}
\]
The total curve class and contact orders are computed from the decorations as
\[
A=A_\inc+d[X_{\sigma_\fop}]+\sum_{i=1}^l \mu_i A_i,
\]
and
\[
u_\out= u_\inc+\sum_{i=1}^l \mu_i u_i.
\]
The last three equalities together imply the claimed
scattering equation~\eqref{Eqn: scattering equation}.
\end{proof}


\section{Consistency of the canonical wall structure}
\label{sec:main theorem}

We now prove the main result of the paper, namely that
$\scrS_{\can}$ is a consistent wall structure. The actual
definition of consistency, while summarized in \cite[Def.\ 3.9]{Theta},
is a bit complicated. However, the main point of the definition is
to enable the construction of so-called theta functions. In other
words, consistency guarantees that the local description
$\vartheta_p(x)$ for theta functions
patch together to give global functions on 
a scheme constructed from the wall structure.

Here, we proceed in the opposite direction. We will first show
that the $\vartheta_p(x)$ satisfy the necessary patching criteria, 
and then show this patching implies consistency as defined in 
\cite[Def.\ 3.9]{Theta}.

\subsection{Patching of theta functions}

As the discussion from now on largely follows the notions of
\cite{Theta}, we now adopt the conventions and notation of
that paper with regards to wall structures. In particular, we
now work with a wall structure $\scrS$ satisfying precisely the
properties of \cite[Def.\ 2.11,2]{Theta}, by assuming that the
walls of $\scrS$ comprise the codimension one cells of a polyhedral 
cone
refinement $\P_{\scrS}$ of $\P$. As mentioned in Remark~\ref{rem:differences},
this may be achieved for $\scrS_{\can}$ by subdividing walls, combining
walls with the same support, and adding walls with attached function $1$
if necessary. In particular, as in \cite{Theta}, we use the notation
$\fou$ for a maximal cell of $\P_{\scrS}$ and $\fob$ for a codimension
one cell of $\P_{\scrS}$ contained in $\rho\in\P^{[n-1]}$. We call
$\fou$ a \emph{chamber} and $\fob$ a \emph{slab} or
\emph{codimension one wall}, while we call
those walls intersecting the interior of a maximal cell a
\emph{codimension zero wall}. A \emph{joint} $\foj$ is a
codimension two cell in $\scrP_{\scrS}$, and the codimension of
$\foj$ is the codimension of the smallest cell of $\P$ containing
$\foj$. 

Recalling further notation from \cite{Theta}, if $\fou$ is a chamber
contained in $\sigma\subseteq\P_{\max}$, we set
\[
R_{\fou}=R_{\sigma}:=\kk[\shP_x^+]/I_x = (\kk[Q]/I)[\Lambda_{\sigma}]
\]
for any point $x\in \Int(\fou)$. 

Given a codimension zero wall $\fop$ separating chambers $\fou,\fou'$,
we obtain a $\kk[Q]/I$-algebra homomorphism
\begin{equation}
\label{eq:theta fop def}
\theta_{\fop}:R_{\fou}\rightarrow R_{\fou'}, \quad z^m \mapsto
f_{\fop}^{\langle n_{\fop}, m \rangle} z^m,
\end{equation}
where $n_{\fop}\in\Hom(\Lambda_{\sigma},\ZZ)$ is the 
primitive normal vector
to $\fop$ positive on $\fou$. Since $f_{\fop}\equiv 1 \bmod \fom$
and $\sqrt{I}=\fom$, $f_{\fop}$ is in fact invertible so that $\theta_{\fop}$
is an automorphism.

For a slab $\fob$, after choosing
$\xi$ as chosen for the isomorphism \eqref{eq:P+x alternate},
we define
\[
R_{\fob}:=
(\kk[Q]/I)[\Lambda_{\rho}][Z_+,Z_-]/(Z_+Z_--f_{\fob}t^{\kappa_{\rho}}),
\]
where $f_{\fob}$ is the function attached to $\fob$.
If $\fou,\fou'$ are the adjacent chambers to $\fob$, there 
are natural localization maps at $Z_+$ and $Z_-$ respectively 
\begin{equation}
\label{eq:localization}
\chi_{\fob,\fou}:R_{\fob}\rightarrow R_{\fou},
\quad
\chi_{\fob,\fou'}:R_{\fob}\rightarrow R_{\fou'},
\end{equation}
given by
\[
\chi_{\fob,\fou}(t^Az^mZ_+^aZ_-^b)
=t^{A+b\kappa_{\rho}}z^{m+(a-b)\xi}f_{\fob}^{b},
\quad
\chi_{\fob,\fou'}(t^Az^mZ_+^aZ_-^b)
=t^{A+a\kappa_{\rho}}z^{m+(a-b)\xi}f_{\fob}^{a},
\]
where $A\in Q$ and $m\in\Lambda_\rho$.
Note that these maps differ from those given in 
\eqref{eq:transport1}, \eqref{eq:transport2} only in the powers of
$f_{\fob}$ appearing here.

We may now make precise the notion that the $\vartheta_p(x)$ patch.

\begin{theorem} 
\label{thm:patching}
For the canonical wall structure $\scrS_{\can}$ and any $p\in B(\ZZ)
\setminus \{0\}$, we have
\begin{enumerate}
\item If $x,x'\in \Int(\fou)$ for a chamber $\fou$, then
\[
\vartheta_p(x)=\vartheta_p(x').
\]
\item If $\fop$ is a codimension zero wall separating chambers
$\fou,\fou'$, and $x\in\Int(\fou)$, $x'\in\Int(\fou')$, then
\[
\vartheta_p(x')=\theta_{\fop}\big(\vartheta_p(x)\big).
\]
\item
If $\fob$ is a slab with adjacent chambers $\fou,\fou'$,
$x\in \Int(\fou)$, $x'\in\Int(\fou')$, then
there exists a unique $\vartheta_p(\fob) \in R_{\fob}$ such that
\[
\chi_{\fob,\fou}(\vartheta_p(\fob))=\vartheta_p(x),
\quad
\chi_{\fob,\fou'}(\vartheta_p(\fob))=\vartheta_p(x').
\]
\end{enumerate}
\end{theorem}

\begin{proof}
{\bf Step 1.} \emph{Establishing the setup.}
We fix $p\in B(\ZZ)\setminus\{0\}$ once and for all. 
We work throughout
with $\vartheta^{\log}_p(x)$ using Theorem~\ref{thm:theta correspondence}.
Let $\scrB$ be the set of isomorphism classes of decorated broken line types
given as
\[
\scrB:=\{\btau=(\tau,\mathbf{A})\,|\,\hbox{$p_{\tau}=p$, the total curve class of
$\mathbf{A}$ lies in $Q\setminus I$ and $N_{\btau}\not=0$}\}.
\]
$\scrB$ is a finite set by Lemma~\ref{lem:theta finite}.
Thus we may choose a polyhedral cone complex $\P_{\scrB}$
refining $\P_{\scrS}$ with the property that for each $\btau\in\scrB$,
$h(\btau_{\out})$ is a union of maximal cells 
$\fou\in\P_{\scrB}$. At the same time, we refine the wall structure
$\scrS$ so that we may take $\P_{\scrS}=\P_{\scrB}$.

It is immediate from the definition of the log theta
functions that if $\fou\in\P_{\scrB}$ is a maximal cone then 
$\vartheta^{\log}_p(x)=\vartheta^{\log}_p(x')$ for any $x,x'\in
\Int(\fou)$. Thus, it is now sufficient to verify that
if $\fou,\fou'\in \scrP_{\scrB}$ are two chambers
separated by a codimension one cell $\fop$, then statements
(2) or (3) of the theorem
hold, depending on whether $\fop$ is a codimension zero wall or a slab.

\medskip

{\bf Step 2.} \emph{Broken lines transversal to $\fop$.}
For $y\in\Int(\fop)$, let $n_{\fop}\in\check\Lambda_y$ be a primitive
normal vector to $\fop$, positive on $\fou$. We may then decompose,
for $x\in\fou$, $x'\in\fou'$,
\begin{align*}
\vartheta^{\log}_p(x) = {} & \vartheta_++\vartheta_-+\vartheta_0\\
\vartheta^{\log}_p(x') = {} & \vartheta'_++\vartheta'_-+\vartheta'_0
\end{align*}
where $\vartheta_+$ (resp.\ $\vartheta_-$, $\vartheta_0$) consists
of a sum of those monomials $a t^A z^m$ appearing in
$\vartheta^{\log}_p(x)$ with $\langle n_{\fop},m\rangle >0$
(resp.\ $\langle n_{\fop},m\rangle<0$ and $\langle n_{\fop},m\rangle=0$).
The terms $\vartheta'_{\pm}$, $\vartheta'_0$ are defined analogously.

Using the broken line description of these theta functions, the
now standard argument of the proof of \cite[Thm.~4.12]{GrossP2} 
or \cite[Lem.~4.9]{CPS},
shows immediately that 
if $\fop$ is a codimension zero wall, then
\begin{equation}
\label{eq:theta plus change}
\theta_{\fop}(\vartheta_+)=\vartheta_+', \quad \theta_{\fop}(\vartheta_-)
=\vartheta_-'.
\end{equation}
We briefly recall the argument. Given a decorated broken line
contributing to $\vartheta_+$ with endpoint in $\fou$ very close
to $\fop$, we may perturb the broken line by moving its endpoint to
a nearby point in $\fop$. We may then add an additional line segment
(or simply extend the final line segment through $\fop$) following
the definition of a decorated broken line to obtain a decorated
broken line with endpoint nearby in $\fou'$. If this is done
in all possible ways, it then follows from the propagation rule
for monomials \eqref{eq:parallel transport for correspondence}
that $\theta_{\fop}(\vartheta_+)=\vartheta_+'$. Finally,
any monomial in $\vartheta_0$ or $\vartheta_0'$ has exponent tangent
to $\fop$, and hence is left invariant under transport from $\fou$ to
$\fou'$. Thus in this case, it is sufficient to show that 
$\vartheta_0=\vartheta_0'$ under transport $\fot_{\sigma,\sigma'}$.

If $\fop=\fob$ is a slab contained in $\rho$, $\fou\subseteq\sigma$,
$\fou'\subseteq\sigma'$, then we may write
\[
\vartheta_+= \sum_i a_iz^{m_i+\alpha_i\xi}, \quad \vartheta_-'=
\sum_i a_i' z^{m_i'-\alpha_i'\xi}
\]
with $a_i, a_i' \in \kk[Q]/I$, $m_i,m_i'\in\Lambda_{\rho}$ and
$\alpha_i,\alpha_i'>0$.
We may then define 
\[
\vartheta_+(\fob):=\sum_{i} a_iz^{m_i}Z_+^{\alpha_i},\quad\quad
\vartheta_-(\fob):=\sum_{i} a'_iz^{m'_i}Z_-^{\alpha'_i}.
\]
By construction, 
$\chi_{\fob,\fou}(\vartheta_+(\fob))=\vartheta_+$
and
$\chi_{\fob,\fou'}(\vartheta_-(\fob))=\vartheta_-'$.
On the other hand, it follows from exactly the same
argument as outlined above and the formulas for 
$\chi_{\fob,\fou'}(Z_+)$ and $\chi_{\fob,\fou}(Z_-)$ of 
\eqref{eq:localization} that
\begin{equation}
\label{eq:theta plus change slab}
\chi_{\fob,\fou'}(\vartheta_+(\fob))=\vartheta'_+,\quad\quad
\chi_{\fob,\fou}(\vartheta_-(\fob))=\vartheta_-.
\end{equation}

As before, any monomial in $\vartheta_0$ or $\vartheta_0'$ 
is left invariant under transport
from $\fou$ to $\fou'$. Thus again it is sufficient to show that
$\vartheta_0=\vartheta_0'$ under
transport $\fot_{\sigma,\sigma'}$, as the expression
\[
\vartheta_p(\fob):=\vartheta_+(\fob)+\vartheta_-(\fob)+\vartheta_0
\]
then satisfies the desired properties.
\medskip

{\bf Step 3.} \emph{Reduction to a non-virtual linear equivalence
relation on moduli spaces.}
Let $\scrB_{\fou}$ (resp.\ $\scrB_{\fou'}$) denote
the collection of isomorphism classes of
decorated broken line types $\btau$ contributing to
$\vartheta_0$ (resp.\ $\vartheta_0'$) such that $\fou\subseteq h(\tau_{\out}),
\fou'\not\subseteq h(\tau_{\out})$ (resp. $\fou'\subseteq h(\tau_{\out}),
\fou\not\subseteq h(\tau_{\out})$). Any other broken line type 
$\btau$ contributing to either $\vartheta_0$ or $\vartheta_0'$ must then
have $\fou\cup \fou' \subseteq h(\tau_{\out})$, and hence contributes
equally to both $\vartheta_0$ and $\vartheta_0'$. Hence it is sufficient
to show that
\[
\sum_{\btau\in \scrB_{\fou}} k_{\tau} N_{\btau} t^A z^{-u_{\tau}}
=
\sum_{\btau\in \scrB_{\fou'}} k_{\tau} N_{\btau} t^A z^{-u_{\tau}}.
\]

Let $\scrD$ be the set of isomorphism classes of degenerate 
decorated broken line types (see
Definition~\ref{def:broken line type})
$\btau'$ for which there exists a $\btau\in \scrB$ and a contraction
morphism $\btau\rightarrow\btau'$.
Recall that for a degenerate broken line type $\tau'$,
$\dim \tau'=n-2$ and $\dim h_{\tau'}(\tau'_{\out})=n-1$. Let 
\[
\scrD_{\fop}:=\{\btau'\in \scrD\,|\, \fop\subseteq h_{\tau'}(\tau'_{\out})\}.
\]
Note by the defining assumption on $\P_{\scrB}$, if $\btau
\in \scrB\cup\scrD$ is a 
broken line type or a degenerate broken line type and
$\dim\fop\cap h(\tau_{\out})=n-1$, then $\fop\subseteq h(\tau_{\out})$.

For any $\btau\in \scrB_{\fou}\cup\scrB_{\fou'}$, since $u_{\tau}$
is tangent to $\fop$ and $\fop$ is contained in a codimension one
face of $h(\tau_{\out})$, necessarily there is a unique choice
of contraction morphism of decorated types
$\phi:\btau\rightarrow \btau'$ with $\btau'\in \scrD_{\fop}$.
Note further the isomorphism class of $\btau'$ only depends
on the isomorphism class of $\btau$.
This gives maps $\Psi:\scrB_{\fou}\rightarrow \scrD_{\fop}$,
$\Psi':\scrB_{\fou'}\rightarrow \scrD_{\fop}$ taking $\btau$ to the
corresponding degenerate decorated broken line type $\btau'$. 
It is now sufficient
to show that for any $\btau'\in \scrD_{\fop}$, we have
\begin{equation}
\label{eq:equality of numbers}
\sum_{\btau \in \Psi^{-1}(\btau')} k_{\tau} N_{\btau}
=
\sum_{\btau \in (\Psi')^{-1}(\btau')} k_{\tau} N_{\btau}.
\end{equation}

To prove this, we now fix $\btau'\in\scrD_{\fop}$ with
underlying type $\tau'$. We consider the moduli space 
\[
\foM^{\ev}(\shX,\tau'):=\foM(\shX,\tau')\times_{\ul{\shX}}\ul{X},
\]
where the morphism $\foM(\shX,\tau')\rightarrow \ul{\shX}$ is
given by schematic evaluation at the section of the universal curve
corresponding to $L_{\out}$.
For any contraction morphism $\tau\rightarrow\tau'$,
we obtain by \cite[Prop.~5.19]{ACGSII} a Cartesian 
diagram
\begin{equation}
\label{eq:cartesian diagram}
\xymatrix@C=30pt
{
\coprod_{\btau=(\tau,{\bf A})} \scrM(X,\btau)\ar[d]_{\varepsilon_{\tau}}
\ar[r]^>>>>>{\iota_{\tau}'}
&
\scrM(X,\btau')\ar[d]^{\varepsilon_{\tau'}}\\
\foM^{\ev}(\shX,\tau)\ar[r]_{\iota_{\tau}}&\foM^{\ev}(\shX,\tau')
}
\end{equation}
Here the disjoint union is over all decorations $\btau=(\tau,{\bf A)}$ of $\tau$
such that the contraction morphism $\tau\rightarrow\tau'$ induces a contraction
morphism $\btau\rightarrow\btau'$. The maps $\iota_{\tau},\iota'_{\tau}$ are
induced by the contraction morphisms $\tau\rightarrow\tau'$ and
$\btau\rightarrow\btau'$. Note that
\begin{align}
\label{eq:first sum}
\begin{split}
\sum_{\btau\in \Psi^{-1}(\btau')} k_{\tau} N_{\btau}
= {} & 
\sum_{\btau\in \Psi^{-1}(\btau')} k_{\tau} {\deg [\scrM(X,\btau)]^{\virt}
\over |\Aut(\btau)|}
\\
= {} & \sum_{\tau\rightarrow\tau' \atop \btau=(\tau,{\bf A})} k_{\tau} {\deg [\scrM(X,\btau)]^{\virt} \over
|\Aut(\btau)| |\Aut(\tau/\tau')||\Aut(\btau/\btau')|^{-1}}\\
= {} & \sum_{\tau\rightarrow\tau' \atop \btau=(\tau,{\bf A})} k_{\tau} {\deg [\scrM(X,\btau)]^{\virt} \over
|\Aut(\btau')| |\Aut(\tau/\tau')|}.
\end{split}
\end{align}
The last two sums are over (1) isomorphism classes of
types $\tau$ which are the underlying type of some element of
$\Psi^{-1}(\btau')$ and (2) decorations $\btau=(\tau,{\bf A})$
of $\tau$ yielding an induced contraction map $\btau\rightarrow\btau'$.
The factor $|\Aut(\tau/\tau')||\Aut(\btau/\btau')|^{-1}$ arises in
the third summation because 
we are now summing over multiple representatives for an
isomorphism class in $\Psi^{-1}(\btau')$. Indeed, giving
a fixed representative $\tau$ equipped with its unique contraction
morphism $\tau\rightarrow\tau'$,
$\Aut(\tau/\tau')$ now acts on the set of
decorations $\btau=(\tau,\mathbf{A})$ of $\tau$ 
such that $\tau\rightarrow \tau'$ induces a contraction morphism
$\btau\rightarrow\btau'$. The stabilizer of this action is
the subgroup $\Aut(\btau/\btau')\subseteq \Aut(\tau/\tau')$, and hence
each $\btau\in \Psi^{-1}(\tau')$ appears
$|\Aut(\tau/\tau')|/|\Aut(\btau/\btau')|$ times in the last two sums.
For the last equality, we use
$|\Aut(\btau)|=|\Aut(\btau/\btau')||\Aut(\btau')|$. 

Using \eqref{eq:cartesian diagram}, if we fix $\tau\rightarrow\tau'$,
we may write
\[
\sum_{\btau=(\tau,{\bf A})} \deg [\scrM(X,\btau)]^{\virt}=
\deg (\varepsilon_{\tau})^![\foM^{\ev}(\shX,\tau)].
\]
Thus, with 
all sums below being over isomorphism classes of broken line types
$\tau$ which are the underlying type of elements of
$\Psi^{-1}(\btau')$,
the quantity of \eqref{eq:first sum} may now be written as
\begin{align*}
\sum_{\tau} k_{\tau} {\deg \iota'_{\tau,*}(\varepsilon_{\tau})^![\foM^{\ev}(\shX,\tau)]
\over |\Aut(\btau')||\Aut(\tau/\tau')|}
= {} & 
\sum_{\tau} k_{\tau} {\deg \varepsilon_{\tau'}^!\iota_{\tau,*}[\foM^{\ev}(\shX,\tau)]
\over |\Aut(\btau')||\Aut(\tau/\tau')|}\\
= {} &  \sum_{\tau} k_{\tau} { |\Aut(\tau/\tau')|\deg
 \varepsilon_{\tau'}^![\foM^{\ev}_{\tau}(\shX,\tau')]
\over |\Aut(\btau')||\Aut(\tau/\tau')|}\\
= {} &  \sum_{\tau} k_{\tau}
{\deg \varepsilon_{\tau'}^![\foM^{\ev}_{\tau}(\shX,\tau')]
\over |\Aut(\btau')|}.
\end{align*}
Here the first equality follows from \cite[Thm.\ 4.1]{Man},
while the second equality follows from the morphism
$\iota_{\tau}:\foM^{\ev} (\shX,\tau)\rightarrow \foM^{\ev}_{\tau}(\shX,\tau')$
being of degree $|\Aut(\tau/\tau')|$.
Thus it will be sufficient to show the following relation
in the codimension one Chow group of $\foM^{\ev}(\shX,\tau')$
\begin{equation}
\label{eq:equality of classes}
\sum_{\tau\rightarrow\tau'\atop \fou\subseteq h_{\tau}(\tau_{\out})}
k_{\tau}[\foM^{\ev}_{\tau}(\shX,\tau')]
=
\sum_{\tau\rightarrow\tau'\atop \fou'\subseteq h_{\tau}(\tau_{\out})}
k_{\tau}[\foM^{\ev}_{\tau}(\shX,\tau')],
\end{equation}
where the sums are over isomorphism classes of broken line
types $\tau$ with a contraction morphism $\tau\rightarrow\tau'$
and with the stated inclusion.

\medskip

{\bf Step 4.} \emph{Evaluation maps and local structure of moduli spaces.}
We continue with $\btau'\in
\scrD_{\fop}$ fixed as in the previous step. If
$\foC^{\circ}\rightarrow\foM(\shX,\tau')$ is the universal punctured curve over
the moduli space $\foM(\shX,\tau')$, with section $x_{\out}$ corresponding to
the leg $L_{\out}$, we write 
\[
\widetilde\foM(\shX,\tau'):=(\ul{\foM(\shX,\tau')}, x_{\out}^*
\shM_{\foC^{\circ}})^{\sat},
\]
the saturation of the log structure of $\foC^{\circ}$ pulled back via
$x_{\out}$: see \cite[\S5.2]{ACGSII}. 
Denote the reduction by
\[
\overline{\foM}(\shX,\tau'):=\widetilde\foM(\shX,\tau')_{\red}.
\]
By \cite[Prop.\ 5.5]{ACGSII}, the composition
$\overline{\foM}(\shX,\tau')\rightarrow\widetilde\foM(\shX,\tau')
\rightarrow\foM(\shX,\tau')$ induces an isomorphism on underlying
stacks, as $\foM(\shX,\tau')$ is already reduced by 
\cite[Prop.~3.28]{ACGSII}.

By composing the universal morphism
$f:\foC^{\circ}\rightarrow\shX$ with the section $x_{\out}$,
we obtain an evaluation map $\ev:\overline\foM(\shX,\tau')\rightarrow\shX$.
As the former space is reduced, this evaluation map
factors through the stratum $\shX_{\sigma}$ of $\shX$
where $\sigma\in\P$ is the minimal cell containing $\fop$. 

We recall from \cite[Def.\ 3.22]{ACGSII} that
$\foM(\shX,\tau')$ carries an idealized structure given
by a coherent sheaf of monoid ideals $\shI\subseteq \shM_{\foM(\shX,\tau')}$. 
This sheaf may be described by giving the stalks of the
image sheaf $\overline{\shI}$ of $\shI$ in the ghost sheaf
$\overline{\shM}_{\foM(\shX,\tau')}$.
By \cite[Prop.\ 3.23]{ACGSII}, because $\tau'$ is realizable, for
a geometric point $\bar w$ of $\foM(\shX,\tau')$, the monoid ideal
$\overline{\shI}_{\bar w}$
has the following description. Let $Q_{\bar w}$ be the stalk of
the ghost sheaf at $\bar w$, and $Q_{\tau'}$ the stalk of the ghost
sheaf at a generic point of $\foM(\shX,\tau')$. Then there is
a well-defined generization map 
$\chi_{\bar w}:Q_{\bar w}\rightarrow Q_{\tau'}$, and 
$\overline{\shI}_{\bar w}=\chi_{\bar w}^{-1}(Q_{\tau'}\setminus \{0\})$.

Similarly, we may put an idealized structure on
$\overline{\foM}(\shX,\tau')$, with ideal sheaf $\shJ$. 
For a geometric point $\bar w$,
let $Q^{\circ,\sat}_{\bar w}$ be the stalk of the ghost sheaf of
this latter log structure; necessarily, this monoid is
contained in $Q_{\bar w}\oplus\ZZ$
and contains $Q_{\bar w}\oplus\NN$. Then we take the stalk of the corresponding
monoid ideal $\overline{\shJ}_{\bar w}$ 
to be the monoid ideal generated by $\overline{\shI}_{\bar w}\oplus 0$,
$Q^{\circ,\sat}\setminus (Q\oplus\NN)$ and $(0,1)$. 
It is straightforward to check that $\shJ$ yields an idealized structure
on $\overline{\foM}(\shX,\tau')$. Further, the natural morphism
$\overline{\foM}(\shX,\tau')\rightarrow \foM(\shX,\tau')$
is idealized log \'etale.

\cite[Cor.~2.9]{Wu} now shows that the map $\ev:\overline{\foM}(\shX,\tau')
\rightarrow \shX$ is idealized log smooth. As this map factors through the
idealized log \'etale strict closed embedding $\shX_{\sigma}\rightarrow \shX$,
we write $\ev:\overline{\foM}(\shX,\tau') \rightarrow \shX_{\sigma}$, also
idealized log smooth.

After a base-change, we also obtain
\[
\overline{\foM}^{\ev}(\shX,\tau'):=\overline{\foM}(\shX,\tau')
\times_{\shX_{\sigma}} X_{\sigma} \rightarrow X_{\sigma}
\]
idealized
log smooth, with the underlying stack of $\overline{\foM}^{\ev}(\shX,
\tau')$ isomorphic to the underlying stack of
\[
\foM^{\ev}(\shX,\tau')=\foM(\shX,\tau')\times_{\ul{\shX}_{\sigma}}
\ul{X}_{\sigma}.
\]
\medskip

{\bf Step 5.} \emph{Proof of the linear equivalence relation
\eqref{eq:equality of classes}}.
Continuing with $\sigma\in\P$ the minimal cone containing $\fop$,
denote by $\sigma_{\fou},\sigma_{\fou'}$ the cones containing 
$\fou$ and $\fou'$ respectively. We write $X_{\fou}, X_{\fou'}$ for
the corresponding zero-dimensional strata of $X$, and
write $P_{\sigma},P_{\fou},P_{\fou'}$ for the stalks
of the ghost sheaf of $X_{\sigma}$, $X_{\fou}$ and $X_{\fou'}$
at their generic points.

Our goal is to construct a rational function $\psi$ on
$\overline{\foM}^{\ev}(\shX,\tau')$ and determine its divisor of zeros and
poles. We note that as $\overline{\foM}^{\ev}(\shX,\tau')$ is idealized log
smooth over $\Spec\kk$, it is isomorphic, locally in the smooth topology, to a
subscheme of a toric variety defined by a monomial ideal, see
\cite[Prop.~B.2]{ACGSII}. Further, by the cited proposition and the explicit
description of the idealized structure of $\overline\foM(\shX,\tau')$, hence of
$\overline\foM^{\ev}(\shX,\tau')$, in fact $\overline\foM^{\ev}(\shX,\tau')$ is
smooth locally isomorphic to a toric stratum of a toric variety, and hence is
normal.

Let $\bar s\in \Gamma(X_{\sigma},\overline\shM^{\gp}_{X_{\sigma}})$ be defined
as follows. Note that giving such a section is equivalent to giving an integral
piecewise linear function on $\Star(\sigma) \subseteq B$. Using the affine
structure on $\Star(\sigma)$ induced by that on $B$, we take this function in
fact to be linear in this affine structure, primitive, vanishing on $\fop$ and
positive on $\fou$. We may then choose a lift of $\bar s$ to $s\in
\Gamma(X_{\sigma}, \shM^{\gp}_{X_{\sigma}})$. The existence of a lift is obvious
when $\dim X_{\sigma}=0$. When $\dim X_{\sigma}=1$, one needs to check that the
corresponding torsor is trivial. However, using the isomorphism of
Lemma~\ref{lem:toric dim one}, we may identify the line bundle associated to the
corresponding torsor with the restriction of a line bundle on
$X_{\Sigma_{\sigma}}$. The line bundle on $X_{\Sigma_{\sigma}}$ 
is defined by the linear function $\bar
s$ on $\Sigma_{\sigma}$. However,
it is a standard fact of toric geometry that the line bundle corresponding to a
linear (as opposed to piecewise linear) function is in fact trivial.

Let $U\subseteq \overline{\foM}^{\ev}(\shX,\tau')$ be the dense
open stratum, i.e., the open substack whose geometric points are
precisely those whose corresponding type is $\tau'$. 
Note that the image of $\ev^{\flat}(s)$ in 
$\overline{\shM}_{\overline{\foM}^{\ev}(\shX,\tau')}^{\gp}$ vanishes on
$U$. Indeed, this follows from the fact that
$h_{\tau'}(\tau'_{\out})\subseteq \fop$ and $h_{\tau'}|_{\tau'_{\out}}$ 
is dual to $\bar\ev^{\flat}:P_{\sigma}\rightarrow Q^{\circ,\sat}_{\tau'}$.
Thus, viewing $\O_{\overline{\foM}^{\ev}(\shX,\tau')}^{\times}$ as
a subsheaf of $\shM_{\overline{\foM}^{\ev}(\shX,\tau')}^{\gp}$ 
via the inverse of the structure map $\alpha$, we see that $\ev^{\flat}(s)$ 
restricts to an invertible function on $U$, and hence defines
a rational function $\psi$ on $\overline{\foM}^{\ev}(\shX,\tau')$.
We complete the proof
by determining the divisor of zeros and poles of $\psi$.

We now use the identification of underlying stacks of
$\overline{\foM}^{\ev}(\shX,\tau')$ and $\foM^{\ev}(\shX,\tau')$.
Since the complement of $U$ is a union of strata of 
$\foM^{\ev}(\shX,\tau')$, it is enough to
check the order of vanishing of $\psi$ along each codimension
one stratum of $\foM^{\ev}(\shX,\tau')$. These strata are given by
$\foM^{\ev}_{\tau}(\shX,\tau')$ with $\phi:\tau\rightarrow\tau'$
a contraction morphism with $\dim\tau=n-1$, see \cite[Rem.\ 3.28]{ACGSII}.
We write $\overline{\foM}^{\ev}_{\tau}(\shX,\tau')$ for
the corresponding stratum of $\overline{\foM}^{\ev}(\shX,\tau')$.

Given such a $\phi:\tau\rightarrow\tau'$, we have several cases:

\smallskip

{\bf Case 1.} $\dim h_{\tau}(\tau_{\out})=n-1$. In this case, 
necessarily $\bar s$ still vanishes on $h_{\tau}(\tau_{\out})$, so as above,
$\psi$ extends as an invertible function across $\foM^{\ev}_{\tau}(\shX,
\tau')$. 

\smallskip

{\bf Case 2.} \emph{$h_{\tau}(\tau_{\out})$ intersects the interior of $\fou$.} 
We adopt the notation from \cite[App.\ B]{ACGSII} that if $P$ is a monoid, 
$K\subseteq P$ a monoid ideal, then $A_{P,K}:=\Spec\kk[P]/K$ and 
$\shA_{P,K}:=[A_{P,K}/\Spec \kk[P^{\gp}]]$.

Let $Q_{\tau}^{\circ,\sat}$ be the stalk of the ghost sheaf of
$\overline{\foM}^{\ev}(\shX,\tau')$ at a geometric generic point $\bar w$ of
$\overline{\foM}^{\ev}_{\tau}(\shX,\tau')$, and let $J \subseteq
Q_{\tau}^{\circ,\sat}$ be the monoid ideal arising from the idealized structure.
As $\ev$ is idealized log smooth, smooth locally in a neighbourhood of $\bar w$,
$\ev$ concides with the morphism $\shA_{Q^{\circ,\sat},J}\rightarrow
\shA_{P_{\fou},K}$, by \cite[Prop.\ B.4]{ACGSII}. Here $K=P_{\fou}\setminus
\chi_{\sigma\fou}^{-1}(0)$ where $\chi_{\sigma\fou}:
P_{\fou}\rightarrow P_{\sigma}$ is the generization map. In particular, the
stratum $X_{\sigma} \subseteq X$, smooth locally in a neighbourhood of the point
$X_{\fou}$, is isomorphic to $A_{P_{\fou},K}$.

Let $\chi_{\tau'\tau}:Q^{\circ,\sat}_{\tau}\rightarrow Q^{\circ,\sat}_{\tau'}$
be the generization map, necessarily a localization along a face $F$ of
$Q^{\circ,\sat}_{\tau}$. As $\dim\tau=\dim\tau'+1$ we see that $F\cong\NN$. Note
that $\bar s$ is non-negative on $\fou$. Since $h_{\tau}(\tau_{\out})$ contains
$\fou$, on which $\bar s$ is non-negative, and $h_{\tau'}(\tau'_{\out})$ is a
face of $h_{\tau}(\tau_{\out})$, on which $\bar s$ vanishes, it follows that
$\bar s$ is non-negative on $h_{\tau}(\tau_{\out})$. Dually, it follows that the
stalk of $\bar\ev^{\flat}(\bar s)$ at $\bar w$ lies in $Q^{\circ,\sat}_{\tau}
\subseteq Q^{\circ,\gp}_{\tau}$.
Further, the stalk necessarily lies in $\chi_{\tau'\tau}^{-1}(0)=F$. The order
of vanishing of $\psi$ along the stratum determined by $\tau$ is then, by
standard toric geometry, equal to the germ of $\bar\ev^{\flat}(\bar s)$ under
the identification of $F$ with $\NN$. Dually, since by construction $\bar s$ generates
$P_{\fou}^{\gp}\cap \Lambda_{\fop}^{\perp}$, this number is the same as
\[
d_{\tau}:=|\coker(\Lambda_{\tau_{\out}}/\Lambda_{\tau'_{\out}}\rightarrow
\Lambda_{\fou}/\Lambda_{\fop})|.
\]
This is the order of vanishing of $\psi$ along the stratum
$\foM_{\tau}(\shX,\tau')$.

Note an elementary diagram chase shows that
\[
k_{\tau} = d_{\tau} |\coker (\Lambda_{\tau'_{\out}}\rightarrow
\Lambda_{\fop})_{\tors}|.
\] 

{\bf Case 3.} \emph{$h(\tau_{\out})$ intersects the interior of $\fou'$.} 
The same analysis as in Case 2 applies once $\bar s$ is replaced
by $-\bar s$. Hence $\psi$ has a pole of order $d_{\tau}$ along
$\foM_{\tau}(\shX,\tau')$.

\medskip

Putting these three cases together, we see that
the divisor of zeros and poles of $\psi$ gives the relation in
the Chow group of $\foM^{\ev}(\shX,\tau')$:
\[
\sum_{\tau\rightarrow\tau'\atop \fou\subseteq h_{\tau}(\tau_{\out})}
d_{\tau}[\foM^{\ev}_{\tau}(\shX,\tau')]
=
\sum_{\tau\rightarrow\tau'\atop \fou'\subseteq h_{\tau}(\tau_{\out})}
d_{\tau}[\foM^{\ev}_{\tau}(\shX,\tau')].
\]
If we then multiply both sides by 
$|\coker (\Lambda_{\tau'_{\out}}\rightarrow
\Lambda_{\fop})_{\tors}|$, we get the desired
\eqref{eq:equality of classes}.
\end{proof}

\subsection{Consistency from theta functions}

\begin{theorem}
\label{thm:consistency}
$\scrS_{\can}$ is a consistent wall structure.
\end{theorem}

\begin{proof}
Consistency as defined in \cite[Def.\ 3.9]{Theta} involves checking
consistency in codimensions zero, one and two. We check each case in turn.

\medskip

{\bf Consistency in codimension zero}. Let $\foj$ be a codimension zero joint,
$\fop_1,\ldots,\fop_r$ the walls containing $\foj$, taken in cyclic order, with
$\fop_i$ contained in chambers $\fou_i$ and $\fou_{i+1}$, with indices taken
modulo $r$. With $\theta_{\fop_i}:R_{\fou_i}\rightarrow R_{\fou_{i+1}}$ as
defined in \eqref{eq:theta fop def}, consistency at the joint means \cite[Def.\
2.13]{Theta} that
\begin{equation}
\label{eq:codim zero consistent}
\theta:=\theta_{\fop_r}\circ\cdots\circ\theta_{\fop_1}=\id
\end{equation}
as an automorphism of $R_{\sigma}$, where $\sigma\in\P$ is the
cell containing $\foj$. Note that for $p \in \sigma_{\ZZ}
\setminus \{0\}$ and $x\in \Int(\sigma)\setminus |\scrS_{\can}|$, 
$\vartheta_p(x)=z^p \mod \fom_x$, where $\fom_x$ is as in
\eqref{eq:Ix def}. Indeed, the trivial broken line
with no bends contributes the term $z^p$, and any other broken line
$\beta$ must have $a_{\beta}\in \fom_x$. Since the ideal $\fom_x$ is
nilpotent in $R_{\sigma}$ and $z^p$ is invertible in $R_{\sigma}$,
it follows that
$\vartheta_p(x)$ is invertible in $R_{\sigma}$.
Working inductively modulo powers of $\fom_x$, one then sees
that any element of $R_{\sigma}$ may be written
as a $\kk[Q]/I$-linear combination of ratios 
$\vartheta_p(x)/\vartheta_{p'}(x)$ for 
$p,p'\in \sigma_{\ZZ}$. By Theorem~\ref{thm:patching},(2),
applied successively for the walls $\fop_1,\ldots,\fop_r$, we
obtain $\theta\big(\vartheta_p(x)\big)=\vartheta_p(x)$, and hence
$\theta$ is the identity on $R_{\sigma}$. 

Consistency in codimension zero follows, as this just means each
codimension zero joint is consistent.

\medskip

{\bf Consistency in codimension one.} Let $\foj$ be a codimension
one joint, contained in slabs $\fob_1,\fob_2\subseteq\rho\in
\P^{[n-1]}_{\inte}$, maximal cones
$\sigma,\sigma'$, and codimension zero walls $\fop_1,\ldots,\fop_r\subseteq
\sigma$ and $\fop'_1,\ldots,\fop'_s\subseteq \sigma'$. We order these
so that 
\[
\fob_1,\fop_1,\ldots,\fop_r,\fob_2,\fop'_1,\ldots,\fop'_s
\]
are cyclically ordered about $\foj$. Then consistency at the
joint $\foj$ is expressed in \cite[Def.~2.14]{Theta} as follows.
We set
\begin{eqnarray*}
\theta:= \theta_{\fop_r}\circ\theta_{\fop_{r-1}}
\circ\ldots\circ\theta_{\fop_1}:&& R_\sigma\to R_\sigma\\
\theta':= \theta_{\fop_1'}\circ\theta_{\fop_2'}\circ\ldots\circ
\theta_{\fop_s'}:&&
R_{\sigma'}\to R_{\sigma'}.
\end{eqnarray*}
We use the notation $\chi_{\fob_i,\sigma}, \chi_{\fob_i,\sigma'}$
for the localization maps of \eqref{eq:localization}. 
Then consistency is the statement that
\[
(\theta\times\theta')\big((\chi_{\fob_1,\sigma},\chi_{\fob_1,\sigma'})
(R_{\fob_1})\big)= (\chi_{\fob_2,\sigma},\chi_{\fob_2,\sigma'})
(R_{\fob_2}).
\]

To show this, let $x_i\in\sigma$ be a point in the chamber adjacent
to $\fob_i$, and similarly $x_i'\in\sigma'$.
If $p\in \rho_{\ZZ}\setminus \{0\}$, then 
$\vartheta_p(x_i)=z^p \mod \fom_{x_i}$, and similarly for $x_i'$, 
as in the codimension zero case.
Thus $\vartheta_p(\fob_i)$, which exists by
Theorem~\ref{thm:patching}, (3), satisfies 
$\chi_{\fob_i,\sigma}(\vartheta_p(\fob_i)-z^p)$ nilpotent,
and the same for $\chi_{\fob_i,\sigma'}$. By the 
discussion preceding \cite[Def.\ 2.14]{Theta},
the map $(\chi_{\fob_i,\sigma},\chi_{\fob_i,\sigma'}):
R_{\fob_i}\rightarrow R_{\sigma}\times R_{\sigma'}$ is injective.
Thus we conclude that $\vartheta_p(\fob_i)-z^p$ is nilpotent in
$R_{\fob_i}$, and hence 
$\vartheta_p(\fob_i)$ is invertible, as $z^p$ is invertible. Further,
if $p\in \sigma_{\ZZ}\setminus \{0\}$, we may write $p=p'+a\xi$ with
$a\ge 0$
$p'\in\Lambda_{\rho}$, and then $\vartheta_p(\fob_1)=z^{p'}Z_+^a 
\mod \fom$, and similarly if $p\in\sigma'_{\ZZ}\setminus \{0\}$,
$p=p'-a\xi$, we have $\vartheta_p(\fob_1)=z^{p'}Z_-^a\mod \fom$.
 From this we see that any element of $R_{\fob_1}$ may be
written as a $\kk[Q]/I$-linear combination of Laurent monomials in theta
functions of the form 
$(\prod \vartheta_{p_i}(\fob_1)^{a_i})/\vartheta_p(\fob_1)$ with 
$p\in \rho_{\ZZ}$, $p_i\in \sigma_{\ZZ}\cup\sigma'_{\ZZ}$.
Note here we only have control over the expressions 
$\vartheta_p(\fob_1)$ for $p\in \sigma_{\ZZ}\cup\sigma'_{\ZZ}$, so we need such
Laurent monomials in order to get all elements of $R_{\fob_1}$.

Now it follows from Theorem~\ref{thm:patching},(2) that for $p\in
B(\ZZ)\setminus\{0\}$ \begin{align*} &(\theta\times
\theta')\big((\chi_{\fob_1,\sigma},\chi_{\fob_1,\sigma'})
(\vartheta_p(\fob_1))\big)=(\theta\times\theta')\big(\vartheta_p(x_1),
\vartheta_p(x'_1)\big)\\ = {} & \big(\vartheta_p(x_2),\vartheta_p(x_2')\big)
=(\chi_{\fob_2,\sigma}\times\chi_{\fob_2,\sigma'})(\vartheta_p(\fob_2)).
\end{align*} Combined with the generation statement of the previous paragraph,
we obtain the desired equality.

This shows consistency at codimension one joints, noting there
is no condition for codimension one joints contained in $\partial B$.
Thus we obtain consistency in codimension one.
\medskip

{\bf Consistency in codimension two.} Consistency at codimension
two joints is defined in \cite[\S3.2]{Theta}. If $\foj$ is
a boundary joint, i.e., contained in $\partial B$, then it is
easy to check that $\scrS_{\can}$ is convex at $\foj$
in the sense of \cite[Def.\ 3.12]{Theta}. Indeed, if
$\partial B\not=\emptyset$, it follows
from Propositions~\ref{prop:gtrop affine submersion} 
and~\ref{prop:relative case over S} that if $\tau$ is a wall type,
then $u_{\tau}$ is tangent to $\partial B$.
Thus by \cite[Thm.\ 3.13]{Theta},
the joint is consistent.

In case $\foj$ is an interior joint, we refer the reader to
\cite[\S3.2]{Theta} for the definition of a corresponding polyhedral 
affine manifold $(B_{\foj},\P_{\foj})$,
with $\P_{\foj}$ consisting of tangent cones to elements of $\P$
containing $\foj$ along $\foj$. Further, $\scrS_{\can}$ induces
a wall structure on $B_{\foj}$ we denote as $\scrS_{\foj}$.
For $p\in B_{\foj}(\ZZ)\setminus \{0\}$, broken lines on $B_{\foj}$
then define for general $x\in B_{\foj}$ theta functions 
$\vartheta^{\foj}_p(x)$ as before. Consistency is then the statements
(1)--(3) of Theorem~\ref{thm:patching} for $B_{\foj}$ and $\scrS_{\foj}$
rather than $B$ and $\scrS_{\can}$.

We do not prove consistency directly from Theorem~\ref{thm:patching}
as it is not clear what the relationship between theta functions
on $B$ and on $B_{\foj}$ is. Instead, we proceed as follows. 
Write $\Star(\foj)$ for the open star of $\foj$ with respect to
the polyhedral cone decomposition $\scrP_{\scrB}$. Then there is a natural
embedding of $\Star(\foj)$ in $B_{\foj}$, identifying a cone of
$\scrP_{\scrB}$ containing $\foj$ and contained in $\sigma\in\P$
with a cone in the tangent wedge of $\sigma$ along $\foj$, with
vertex at the origin of the tangent wedge.
Note the walls of
$\scrS_{\foj}$ are in one-to-one correspondence with the walls of
$\scrS_{\can}$ intersecting $\Star(\foj)$. 

The tangent space $\Lambda_{\foj,\RR}$ to 
$\foj$ acts by translation on $B_{\foj}$ and
$\Lambda_{\foj}$ acts on $B_{\foj}(\ZZ)$. Further, all walls of
$\scrS_{\foj}$ are invariant under this translation, and translation
takes broken lines to broken lines. In particular, 
if $p\in B_{\foj}(\ZZ)$, $v\in \Lambda_{\foj,\RR}$, it is immediate
that
\[
\vartheta^{\foj}_p(x)=\vartheta^{\foj}_p(x+v)
\]
for general $x$. Further, since the monomials $z^v$ for $v\in\Lambda_{\foj}$
are invariant under crossing all walls in $\scrS_{\foj}$, we in fact
have
\[
\vartheta_{p+v}^{\foj}(x)= z^v \vartheta_p^{\foj}(x).
\]
Note that the broken lines defining the two theta functions are
not precisely the same; rather, if $a_i z^{m_i}$ is a monomial
attached to a segment of a broken line contributing to
$\vartheta_p^{\foj}(x)$, then the corresponding broken line
contributing to $\vartheta_{p+v}^{\foj}(x)$ has attached monomial
$a_i z^{m_i+v}$, and thus the actual maps $\beta$ are different
as they have different derivatives. However, there is clearly a
one-to-one correspondence between such broken lines.
In particular, if items (1)--(3) of Theorem~\ref{thm:patching}
hold for a given choice of $p$, for the functions 
$\vartheta_p^{\foj}(x)$ rather than $\vartheta_p(x)$,
they will also hold for $p+v$ for any $v\in \Lambda_{\foj}$.

We will now show items (1)--(3) of Theorem~\ref{thm:patching} for the
functions $\vartheta^{\foj}_p(x)$ for arbitrary non-zero $p\in B_{\foj}(\ZZ)$,
which is fixed throughout the following discussion. Because we may
replace $p$ with $p+v$ for any $v\in\Lambda_{\foj}$, we may assume, as we do
from now on in this proof, that $p\in \Star(\foj)$. Indeed, for any point $x \in
B_{\foj}(\ZZ)$, there is some $v\in \foj\cap B(\ZZ)$ such that $x+v \in
\Star(\foj)$ under the canonical embedding of $\Star(\foj)$ in $B_{\foj}$. 

More generally,
it follows easily from this observation that if $Z\subseteq B_{\foj}$ 
is any compact subset,
then there is a $v\in \foj\cap B(\ZZ)$ such that $Z+v \subseteq \Star(\foj)$.

Now suppose given an open subset $U\subseteq B_{\foj}$ whose closure $Z$ is
compact. For each $x\in Z$ not contained in $|\scrS_{\foj}|$, we have a finite
set of broken lines contributing to $\vartheta^{\foj}_p(x)$, and these all have
asymptotic direction $p$ contained in $\Star(\foj)$. Thus if $\beta$ is a broken
line contributing to this theta function, $\beta([t_1,0])$ is a compact subset
of $B_{\foj}$ and we may translate $x$ by some $v\in \foj$ so that
$\beta([t_1,0])\subseteq \Star(\foj)$. As the asymptotic direction $p$
necessarily lies in a cone of $\scrP_{\scrB}$ containing $\beta(t_1)$, we see
also that $\beta((-\infty,t_1])\subseteq \Star(\foj)$. By summing over the $v$'s
for the finite number of broken lines contributing to $\vartheta^{\foj}_p(x)$ we
may thus translate $x$ by some $v\in \foj$ to guarantee that all broken lines
contributing to $\vartheta^{\foj}_p(x)$ have image contained in $\Star(\foj)$.
Noting that broken lines can be deformed continuously inside some polyhedron
without changing their type (see \cite[Prop.~3.5]{Theta}), we may further in
fact assume that all broken lines contributing to $\vartheta^{\foj}_p(x)$ for
$x\in Z\setminus |\scrS_{\foj}|$ have image in $\Star(\foj)$.

 From the construction of $\scrS_{\foj}$,
it is now clear that provided $x$ has been translated
sufficiently, $\vartheta^{\foj}_p(x)$ is the sum of contributions
of those broken lines contributing to $\vartheta_p(x)$ with image
wholly contained in $\Star(\foj)$.

One now checks that the proof of Theorem~\ref{thm:patching} applies
equally well when $p\in\Star(\foj)$ if one restricts attention to 
endpoints $x\in \Star(\foj)$ translated ``sufficiently far'' by
an element $v\in \foj$. Indeed, there are two key points of
the argument in the proof of Theorem~\ref{thm:patching} where one needs
to control the behaviour of the broken lines being considered.

First, in Step 2, the proof of \eqref{eq:theta plus change} and 
\eqref{eq:theta plus change slab} works for any wall structure (not
necessarily a consistent one) and hence works for $\scrS_{\foj}$.

Second, in Steps 3 and following, we focus on a single degenerate broken
line type $\tau'$ and the result follows from the relation 
\eqref{eq:equality of classes}. However, if we fix a point $x\in
h_{\tau'}(\tau'_{\out})$ and a neighbourhood $U$ of $x$ with compact closure
$Z$,
we then may translate $Z$ by some $v\in \foj$ so that all broken lines
with endpoints in $Z\setminus|\scrS_{\foj}|$ lie in $\Star(\foj)$,
and thus in \eqref{eq:equality of classes}, we only need to consider
those types $\tau$ with contraction
morphisms $\tau\rightarrow \tau'$ whose corresponding broken lines
with endpoint in $Z$ lie entirely in $\Star(\foj)$.
\end{proof}

We end this section with a brief observation in the relative case
$g:X\rightarrow S$,
following on from Proposition~\ref{prop:relative B prime} 
and Remarks~\ref{rem:phi restrict},\ref{rem:cone}. Assume that the hypotheses of
Proposition~\ref{prop:relative B prime} hold, so that $(B',\scrP')$
is a polyhedral affine pseudomanifold. In this case, the canonical
wall structure $\scrS_{\can}$ induces a wall structure
$\scrS'_{\can}$ on $(B',\scrP')$. Indeed, first note that
$\Lambda_{B'}$, the local system of integral tangent vectors on 
$B'$, is a subsheaf of $\Lambda_B|_{B'}$ consisting of those tangent
vectors $v$ with $(g_{\trop})_*(v)=0$. Using $\varphi$ the MPL function
on $B$ given by Construction~\ref{const:phi canonical}, we obtain
also an MPL function $\varphi|_{B'}$ as in Remark~\ref{rem:phi restrict}.
This induces a local system $\shP_{B'}$, and 
$\shP_{B'}\subseteq \shP_B|_{B'}$ is the subsheaf consisting of
those sections $m$ of $\shP_B|_{B'}$ such that $(g_{\trop})_*(\bar m)=0$.

Next, given a wall $(\fop,f_{\fop})\in\scrS_{\can}$, we define
the \emph{index} of $\fop$ to be
\[
\operatorname{ind}(\fop):=|\coker(g_{\trop,*}:\Lambda_{\fop}\rightarrow \ZZ)|.
\]
We may now define 
\[
\scrS_{\can}':=\{(\fop\cap B',f_{\fop}^{\operatorname{ind}(\fop)})\,|\,(\fop,
f_{\fop})\in\scrS_{\can}\},
\]
noting that for any $x\in B'\setminus\Delta$,
$f_{\fop} \in \kk[\shP^+_{B',x}]\subseteq \kk[\shP^+_{B,x}]$
by Proposition~\ref{prop:relative case over S} and the construction
of the canonical wall structure. 

\begin{proposition}
Assuming $g:X\rightarrow S$ satisfies the hypotheses of 
Proposition~\ref{prop:relative case over S}, then $\scrS_{\can}'$
is a consistent wall structure.
\end{proposition}

\begin{proof}
This follows from the definition of consistency and
the consistency of $\scrS_{\can}$. Indeed, checking consistency
of $\scrS_{\can}'$ in codimensions zero, one and two involves
checking certain properties in rings which are subrings of
the corresponding rings for $B$. For example, for consistency
in codimension zero, to verify \eqref{eq:codim zero consistent}
in the ring $R_{\sigma\cap B'}$ defined using the data
$(B',\scrP',\varphi|_{B'})$, it is enough to note that
$R_{\sigma\cap B'}$ is a subring of $R_{\sigma}$ and that
the ring automorphism $\vartheta_{\fop}$ of $R_{\sigma}$ restricts to
an automorphism of $R_{\sigma\cap B'}$ given by the wall
$(\fop\cap B',f_{\fop}^{\operatorname{ind}(\fop)})$. This is 
straightforward: see the discussion of \cite[\S4.2]{Theta} as
to why the power of $f_{\fop}$ is necessary. Indeed, 
\cite[\S4.2]{Theta} goes the opposite way, from a wall structure
for $B'$ to a wall structure for $\mathbf{C}B'=B$, see Remark~\ref{rem:cone}.
Thus there one must pass to a $\operatorname{ind}(\fop)^{th}$ root
rather than power. We leave the remaining
details to the reader.
\end{proof}


\section{Comparison with intrinsic mirror symmetry}
\label{sec:intrinsic mirror symmetry}

We continue with $g:X\rightarrow S$ in the absolute or relative cases satisfying
Assumptions \ref{ass:absolute} or \ref{ass:relative}. Having now constructed a
consistent wall structure $\scrS_{\can}$ on $B$, we observe that the data
$(B,\P)$ and $\scrS_{\can}$ are \emph{conical} in the sense of \cite[Def.\
3.20]{Theta}. Thus this data first specifies a scheme\footnote{This scheme is
written as $\foX^{\circ}$ in \cite{Theta}; we include the check here to indicate
it is the mirror to $(X,D)$.} $\check\foX^{\circ}$ flat over $\kk[Q]/I$, as
constructed in \cite[Prop.\ 2.4]{Theta} by gluing together copies of spectra of
rings $R_{\fou}$ and $R_{\fob}$, as well as some additional rings corresponding
to $n-1$ dimensional cells of $\scrP_{\scrS}$ contained in $\partial B$. We
refer the reader to the cited result for details. Most importantly for us,
\cite[Prop.\ 3.21]{Theta} tells us that $R:=\Gamma(\check\foX^{\circ},
\O_{\check\foX^{\circ}})$ is freely generated as a $\kk[Q]/I$-module by a basis
$\{\vartheta_p\,|\,p\in B(\ZZ)\}$. Further, $\check\foX:=\Spec R$ contains
$\check\foX^{\circ}$ as a dense open subscheme and is also flat over $\Spec
\kk[Q]/I$. The theta functions $\vartheta_p$ restrict on standard pieces $\Spec
R_{\fou}$ to the expressions $\vartheta_p(x)$ for $x\in \Int(\fou)$ defined in
Definition \ref{def:broken line theta function}.

In the relative case, $R$ is naturally a graded ring via
$\deg(\vartheta_p)=g_{\trop}(p)\in \NN$. Thus to obtain the mirror
in the relative case, one considers instead the flat family 
$\Proj R \rightarrow \Spec \kk[Q]/I$; see \cite[\S4]{Theta}
and \cite[\S1]{Assoc} for more details on this point of view.

\cite[Thm.\ 3.24]{Theta} gives a description of the structure constants
for $R$. Write, for $p_1,p_2\in B(\ZZ)\setminus\{0\}$,
\[
\vartheta_{p_1}\cdot \vartheta_{p_2}
=
\sum_{r\in B(\ZZ)} \alpha^{\trop}_{p_1p_2r}\cdot \vartheta_r
\]
with $\alpha^{\trop}_{p_1p_2r}\in \kk[Q]/I$. These structure constants,
written as $\alpha_r(p_1,p_2)$ in \cite{Theta}, 
are defined tropically as follows. Let $\fou$ be a
chamber of $\scrS_{\can}$ such that $r\in \fou$, and let $x\in \fou$
be general. Then
\begin{equation}
\label{def:alpha trop}
\alpha^{\trop}_{p_1p_2r}=\sum_{(\beta_1,\beta_2)} a_{\beta_1}a_{\beta_2}
\end{equation}
where the sum is over all pairs $(\beta_1,\beta_2)$ of broken lines
with asymptotic monomials $p_1, p_2$ respectively, with endpoint $x$,
and such that $m_{\beta_1}+m_{\beta_2}=r$, viewed as an equation in
$\Lambda_{\fou}$. We note that the restriction that $p_1,p_2\not=0$
is not important as $\vartheta_0$ was defined as the unit in the ring $R$.

On the other hand, \cite{Assoc} defines structure constants on
the free $\kk[Q]/I$-module with basis $\{\vartheta_p\,|\,p\in B(\ZZ)\}$
directly in terms of punctured Gromov-Witten invariants. Following
the notation of \cite{Assoc}, we have a potentially different
product rule
\[
\vartheta_{p_1}\cdot \vartheta_{p_2}
=
\sum_{r\in B(\ZZ)} \alpha^{\log}_{p_1p_2r}\cdot \vartheta_r.
\]
These structure constants are written simply as $\alpha_{p_1p_2r}$
in \cite{Assoc}. They are given by a formula
\begin{equation}
\label{eq:alpha log def}
\alpha^{\log}_{p_1p_2r}= \sum_{A} N^A_{p_1p_2r} t^A \in \kk[Q]/I,
\end{equation}
where the numbers $N^A_{p_1p_2r}$ are defined in 
\cite[Def.\ 3.21]{Assoc} as certain punctured invariants with contact orders $p_1,p_2$ and $-r$.

In this section, we show that in fact we obtain the same algebra,
i.e.,

\begin{theorem}
\label{thm:assoc comparison}
For all $p_1,p_2\in B(\ZZ)\setminus\{0\}$, $r\in B(\ZZ)$,
\[
\alpha^{\trop}_{p_1p_2r}=\alpha^{\log}_{p_1p_2r}.
\]
\end{theorem}

Before proving this theorem, let us introduce some additional terminology.

\begin{definition}
\label{def:product type}
A \emph{product type} with inputs $p_1,p_2\in B(\ZZ)$ and output $r\in B(\ZZ)$
is a type $\tau=(G,\bsigma,{\mathbf u})$ of tropical map to $\Sigma(X)$ 
such that:
\begin{enumerate}
\item
$G$ is a genus zero graph with $L(G)=\{L_1,L_2,L_{\out}\}$ with
$p_i\in\bsigma(L_i)$, $\mathbf{u}(L_i)=p_i$, $\bsigma(L_{\out})\in\P$, 
$r\in \bsigma(L_{\out})$ and 
${\bf u}(L_{\out})=-r$.
\item $\tau$ is realizable and balanced.
\item Let $h:\Gamma(G,\ell)\rightarrow \Sigma(X)$ be the corresponding universal
family of tropical maps, and let
$\tau_{v_{\out}}\in \Gamma(G,\ell)$ be the cone corresponding to the vertex
$v_{\out}$ of $G$ adjacent to $L_{\out}$. Then $\dim\tau= \dim
h(\tau_{v_{\out}})=n$. 
\end{enumerate}
A \emph{decorated product type} is a decorated type 
$\btau=(\tau,\mathbf{A})$ with $\tau$ a product type.
\end{definition}

Identically to Lemmas \ref{lem:wall virtual dim} and 
\ref{lem:virt dim broken lines}, we have

\begin{lemma}
Let $\btau$ be a decorated product type. Then $\scrM(X,\btau)$ is proper
over $\Spec\kk$ and carries a virtual fundamental class of dimension zero.
\end{lemma}

We now define
\[
N_{\btau}:= {\deg [\scrM(X,\btau)]^{\virt}\over |\Aut(\btau)|}
\]
as usual. We also have a map $h_*:\Lambda_{\tau_{v_{\out}}}
\rightarrow \Lambda_{\bsigma(v_{\out})}$, necessarily of finite index,
and define
\[
k_{\tau}:=|\coker h_*|=
|\Lambda_{\bsigma(v_{\out})}/h_*(\Lambda_{\tau_{v_{\out}}})|.
\]

We split the proof of Theorem \ref{thm:assoc comparison} into
several steps.

\begin{theorem}
\label{thm:alpha log description}
Fix $p_1,p_2,r\in B(\ZZ)$, and let $\fou$ be a chamber of
$\scrP_{\scrB}$ containing $r$. Then there exists a top-dimensional
subcone $\fou'\subseteq\fou$ containing $r$ such that
\[
\alpha^{\log}_{p_1p_2r} = \sum_{\btau=(\tau,{\bf A})} k_{\tau} N_{\btau} t^A
\in\kk[Q]/I,
\]
where the sum is over all isomorphism classes of
decorated product types with inputs $p_1,p_2$,
output $r$, and $\fou'\subseteq h(\tau_{v_{\out}})$. Here $A$ is the
total class of ${\bf A}$.
\end{theorem}

\begin{proof}
Throughout the proof, we assume familiarity with the notation of
\cite{Assoc}, and do not review it, giving only references where
appropriate. 

{\bf Step 1.} \emph{Alternative description of
$N^A_{p_1p_2r}$.} Fix a curve class $A\in Q\setminus I$, 
and let $\beta$ be the type of punctured map of curve class $A$,
with three punctured points $x_1,x_2,x_{\out}$ with contact orders given by
$p_1,p_2$ and $-r$ respectively. Thus we obtain moduli spaces of punctured
maps $\scrM(X,\beta)$ and $\foM^{\ev}(\shX,\beta)$, where
evaluation is at $x_{\out}$ only.

Let $\sigma_{\fou}\in\P$ be the maximal cell containing $\fou$. Let $P_{\fou}$
be the stalk of $\overline{\shM}_X$ at the zero-dimensional stratum
$X_{\sigma_{\fou}}$. Let $R:=\NN^n$ be generated by $e_1,\ldots,e_n$, with dual
monoid generated by $e_1^*,\ldots,e_n^*$. Choose a monoid homomorphism
$\psi:P_{\fou}\rightarrow R$ in such a way that the transpose
$\psi^t:R^\vee\rightarrow P_{\fou}^{\vee}=\sigma_{\fou}\cap \Lambda_{\fou}$ is
(1) injective; (2) has image contained in $\fou$. Further, if $r\not=0$, we also
require (3) $\psi^t(e_1^*)=r$.

If $r\not=0$, define $W$ as the log stack quotient of 
$\Spec (R\rightarrow \kk)$ by the $\GG_m$ action which acts
on the torsor associated to $m \in R$ with weight $\langle e_1^*,m\rangle$.
Put another way, $\ul{W}=B\GG_m$, $\overline\shM_W=R$, and the
torsor associated to $m\in R$ is $\shU^{\otimes \langle e_1^*,m\rangle}$
where $\shU$ is the universal torsor on $B\GG_m$. If, on the other hand,
$r=0$, we set 
\[
W:=\Spec (R\rightarrow\kk)\times B\GG_m^{\dagger},
\]
where $B\GG_m^{\dagger}$ denotes the log structure on $B\GG_m$
induced by the inclusion as a divisor $B\GG_m \subseteq [\AA^1/\GG_m]$.

We may then define a morphism $g:W\rightarrow \scrP(X,r)$, where the latter
stack is as defined in \cite[\S3.1]{Assoc}. Using the notation of \cite{Assoc},
write $Z_r$ for the stratum $X_{\sigma}$ where $\sigma\in\P$ is the minimal cone
containing $r$. Then $\scrP(X,r)$ may be described as $[Z_r/\GG_m]$, where
$\GG_m$ acts trivially on $\ul{Z}_r$ and acts on the log structure on $Z_r$ as
described in \cite[Rem.~3.1]{Assoc}. Note that since $r\in \sigma_{\fou}$,
$X_{\sigma_{\fou}}\subseteq Z_r$. There then exists a
morphism $g':\Spec(R\rightarrow\kk) \rightarrow X_{\sigma_{\fou}}$
which is an isomorphism on underlying schemes and such that
$\overline{(g')}^{\flat}=\psi$. We also denote by $g'$ the composition with the
inclusion into $Z_r$. If $r\not=0$, then the action of $\GG_m$ on
$\Spec(R\rightarrow\kk)$ used to define $W$ is compatible with the action of
$\GG_m$ on $Z_r$ using the fact that $\psi^t(e_1^*)=r$ and the description of
\cite[Rem.~3.1]{Assoc}, and hence $g'$ descends to the desired map
$g:W\rightarrow \scrP(X,r)$. If, on the other hand, $r=0$, $\scrP(X,r)=X\times
B\GG_m$, and hence we obtain the desired morphism $g:W\rightarrow \scrP(X,r)$ as
the product of $g'$ with the morphism $B\GG_m^{\dagger}\rightarrow B\GG_m$ which
is an isomorphism on underlying stacks.

We have canonical morphisms $\ev_{\shX}:\foM^{\ev}(\shX,\beta)
\rightarrow \P(X,r)$ and $\ev_X:\scrM(X,\beta)\rightarrow\P(X,r)$,
see \cite[Def.\ 3.5]{Assoc}. We may then define
\[
\foM^{\ev}(\shX,\beta)_W= W \times_{\scrP(X,r)}^{\fs}\foM^{\ev}(\shX,\beta),
\quad
\scrM(X,\beta)_W= W \times_{\scrP(X,r)}^{\fs}\scrM(X,\beta).
\]
By choosing the
image of $\psi^t$ to be sufficiently small, we may assume that
$g$ is transverse to $\ev_{\shX}$ in the sense of \cite[Def.\ 2.6]{Assoc},
see \cite[Thm.\ 2.9]{Assoc} for the necessary tropical criterion.
By \cite[Def.\ 7.9]{Assoc}, we then obtain a number
\[
N^{A,W}_{p_1p_2r}=\deg [\scrM(X,\beta)_W]^{\virt}
\]
with the virtual fundamental class defined via an obstruction theory
for 
\[
\varepsilon_W:\scrM(X,\beta)_W \rightarrow \foM^{\ev}(X,\beta)_W
\]
pulled back from that for $\scrM(X,\beta)\rightarrow \foM^{\ev}(X,\beta)$. By 
\cite[Lem.\ 7.8]{Assoc}, $\foM^{\ev}(X,\beta)_W$
is pure-dimensional of dimension zero, so that the virtual pull-back
of its fundamental class gives the virtual fundamental class
of $\scrM(X,\beta)_W$.

The hypotheses of \cite[Thm.\ 7.10]{Assoc} now hold, and we
obtain $N^{A,W}_{p_1p_2r} = N^A_{p_1p_2r}$. Thus by
\eqref{eq:alpha log def}, to prove the theorem it is sufficient to take
$\fou':=\psi^t(R^{\vee}_{\RR})$ and show that 
\[
N^{A,W}_{p_1p_2r} = \sum_{\btau} k_{\tau}N_{\btau}
\]
where the sum is over all isomorphism classes of decorated product types 
with total class $A$ as in the statement of the theorem.

We now make a slight change in the definition of $W$, $\foM^{\ev}(\shX,\beta)_W$
and $\scrM(X,\beta)_W$ in the case that $r=0$, to make the remainder
of the proof more uniform. Note that in this case,
\begin{align*}
\foM^{\ev}(\shX,\beta)_W={} & (\Spec(R\rightarrow\kk)\times B\GG_m^{\dagger})
\times^{\fs}_{X\times B\GG_m} \foM^{\ev}(\shX,\beta)\\
= {} & B\GG_m^{\dagger}\times_{B\GG_m} ((\Spec(R\rightarrow \kk)\times B\GG_m)
\times^{\fs}_{\scrP(X,r)} \foM^{\ev}(\shX,\beta)).
\end{align*}
The stack $\foM^{\ev}(\shX,\beta)_W$ then has the same underlying stack
as 
\[
(\Spec(R\rightarrow\kk)\times B\GG_m)\times_{\scrP(X,r)}^{\fs} 
\foM^{\ev}(\shX,\beta),
\]
but with the ghost sheaf on the former being a sum of the ghost sheaf of the
latter and the constant sheaf $\ul{\NN}$. As we are only concerned about virtual
fundamental classes, which are defined using the underlying morphism of stacks
$\ul{\scrM(X,\beta)}_W\rightarrow \ul{\foM^{\ev}(\shX,\beta)}_W$, we may thus
replace $W$ with $\Spec(R\rightarrow \kk)\times B\GG_m$, which also changes
$\foM^{\ev}(\shX,\beta)_W$ and $\scrM(X,\beta)_W$. Note that now the definition
of $W$ is uniform across the two cases, and whether or not $r=0$, the ghost
sheaf of $W$ is $R$.
\medskip

{\bf Step 2.} \emph{The structure of $\foM^{\ev}(X,\beta)_W$.}
We recall from \cite[Lem.\ 7.8]{Assoc} and its proof that the projection 
$\foM^{\ev}(\shX,\beta)_W\rightarrow W$ is integral, log smooth
and of fibre dimension and
log fibre dimension one. For the latter notion, see the review of
\cite[Def.\ A.5]{Assoc}. 

Let $\bar\eta$ be a generic point of an irreducible component 
$\foM_{\bar\eta}$ of 
$\foM^{\ev}(\shX,\beta)_W$. We view the former as an integral closed
substack of the latter, so that
\[
[\foM^{\ev}(\shX,\beta)_W]=\sum_{\bar\eta} \mu_{\bar\eta} [\foM_{\bar\eta}]
\]
in the Chow group of $\foM^{\ev}(\shX,\beta)_W$. Here the sum is over
all generic points $\bar\eta$, and $\mu_{\bar\eta}$ is the multiplicity
of $\foM_{\bar\eta}$ in $\foM^{\ev}(\shX,\beta)_W$.

Let $Q_{\bar\eta}$ be the stalk of the
ghost sheaf of $\foM^{\ev}(\shX,\beta)_W$ at $\bar\eta$. 
We then have
\begin{equation}
\label{eq:Q rank}
\rk Q^{\gp}_{\bar\eta}=\rk \overline\shM_{W}^{\gp}=n;
\end{equation}
the first equality comes from \cite[(A.1)]{Assoc}
and the above-mentioned equality of log fibre dimension and fibre dimension,
while the second equality is from the definition of $W$, as $R=\NN^n$.
Let $\tau=(G,\bsigma,{\bf u})$ be the type of the punctured map corresponding 
to $\bar\eta$.

Note that $\varepsilon_W^![\foM_{\bar\eta}]$ is a zero-dimensional cycle
in the Chow group of $\scrM(X,\beta)_W$. Assume that
$\deg \varepsilon_W^![\foM_{\bar\eta}]\not=0$.
We claim that in this case, $\tau$ is a product type.
Definition~\ref{def:product type},(1) is immediate by the definition of 
$\beta$. For Condition~(2) of the definition, note that
$\tau$ is necessarily realizable as it is the type of a punctured map
to $\shX$. On the other hand, the assumption of non-zero virtual degree
implies that for some decoration $\btau$ of
$\tau$ with total curve class $A$, $\scrM(X,\btau)$ is non-empty.
In particular, one may find a punctured map $C^{\circ}\rightarrow X$
defined over a geometric point whose type is $\tau'$ with a contraction
$\tau'\rightarrow\tau$. Hence by Lemma \ref{lem:balancing}, $\tau'$
is a balanced type, and hence $\tau$ is also balanced. 

For Condition (3) of the definition of product type, let
$Q_{\tau}$ be the basic monoid associated to the type $\tau$, so that
$Q^{\vee}_{\tau,\RR}$ coincides with the cone $\tau$. Note that in
general $Q_{\bar\eta}\not=Q_{\tau}$, as the fibre
product construction of $\foM^{\ev}(\shX,\beta)_W$ changes the
log structure. We wish to show that $\dim\tau=n$.

 From \cite[Prop.\ 5.2]{ACGSI}, we have
$Q_{\bar\eta,\RR}^{\vee}=\tau\times_{\sigma_{\fou}} R^{\vee}_{\RR}$
where the map $R^{\vee}_{\RR}\rightarrow \sigma_{\fou}$ is just the
inclusion
$\psi^t$ and the map $\tau\rightarrow\sigma_{\fou}$, induced
by the morphism $\ev_{\shX}$, is given
by evaluation at the vertex $v_{\out}$ adjacent to $L_{\out}$, i.e.,
is $h|_{\tau_{v_{\out}}}$. See \cite[Lem.\ 3.23]{Assoc} for this latter
statement.

Note the images of the two projections 
$Q_{\bar\eta,\RR}^{\vee}\rightarrow \tau, R^{\vee}_{\RR}$ intersect
the interiors of $\tau$ and $R^{\vee}_{\RR}$ respectively.
This is because these morphisms of cones are dual to local morphisms
of monoids, being induced by the two projections of log stacks
$\fM^{\ev}(\shX,\beta)_W\rightarrow \fM^{\ev}(\shX,\beta), W$.
 From this we may calculate the dimension of $Q^{\vee}_{\bar\eta,\RR}$
by computing its tangent space at a point in the interior
of $Q^{\vee}_{\bar\eta,\RR}$ mapping to the interiors of $\tau$ and
$R^{\vee}_{\RR}$ via the two projections.
This tangent space is $Q^*_{\bar\eta,\RR}=\tau^{\gp}\times_{\sigma_{\fou}^{\gp}}
R_{\RR}^*=\tau^{\gp}$ as $R^*_{\RR}\rightarrow \sigma^{\gp}_{\fou}$ is
an isomorphism of real vector spaces. Thus 
$\dim\tau=\dim Q^{\vee}_{\bar\eta,\RR}=n$, the latter equality
by \eqref{eq:Q rank}. Hence $\dim\tau=\dim h(\tau_{v_{\out}})=n$, as
desired.

\medskip

{\bf Step 3.} \emph{The multiplicity of $\foM_{\bar\eta}$ in
$\foM^{\ev}(\shX,\beta)_W$.} We use the log smoothness of the projection
$\foM^{\ev}(\shX,\eta)_W\rightarrow W$ already mentioned in Step 2. Note that
$W$ carries a natural idealized log structure given by the ideal
$R\setminus\{0\} \subseteq R$. This ideal pulls back to give an idealized
structure on $\foM^{\ev}(\shX,\beta)_W$, and an idealized strict
morphism\footnote{We recall a morphism of idealized log schemes
$f:(X,\shM_X,\shK_X)\rightarrow (Y,\shM_Y,\shK_Y)$ is idealized strict
if $\shK_X$ agrees with the monoid ideal generated by $f^{\flat}(\shK_Y)$.} 
is idealized log smooth if and only if it is log smooth, see \cite[IV,Variant
3.1.22]{Ogus}. Thus $\foM^{\ev}(\shX,\beta)_W\rightarrow W$ is idealized log
smooth, and hence, smooth locally at $\bar\eta$, is given as
$\shA_{Q_{\bar\eta},K}\rightarrow \shA_{R,R\setminus\{0\}}$ by \cite[Prop.\
B.4]{ACGSII}, where $K\subseteq Q_{\bar\eta}$ is the monoid ideal generated by
the image of $R\setminus \{0\}$ in $Q_{\bar\eta}$ under the map
$R\rightarrow Q_{\bar\eta}$. 

The multiplicity $\mu_{\bar\eta}$ of the component $\foM_{\bar\eta}$ may now be
calculated as the multiplicity of $\shA_{Q_{\bar\eta},K}$,
or equivalently, $\dim_{\kk}\kk[Q_{\bar\eta}]/K$. Note that
$Q^{\vee}_{\bar\eta,\RR}\rightarrow R^{\vee}_{\RR}$
is surjective, as follows from the integrality of
$\foM^{\ev}(\shX,\beta)_W\rightarrow W$ and \cite[Prop.~2.4]{Assoc}.
Further, 
$Q^{\vee}_{\bar\eta,\RR}$ and $R^{\vee}_{\RR}$
are of the same dimension. It thus follows that
$R_{\RR}\rightarrow Q_{\bar\eta,\RR}$ is an isomorphism of real cones.
Identifying $Q^{\gp}_{\bar\eta}$ with a super-lattice of 
$R^{\gp}=\ZZ^n$, the multiplicity is then the number of
points of $Q^{\gp}_{\bar\eta}$ contained in $[0,1)^n\subseteq R_{\RR}$.
This may be expressed as
\begin{equation}
\label{eq:multiplicity}
\mu_{\tau}:=\mu_{\bar\eta}=|\coker(R^{\gp}\rightarrow Q^{\gp}_{\bar\eta})|,
\end{equation}
noting this expression only depends on the type $\tau$.

\medskip

{\bf Step 4.} \emph{Comparison with the $N_{\btau}$.}
Consider the commutative diagram
\[
\xymatrix@C=30pt
{
\foM^{\ev}(\shX,\tau)_W\ar[r]^{j'_{\tau}} \ar[d]_k &
\foM^{\ev}_{\tau}(\shX,\beta)_W\ar[d]\ar[r]^{\iota_{\tau}'}&
\foM^{\ev}(\shX,\beta)_W\ar[d]\\
\foM^{\ev}(\shX,\tau)\ar[r]_{j_{\tau}} & 
\foM^{\ev}_{\tau}(\shX,\beta)\ar[r]_{\iota_{\tau}}&
\foM^{\ev}(\shX,\beta)
}
\]
where 
\[
\foM^{\ev}(\shX,\tau)_W=W\times^{\fs}_{\scrP(X,r)} \foM^{\ev}(\shX,
\tau),
\quad
\foM^{\ev}_{\tau}(\shX,\beta)_W=W\times^{\fs}_{\scrP(X,r)} 
\foM^{\ev}_{\tau}(\shX, \beta).
\]
In this diagram, all squares are Cartesian in both the log and fs log
categories, as all horizontal morphisms are strict. The morphism $\iota_{\tau}$,
hence also $\iota'_{\tau}$, is a closed embedding,
while
$j_{\tau}$, hence also $j'_{\tau}$, is a finite morphism of degree
$|\Aut(\tau)|$. Further, as $W\rightarrow \scrP(X,r)$ is a finite morphism, so
is $k$.

For a stack $M$, denote by $M_{\red}$ the reduction. Then
\begin{equation}
\label{eq:M tau to M eta}
\iota'_{\tau,*}[\foM^{\ev}_{\tau}(\shX,\beta)_{W,\red}]=
\sum_{\bar\eta}[\foM_{\bar\eta}]
\end{equation}
where the sum is over all generic points $\bar\eta$ of
$\foM^{\ev}(\shX,\beta)_W$ whose type is $\tau$. Indeed, any
such generic point $\bar\eta$ corresponds to a punctured map
$C^{\circ}/\bar\eta\rightarrow\shX$ in $\foM^{\ev}(\shX,\beta)$
of type $\tau$, hence lies in $\foM_{\tau}^{\ev}(\shX,\beta)$.
Combining \eqref{eq:M tau to M eta} with $\deg j'_{\tau}=|\Aut(\tau)|$ gives
\begin{equation}
\label{eq:another equality of classes}
(\iota_{\tau}'\circ j'_{\tau})_* [\foM^{\ev}(\shX,\tau)_{W,\red}]
=|\Aut(\tau)|\sum_{\bar\eta}[\foM_{\bar\eta}].
\end{equation}

Let 
$k_{\red}:\foM^{\ev}(\shX,\tau)_{W,\red}\rightarrow \foM^{\ev}(\shX,\tau)$ 
be the morphism induced by $k$. Bearing in mind that
$\foM^{\ev}(\shX,\tau)$ is reduced,
we may now calculate the
degree of the finite morphism $k_{\red}$ as follows. We take a 
(strict) geometric generic point 
$\Spec \kappa \rightarrow \foM^{\ev}(\shX,\tau)$.
The desired degree is the length of the reduction of the fibre
\begin{align}
\label{eq:fibre one}
\begin{split}
\Spec\kappa \times_{\foM^{\ev}(\shX,\tau)} \foM^{\ev}(\shX,\tau)_W = {} &
\Spec\kappa \times_{\foM^{\ev}(\shX,\tau)} (\foM^{\ev}(\shX,\tau)
\times^{\fs}_{\scrP(X,r)} W)\\
= {} & \Spec(Q_{\tau}\rightarrow\kappa)\times^{\fs}_{\scrP(X,r)} W.
\end{split}
\end{align}

We may replace $\scrP(X,r)$ with the strict closed substack $Z$
with underlying closed substack $\ul{X}_{\sigma_{\fou}}\times B\GG_m$
as the morphism $\Spec\kappa \rightarrow \scrP(X,r)$ factors through
this closed substack. Further, the morphism
$\Spec\kappa \rightarrow Z$ factors through $X_{\sigma_{\fou}}$
as all line bundles on $\Spec\kappa$ are trivial. Thus we may rewrite
\eqref{eq:fibre one} as
\begin{equation}
\label{eq:fibre two}
\Spec(Q_{\tau}\rightarrow\kappa)\times^{\fs}_{X_{\sigma_{\fou}}}
\Spec(R\rightarrow \kk)
=
(A_{Q_{\tau},Q_{\tau}\setminus \{0\}} \times^{\fs}_{A_{P_{\fou},
P_{\fou}\setminus \{0\}}} A_{R, R\setminus \{0\}})\times_{\Spec\kk}
\Spec\kappa,
\end{equation}
again following the notation of \cite[\S B]{ACGSII}.
By \cite[Prop.~A.4,(4)]{Assoc}, \eqref{eq:fibre two} may then be expressed as
\[
A_{Q_{\tau}\oplus^{\fs}_{P_{\fou}}R,J}\times_{\Spec\kk}\Spec\kappa
\]
where $J$ is the ideal generated by the image of
$Q_{\tau}\setminus \{0\}$ and $R\setminus \{0\}$.
We pass to the reduction
by replacing $J$ with its radical, which, by the explicit
description of the fs push-out of \cite[Prop.~A.4,(3)]{Assoc},
is the complement of the group of
invertible elements of $Q_{\tau}\oplus_{P_{\fou}}^{\fs} R$.
Further, this group of invertible elements
is the torsion part of the cokernel of
the map $(\varphi_1,-\varphi_2):
P_{\fou}^{\gp}\rightarrow Q_{\tau}^{\gp}\oplus R^{\gp}$,
where $\varphi_1:P_{\fou}^{\gp}\rightarrow Q_{\tau}^{\gp}$
and $\varphi_2:P_{\fou}^{\gp}\rightarrow R^{\gp}$ are induced
by $\foM^{\ev}(\shX,\tau)\rightarrow \scrP(X,r)$ and
$W\rightarrow \scrP(X,r)$ respectively.
Thus we see that the reduced fibre is
\[
\Spec
\kappa[\coker(P_{\fou}^{\gp}\rightarrow Q_{\tau}^{\gp}\oplus R^{\gp})_{\tors}],
\]
whose length of course is given by the order of this group of invertible
elements.

In conclusion, we see that
\begin{equation}
\label{eq:k pushforward}
k_{\red,*}[\foM^{\ev}(\shX,\tau)_{W,\red}] =
|\coker(P_{\fou}^{\gp}\rightarrow Q_{\tau}^{\gp}\oplus R^{\gp})_{\tors}|
[\foM^{\ev}(\shX,\tau)].
\end{equation}

Bearing in mind that
$Q_{\bar\eta}=(Q_{\tau}\oplus^{\fs}_{P_{\fou}}R)/\tors$, it
again follows from the explict description of the fs pushout in
\cite[Prop.~A.4,(3)]{Assoc} that
\[
Q_{\bar\eta}^{\gp}=\coker 
(P_{\fou}^{\gp}\rightarrow Q_{\tau}^{\gp}\oplus R^{\gp})/\tors,
\]
and then a simple diagram chase gives
\begin{align}
\label{eq:lattice index}
\begin{split}
|\coker(P_{\fou}^{\gp}\rightarrow Q_{\tau}^{\gp}\oplus R^{\gp})_{\tors}|
|\coker(R^{\gp}\rightarrow Q_{\bar\eta}^{\gp})|
= {} & |\coker(P_{\fou}^{\gp}\rightarrow Q_{\tau}^{\gp})|\\
= {} & |\coker(Q_{\tau}^*\rightarrow P_{\fou}^*)|
=k_{\tau}.
\end{split}
\end{align}

We next consider the diagram
\[
\xymatrix@C=30pt
{
\coprod_{\btau=(\tau,{\bf A})}\scrM(X,\btau)\ar[d]_{\varepsilon_{\tau}}&
\coprod_{\btau=(\tau,{\bf A})}\scrM(X,\btau)_W\ar[d]_{\varepsilon_{\tau,W}}
\ar[l]_{k'}\ar[r]&
\scrM(X,\beta)_W\ar[d]_{\varepsilon_W}\\
\foM^{\ev}(\shX,\tau)&\foM^{\ev}(\shX,\tau)_W\ar[l]^k\ar[r]_{\iota'_{\tau}\circ
j'_{\tau}}&\foM^{\ev}(\shX,\beta)_W
}
\]
where the disjoint unions are over all decorations of $\tau$ with total
degree $A$.
Now observe that
\begin{align*}
N^{A,W}_{p_1p_2r} = \deg [\scrM(X,\beta)_W]^{\virt}
= {} & \deg \varepsilon_W^![\foM^{\ev}(\shX,\beta)_W]\\
= {} & \deg \sum_{\bar\eta} \mu_{\bar\eta} \varepsilon_W^!
[\foM_{\bar\eta}].
\end{align*}
Here we are summing over all generic points $\bar\eta$ of 
$\foM^{\ev}(\shX,\beta)_W$, but of course we may restrict the
sum to those $\bar\eta$ such that 
$\deg\varepsilon^!_{W}[\foM_{\bar\eta}]\not=0$.
Thus by \eqref{eq:multiplicity}, \eqref{eq:another equality of classes} and Step 2, we may rewrite this equation as
\[
N^{A,W}_{p_1p_2r}= 
\deg\sum_{\tau}{\mu_{\tau} \over |\Aut(\tau)|}
\varepsilon_W^!(\iota'_{\tau}\circ j'_{\tau})_*[\foM^{\ev}(\cX,\tau)_{W,\red}]
\]
where, by Step 2, the sum is now over all product types 
$\tau$ with inputs $p_1,p_2$,
output $r$, and $\fou'\subseteq h_{\tau}(\tau_{v_{\out}})$. Now
using the push-pull formula for virtual pull-back of \cite[Thm.\ 4.1]{Man} for
the first and third equality, as well as \eqref{eq:multiplicity},
\eqref{eq:k pushforward} and \eqref{eq:lattice index} for the third, we obtain
\begin{align*}
N^{A,W}_{p_1p_2r}= {}  
&\deg\sum_{\tau}{\mu_{\tau} \over |\Aut(\tau)|}
\varepsilon_{\tau,W}^![\foM^{\ev}(\cX,\tau)_{W,\red}]\\
= {} &\deg\sum_{\tau}{\mu_{\tau} \over |\Aut(\tau)|}
k'_*\varepsilon_{\tau,W}^![\foM^{\ev}(\cX,\tau)_{W,\red}]\\
= {} &\deg\sum_{\tau}
{k_{\tau} \over |\Aut(\tau)|}
\varepsilon_{\tau}^![\foM^{\ev}(\cX,\tau)]\\
= {} & \sum_{\btau=(\tau,{\bf A})} k_{\tau}N_{\btau},
\end{align*}
as desired.
\end{proof} 

\begin{lemma}
\label{lem:splitting product type}
Let $\tau=(G,\bsigma,{\mathbf u})$ be a product type with
$N_{\btau}\not=0$ for some decoration $\btau$ of $\tau$, and let $G'$
be the spine of $G$ in the sense of Definition~\ref{def:spine}.
Then $v_{\out}$ is the unique trivalent vertex of $G'$.
Further, $v_{\out}$ is also a trivalent vertex of $G$, and if
$v_{\out}$ and $L_{\out}$ are removed from $G$, then $\tau$ splits into
two broken line types $\tau_1, \tau_2$.
\end{lemma}

\begin{proof}
Since $G'$ has three legs and no univalent vertices, $G'$ has a unique trivalent
vertex. Suppose that $v_{\out}$ is not this trivalent vertex. Necessarily
$\bsigma(v_{\out})$ is a maximal cell $\sigma\in\P^{\max}$ since $\dim
h(\tau_\out)=n$, and thus for any punctured map $f:C^{\circ}\rightarrow X$ in
the non-empty moduli space $\scrM(X,\btau)$, the union of irreducible components
of $C^{\circ}$ corresponding to $v_{\out}$ are mapped to the zero-dimensional
stratum $X_{\sigma}$. Hence by stability, $v_{\out}$ must be at least trivalent
in $G$, and thus there is an edge $E$ adjacent to $v_{\out}$ which is not
contained in $G'$. Let $\bar\tau=(\bar G,\bar{\bsigma},\bar{\bf u})$ be the type
obtained by cutting $G$ at the edge $E$ and taking the connected component of
the resulting graph not containing $L_1,L_2$ and $L_{\out}$. 
Thus $E$ becomes the unique leg of
$\bar G$. Let $h_{\bar\tau}$ be the universal tropical map of type $\bar\tau$.
If $\bar\tau_E\in \Gamma(\bar G,\bar\ell)$ is the cone corresponding to $E$,
then by Lemma~\ref{lem:key tropical lemma},(1), 
$\dim h_{\bar\tau}(\bar\tau_E)\le
n-1$. But since $\dim h(\tau_{v_{\out}})=n$, we obtain a contradiction. Thus
$v_{\out}$ is a trivalent vertex of $G'$.

A similar argument applies to show that $v_{\out}$ is also trivalent in
$G$: if the valency is higher than three, with an adjacent edge $E$
not an edge of $G'$, we may again cut at $E$ to obtain a contradiction
as above.

Thus by removing $v_{\out},L_{\out}$ from $G$, we obtain, for $i=1,2$, types
$\tau_i$ with legs $L_i, L_{i,\out}$, where $L_{1,\out}, L_{2,\out}$ arise from
the edges (or legs)
$E_1, E_2$ adjacent to $v_{\out}$
other than $L_{\out}$. We only need to show that $\tau_i$ is a broken line type. 

First consider Condition (1) of Definition~\ref{def:broken line type}.
The only thing to check here is that $\bsigma(L_{i,\out})\in\P$
and that ${\bf u}(L_{i,\out})\not=0$. For the first statement, note that
as $\bsigma(v_{\out})=\sigma$ and $\bsigma(v_{\out})
\subseteq \bsigma(E_i)=\bsigma(L_{i,\out})$, 
necessarily $\bsigma(L_{i,\out})=\sigma$
also. For the second statement, note that if ${\bf u}(L_{i,\out})=0$,
then we may vary $s\in \tau$ by changing only
the affine length of the edge $E_i$ of $G$, so that $h_s(v_{\out})$
remains unchanged. Since $\dim h(\tau_{v_{\out}})=n$, this would
imply that $\dim \tau > n$, a contradiction.

Condition (2) of Definition~\ref{def:broken line type} is immediate.
Finally, for Condition (3), it follows from 
$\dim h(\tau_{v_{\out}})=n$ that $\dim h_{\tau_i}(\tau_{i,\out})=n$,
and in particular $\dim\tau_i \ge n-1$. However, if $\dim \tau_i>n-1$
for $i=1$ or $2$, it then immediately follows by gluing together
tropical maps in the families $\tau_1,\tau_2$ that $\dim\tau>n$.
Thus we get $\dim\tau_i=n-1$, as desired.
\end{proof}

\begin{lemma}
\label{lem:product gluing}
Let $\btau$ be a decorated product type with $N_{\btau}\not=0$, and let
$\btau_1,\btau_2$ be the decorated broken line types obtained from splitting
$\btau$ at $v_{\out}$. Then
\[
k_{\tau}N_{\btau}=k_{\tau_1}k_{\tau_2}N_{\btau_1} N_{\btau_2}.
\]
\end{lemma}

\begin{proof}
The proof is by a straightforward application of Yixian Wu's gluing formula
Theorem~\ref{Thm: gluing theorem}. We first assume that neither $\tau_1$ nor
$\tau_2$ is a trivial broken line, that is, has no vertex. Then $v_\out$ has
two adjacent edges $E_1$,$E_2$. Splitting $\btau$ at these two edges leads to,
by Lemma~\ref{lem:splitting product type},
the decorated broken line types $\btau_1$, $\btau_2$, and a third decorated type
$\btau_0=(\tau_0,\bA_0)$, $\tau_0=(G_0,\bsigma_0,\bu_0)$ with only one
vertex $v_\out$ and three adjacent legs $(E_1,v_\out)$, $(E_2,v_\out)$ and
$L_\out$, the outgoing leg of $\tau$. Moreover, $\tau_0$ is a type entirely
contained in the maximal cell $\sigma$ containing $h(\tau_{v_\out})$, that is,
$\bsigma_0(v_{\out})=\sigma$. In particular, the curve
class $\bA_0(v_\out)$ is trivial. Thus any punctured map of
type $\btau_0$ is constant on underlying schemes and
has domain a $\PP^1$ with three punctured points.
It then follows as in \cite[Claim~3.22]{Assoc} that there is a 
unique basic punctured map $f:C^{\circ}/W\rightarrow X$ of type $\btau_0$
with $\ul{W}=\Spec\kk$. Since further $\scrM(X,\btau_0)$
is unobstructed over $\foM(\shX,\btau_0)$, $\scrM(X,\btau_0)$ is a reduced
point.

It is now clear by stability of the tropical situation under small
perturbations, and can be checked also explicitly, that the tropical gluing map
$\varepsilon_\omega^\gp$ in \eqref{Eqn: prod_E omega_E -> Lambda_E} is
surjective already for $\omega=\tau$. Thus the trivial displacement vector
$\nu=0$ in Definition~\ref{Def: Delta(nu)} is general for $\tau$.
Moreover, the dimension formula~\eqref{Eqn: dimension formula splitting} for
$\omega=\tau$ holds:
\[
\textstyle
\dim\widetilde\tau_0+\dim \widetilde\tau_1+\dim\widetilde\tau_2=
(n+2)+n+n= (n+2)+2n= \dim\widetilde\tau+\sum_E \rk\Lambda_E.
\]
Thus the set of transverse types $\Delta(\nu)$ consist only of the one element
$\tau$.

To compute the splitting multiplicity $m(\btau)$ (Definition~\ref{Def:
Delta(nu)},3) consider the following commutative diagram with exact rows and
columns.
\[
\xymatrix{
&0\ar[d]&0\ar[d]&0\ar[d]\\
0\ar[r]&\widetilde\tau^\gp_{\ZZ}\ar[r]\ar[d]&
(\widetilde\tau_0^\gp)_{\ZZ}\times(\widetilde\tau_1^\gp)_{\ZZ}\times
(\widetilde\tau_2^\gp)_{\ZZ}
\ar[r]^(.65){(\varepsilon_\tau^\gp)_{\ZZ}}\ar[d]&\im(\varepsilon_\tau^\gp)_{\ZZ}\ar[r]\ar[d]&0\\
0\ar[r]&\Lambda\times\ZZ^2\ar[r]\ar[d]&
(\Lambda\times\ZZ^2)\times\Lambda\times\Lambda\ar[d]\ar[r]&
\Lambda\times\Lambda\ar[d]\ar[r]&0\\
0\ar[r]&\ZZ/k_\tau\ZZ\ar[d]\ar[r]&\ZZ/k_{\tau_1}\ZZ\times\ZZ/k_{\tau_2}\ZZ\ar[r]\ar[d]&
\coker(\varepsilon_\tau^\gp)_{\ZZ}\ar[r]\ar[d]&0\\
&0&0&0
}
\]
Here $\Lambda=\sigma^\gp_\ZZ$, the first column is the exact sequence defining
$k_\tau$, along with a trivial $\ZZ^2$ factor in the first
two entries, the middle column the product of the exact sequences defining
$k_{\tau_1}$ and $k_{\tau_2}$, and the isomorphism 
$(\widetilde \tau_0^\gp)_{\ZZ}\to \Lambda\times\ZZ^2$. The third row now shows
\[
m(\btau)=|\coker((\varepsilon_\tau^\gp)_{\ZZ})|= \frac{k_{\tau_1} k_{\tau_2}}{k_\tau}.
\]
Applying Theorem~\ref{Thm: gluing theorem} we conclude
\[
\deg\big[\scrM(X,\btau)\big]^\virt = \frac{k_{\tau_1} k_{\tau_2}}{k_\tau}\cdot
\deg\big[\scrM(X,\btau_1)\big]^\virt\cdot \deg\big[\scrM(X,\btau_2)\big]^\virt.
\]
The claimed equality $k_{\tau}N_{\btau}=k_{\tau_1}k_{\tau_2}N_{\btau_1} N_{\btau_2}$ now follows by the definition of $N_{\btau}$, $N_{\btau_1}$, $N_{\btau_2}$ and noting that $\Aut(\btau)= \Aut(\btau_1)\times\Aut(\btau_2)$.

If one of $\tau_1$ or $\tau_2$ is a trivial broken line, the same analysis
omitting the trivial broken line factor applies to prove the result in this
case. If both $\tau_1,\tau_2$ are trivial broken lines the statement is
trivially true.
\end{proof}

\begin{proof}[Proof of Theorem~\ref{thm:assoc comparison}.]
Take $x\in \Int(\fou)$ in the definition of $\alpha^{\trop}_{p_1p_2r}$;
by consistency of $\scrS_{\can}$, \eqref{def:alpha trop} is 
independent of this choice of $x$. Further, in 
\eqref{def:alpha trop}, we may sum over decorated broken lines,
rather than broken lines, and get the same value. Thus by
Theorem~\ref{thm:main correspondence theorem}, we may instead 
write
\[
\alpha^{\trop}_{p_1p_2r} = \sum_{\btau_1=(\tau_1,{\bf A}_1),
\btau_2=(\tau_2,{\bf A}_2)} k_{\tau_1}k_{\tau_2}
N_{\btau_1}N_{\btau_2} t^{A_1+A_2}
\]
where now the sum is over admissible decorated broken line types
$\btau_1,\btau_2$ with 
$\fou\subseteq h_{\btau_i}(\tau_{i,\out})$,
${\bf u}_{\btau_i}(L_{i,\inc})=p_i$, and $u_{\tau_1}+u_{\tau_2}=r$.
(Note that by definition of $\scrS_{\scrB}$, if $x\in h_{\btau_i}(\tau_{i,\out})$
then $\fou \subseteq h_{\btau_i}(\tau_{i,\out})$.)

For such a pair $\btau_1,\btau_2$, we obtain a decorated product type by taking
a vertex $v_{\out}$ and gluing $G_1$ and $G_2$ to $v_{\out}$ via
the legs $L_{1,\out}$, $L_{2,\out}$, and add an additional leg
$L_{\out}$ adjacent to $v_{\out}$ with ${\bf u}(L_{\out})=-r$. 
The data $\bsigma, {\bf u},{\bf A}$ are then determined by
$\btau_1$, $\btau_2$, with ${\bf A}(v_{\out})=0$. 

We see $\btau$ is a decorated product type. Indeed, we only need
to check conditions (2) and (3) of Definition~\ref{def:product type}.
Balancing just needs to be checked at $v_{\out}$, which follows
from $r={\bf u}_1(L_{1,\out})+{\bf u}_2(L_{2,\out})=u_{\tau_1}+u_{\tau_2}$. 
For realizability and the dimension conditions (3), note that by
the corresponding dimension conditions for broken line types, for
each point $y\in \Int(\fou)$, there exists a unique $s_i\in\tau_i$ such that
$y \in h_{\tau_i,s_i}(L_{i,\out})$. Thus by gluing together 
$h_{\tau_1,s_1}$ and $h_{\tau_2,s_2}$, we obtain a tropical map realizing
$\tau$ which takes $v_{\out}$ to $y$. Furthermore, this map is necessarily
the unique such tropical map of type $\tau$, giving both realizability
and the dimension statements of (3).

Conversely, by Lemma~\ref{lem:splitting product type}, all product
types $\btau$ arise in this way. Thus we obtain using Lemma
\ref{lem:product gluing} that
\[
\alpha^{\trop}_{p_1p_2r} = \sum_{\btau=(\tau,{\bf A})} k_{\tau}N_{\btau}
t^A,
\]
which coincides with $\alpha^{\log}_{p_1p_2r}$ by 
Theorem~\ref{thm:alpha log description}.
\end{proof}


\begin{appendix}
\section{The gluing formula with toric gluing strata}
\label{App: gluing formula}

Let $B=\Spec(Q_B\arr \kk)$ be a logarithmic point, $X$ a Zariski log scheme and
$X\arr B$ a logarithmically smooth, integral, 
projective morphism. We assume $\Sigma(X)$
monodromy-free for simplicity. As before we write
$\cX$ for the Artin fan of $X$. Let $\btau=(\tau,\bA)$ with
$\tau=(G,\bsigma,\bar\bu)$ be a realizable decorated global type of a punctured
map of genus~$0$.\footnote{The restriction to genus~$0$ is for convenience,
since this is the case relevant to us and the general statement is slightly
more complicated.} Recall from \cite[Def.~3.8]{ACGSII} the stacks
$\scrM(X,\btau)$ and $\fM(\cX,\btau)$ of punctured stable maps to $X$ and to
$\cX$ marked by $\btau$.

For each vertex $v\in V(G)$ denote by $\btau_v= (\tau_v,\bA_v)$,
$\tau_v=(G_v,\bsigma_v,\bar\bu_v)$ the decorated global type with $V(G_v)=\{v\}$
and $E(G_v)=\emptyset$ obtained by splitting $G$ at all
edges.\footnote{Splitting at only a subset of edges works the same way, with a
little more bookkeeping notation necessary. For simplicity of presentation we
only discuss splitting at all edges.} Note that each $\tau_v$ is realizable as
well. Then $\fM(\cX,\tau)$ is locally pure-dimensional, and hence defines a
virtual fundamental class for $\scrM(X,\btau)$ by means of the obstruction theory
for $\scrM(X,\btau)$ over $\fM(\cX,\tau)$.

According to \cite[Cor.~5.13, Prop.~5.15]{ACGSII}, the splitting
morphism\medskip
\begin{equation}
\label{Eqn: splitting map}
\delta:\ul\scrM(X,\btau)\ \lra\ \textstyle\prod_{v\in V(G)}\ul\scrM(X,\btau_v)
\end{equation}
is finite and representable.\footnote{We suppress here the log structure which
is irrelevant for the formulation of the statement, although it is central for
its proof.} The gluing formula expresses $\delta_*[\scrM(X,\btau)]^\virt$ in
terms of the virtual fundamental classes of strata of $\scrM(X,\btau_v)$. A
stratum is given by a $\btau_v$-marked decorated global type $\bomega_v$ as the
image of the finite map
\[
j_{\bomega_v}: \scrM(X,\bomega_v)\arr \scrM(X,\btau_v)
\]
defined by changing the marking \cite[Rem.~3.28]{ACGSII}.

We now explain some more notation needed to state the gluing formula. For each
edge $E\in E(G)$ we obtain a gluing stratum $\sigma_E=\bsigma(E)\in\Sigma(X)$
defined by $\tau$. Note that if $v,v'\in V(G)$ are the vertices adjacent to $E$
and $L\in L(G_v)$, $L'\in L(G_{v'})$ are the legs in $G_v, G_{v'}$ obtained by
splitting $E$, then by the definition of $\tau_v$,
\[
\bsigma_v(L)=\bsigma_{v'}(L')= \sigma_E.
\]
We write $L=(E,v)$, $L'=(E,v')$ in the following, where the notation indicates
that $v\in V(G_v)\subseteq V(G)$ and $v'\in V(G_{v'})\subseteq V(G)$ are the
vertices adjacent to $E$.

Now the restriction of a $\tau_v$-marked basic tropical
punctured map, defined say over $\omega_v\in\Cones$, to the leg $(E,v)\in
L(G_v)$ defines a map of cones
\begin{equation}
\label{Eqn: omega_{E,v}->sigma_E}
\omega_{E,v}=\big\{(h,\lambda)\in \omega_v\times \RR_{\ge 0}\,\big|\, \lambda\le \ell(E,v)(h)\big\} \arr \sigma,
\end{equation}
for some $\sigma\in\Sigma(X)$ containing $\sigma_E=\bsigma_v(E,v)$ as a face.
Recall also that $\omega_v$ is the basic cone, or real cone with integral points
the dual of the basic monoid \cite[Def.~2.36]{ACGSII}, associated to $\omega_v$.
We conflate notations for types and associated basic cones and hence write
$\omega_v$ also for this type unless there is a danger of confusion. Similarly,
$\tau_v$ denotes both a global type and the basic cone
$(Q_{\tau_v})_\RR^\vee$ associated to the type. Note also that the marking by
$\tau_v$ induces a face embedding
\begin{equation}
\label{Eqn: tau_v->omega_v}
\tau_v=(Q_{\tau_v})_\RR^\vee \arr \omega_v.
\end{equation}
of basic cones.

Letting $(E,v)$ in \eqref{Eqn: omega_{E,v}->sigma_E} run over all the gluing
legs, we arrive at the following enlargement of the cone $\omega_v$ that
records a point on each puncturing leg:
\begin{equation}
\label{Eqn: tilde omega_v}
\widetilde\omega_v=\big\{(h,\lambda_{E,v})\in\omega_v\times\RR_{\ge0}^{L(G_v)}
\,\big|\, \lambda_{E,v}\le\ell(E,v)(h)\big\}.
\end{equation}
For each leg $(E,v)\in L(G_v)$, the projection forgetting all
$\lambda$-components but $\lambda_{E,v}$ defines a map of cones
\begin{equation}
\label{Eqn: omega_v -> omega_{E,v}}
\widetilde\omega_v\arr \omega_{E,v}.
\end{equation}
Similarly, for a $\tau$-marked tropical punctured map defined over
$\omega\in\Cones$, recording a point on each edge $E\in E(G)$ leads to
\begin{equation}
\label{Eqn: tilde omega}
\widetilde\omega=\big\{(h,\lambda_E)\in\omega\times\RR_{\ge0}^{E(G)}
\,\big|\, \lambda_E\le\ell(E)(h)\big\}.
\end{equation}

We now add the following assumption:

\begin{assumption}
\label{Ass: toric}
Assume all gluing strata $\ul Z_E$, $E\in E(G)$, are complete toric varieties,
with $\ocM_X|_{\ul Z_E}$ invariant under the torus action.
\end{assumption}

Under this assumption there is an embedding of the star of
$\sigma_E\in\Sigma(X)$ as a face-fitting complex of cones in the real vector
space associated to a lattice that we denote $\Lambda_E$. Moreover, this
embedding identifies the integral affine structures on each cone, and the
composition with the quotient by the image of the real subspace with lattice
$(\sigma_E)_\ZZ^\gp\cap \Lambda_E$ maps the face-fitting complex of cones to a
complete fan. If $Z_E$ is strictly embedded as a toric stratum of a toric
variety, $\Lambda_E$ is the lattice underlying the describing fan. We now view
each of the cones $\sigma\in\Sigma(X)$ containing $\sigma_E$ as embedded in
$(\Lambda_E)_\RR$.

Let $\omega=(G_\omega,\bsigma_\omega,\bu_\omega)$ be a not necessarily
realizable global type of punctured map marked by $\tau$. Splitting $\omega$ at
the edges not contracted by the marking leads to a collection $(\omega_v)_{v\in
V(G)}$ of types of tropical punctured maps, with $\omega_v$ marked by $\tau_v$
and with a leg $(E,v)\in L(G_{\omega_v})$ for each edge $E\in E(G)$ adjacent to
$v$. For each $E\in E(G)$ with $v,v'$ the adjacent vertices, the composition of
the embeddings $\bsigma_\omega(E,v),\bsigma_\omega(E,v')\to (\Lambda_E)_\RR$
with the quotient by the diagonal embedding $(\Lambda_E)_\RR\arr
(\Lambda_E)_\RR\times (\Lambda_E)_\RR$ defines a map
\[
\varepsilon_E:\bsigma_\omega(E,v)\times \bsigma_\omega(E,v')\arr
(\Lambda_E)_\RR\times (\Lambda_E)_\RR \arr (\Lambda_E)_\RR.
\]
Explicitly, the arrow on the right can be taken as the difference map
$(a,b)\mapsto a-b$.

Taking the product over all $E\in E(G)$ of the composition of $\varepsilon_E$
with the maps \eqref{Eqn: omega_v -> omega_{E,v}} and \eqref{Eqn:
omega_{E,v}->sigma_E} yields the map
\begin{equation}
\label{Eqn: prod_E omega_E -> Lambda_E}
\textstyle
\varepsilon_\omega:
\prod_{v\in V(G)} \widetilde\omega_v \arr \prod_{E,v} \omega_{E,v}
\arr \prod_E(\Lambda_E)_\RR\times(\Lambda_E)_\RR
\arr \prod_E (\Lambda_E)_\RR.
\end{equation}
Thus $\varepsilon_\omega$ measures the failure of a collection of tropical
punctured maps of types $\omega_v$, together with points on the legs arising
from splitting, to patch to a tropical punctured map of type $\omega$ together
with a choice of point on each of the splitted edges. We refer to an
element of $\prod_E (\Lambda_E)_\RR$ as a \emph{displacement vector}.

\begin{definition}
\label{Def: Delta(nu)}
1)\ We call a displacement vector $\nu=(\nu_E)_{E\in E(G)}\in\prod_E
(\Lambda_E)_\RR$ \emph{general} for $\tau$ if for each global
type $\omega$ marked by $\tau$, either $\varepsilon_\omega^\gp$ from \eqref{Eqn:
prod_E omega_E -> Lambda_E} is surjective or $\nu\not\in
\im(\varepsilon_\omega^\gp)$.\\[1ex]
2)\ Let $\nu=(\nu_E)_E\in \prod_E \Lambda_E$ be general for $\tau$. Define
the \emph{set $\Delta(\nu)$ of transverse types for $\nu$} as the set of
isomorphism classes of global types $\omega$ marked by $\tau$ with
$\nu\in\im(\varepsilon_\omega)$, and such that
\begin{equation}
\label{Eqn: dimension formula splitting}
\textstyle
\sum_v\dim\widetilde\omega_v=\dim\widetilde \tau+
\sum_E \rk \Lambda_E.
\end{equation}
The set of decorated global types $\bomega=(\omega,\bA)$ marked by $\btau$ and
with $\omega\in\Delta(\nu)$ is denoted by $\hat\Delta(\nu)$. We confuse a
transverse type $\omega$ with its splitting $(\omega_v)_{v\in V(G)}$, and
similarly in the decorated case.\\[1ex]
3) For a general displacement vector $\nu$ and $\bomega=(\omega,\bA)\in\hat\Delta(\nu)$
define the \emph{splitting multiplicity} by
\[
m(\bomega)= m(\omega):=\big[\textstyle\prod_E \Lambda_E:
\im (\varepsilon_\omega^\gp)_\ZZ\big].
\]
\end{definition}

\begin{remark}
\label{Rem: interpretation of Delta(nu)}
The set $\Delta(\nu)$ has an interpretation in terms of types of ``broken''
tropical punctured maps, in the sense that the matching condition along an edge
$E$ is replaced by matching translated by $\nu_E$. Specifically, let $\omega$ be
a global type marked by $\tau$ and $(\omega_v)$ the collection of global types
marked by $\tau_v$ obtained by splitting. Then $\omega\in \Delta(\nu)$ iff 
$\omega$ satisfies \eqref{Eqn: dimension formula splitting} and
there
exist $h_v\in\omega_v$ and $\lambda_v,\lambda_{v'}\in\RR_{>0}$ with
\begin{equation}
\label{Eqn: perturbed matching}
V(h_v)+\lambda_v \cdot u_E= V'(h_{v'})-\lambda_{v'}\cdot u_E+ \nu_E,
\end{equation}
as an equation in $(\Lambda_E)_\RR$. For the signs we assume
$E$ oriented from $v$ to $v'$. Here $V:\omega_v\arr (\Lambda_E)_\RR$ is the
restriction of the map in \eqref{Eqn: omega_{E,v}->sigma_E} to the face
$\lambda=0$, that is, the image of the vertex $v$ under the tropical punctured
map $h_v$, and analogously for $V'$ and $v'$.

 From this description it is also obvious that $\Delta(\nu)$ does not change by
rescaling $\nu$ by a positive constant. Moreover, if $\widetilde \tau$ is
defined analogously to $\widetilde\omega$, with $\lambda_E$ recording the
position of a point on the edge $E$, then $\Delta(\nu)$ is also unchanged by
adding to $\nu$ an element of the image of the map $\widetilde\tau \arr
\prod_E\Lambda_E$ evaluating at all points on the splitting edges.

The dimension formula \eqref{Eqn: dimension formula splitting} is equivalent to
requiring $\omega=(\omega_v)_v$ to be a minimal type of punctured map with
$\varepsilon_\omega^\gp$ surjective. 
\end{remark}

We are now in position to state the gluing formula.

\begin{theorem}{\cite{Wu}}
\label{Thm: gluing theorem}
Let $\btau=(\tau,\bA)$, $\tau=(G,\bsigma,\bar\bu)$ be a realizable decorated
global type of punctured maps of genus~$0$ to $X$, and $\btau_v$ the associated
decorated global type defined at $v\in V(G)$ obtained from splitting all edges.
Assume that all gluing strata are toric in the sense of Assumption~\ref{Ass:
toric}. Let $\nu=(\nu_E)_{E\in E(G)}$ be a general displacement vector for
$\tau$, and $\hat\Delta(\nu)$ the set of decorated transverse types for $\nu$
(Definition~\ref{Def: Delta(nu)}). Denote by $\delta$ the splitting morphism
from \eqref{Eqn: splitting map}.

Then the following equality in the Chow group of $\prod_{v\in
V(G)}\ul\scrM(X,\btau_v)$ holds:
\[
\delta_*[\scrM(X,\btau)]^\virt= \sum_{\bomega=(\bomega_v)_v\in{\hat\Delta}(\nu)}
\frac{m(\bomega)}{|\Aut(\bomega/\btau)|}\cdot (j_{\bomega_v})_*
\big[\textstyle\prod_v\scrM(X,\bomega_v)\big]^\virt.
\] 
\end{theorem}

\end{appendix}


\end{document}